\documentclass[11pt]{amsart}

\usepackage[toc,page]{appendix}

\usepackage{etex}
\usepackage{amsmath}
\usepackage{amssymb}
\usepackage{amscd}
\usepackage{amsthm}
\usepackage{hyperref}
\usepackage{colortbl}

\usepackage{mathtools}
\usepackage{extarrows}
\usepackage{ifthen}

\usepackage[all,cmtip]{xy}

\usepackage{tikz}
\usepackage{pgfplots}
\usepackage{tikz-cd}

\newtheorem{thm}{Theorem}[section]
\newtheorem{lem}[thm]{Lemma}
\newtheorem{cor}[thm]{Corollary}
\newtheorem{conj}[thm]{Conjecture}
\newtheorem{prop}[thm]{Proposition}
\newtheorem{question}[thm]{Question}

\newtheorem{remark}[thm]{Remark}

\theoremstyle{remark}

\newtheorem{rem}[thm]{Remark}
\newtheorem{example}[thm]{Example}

\theoremstyle{definition}
\newtheorem{defn}[thm]{Definition}

\numberwithin{equation}{section}

\allowdisplaybreaks[4]

\begin{document}

\vfuzz0.5pc
\hfuzz0.5pc 

\newcommand{\comment}[1]{}

\ifthenelse{\equal{1}{1}}{
\ifthenelse{\equal{2}{3}}{
\newcommand{\blue}[1]{{\color{blue}#1}}
\newcommand{\green}[1]{{\color{green}#1}}
\newcommand{\red}[1]{{\color{red}#1}}
\newcommand{\cyan}[1]{{\color{cyan}#1}}
\newcommand{\magenta}[1]{{\color{magenta}#1}}
\newcommand{\yellow}[1]{{\color{yellow}#1}} 
}{
\newcommand{\blue}[1]{#1}
\newcommand{\green}[1]{#1}
\newcommand{\red}[1]{#1}
\newcommand{\cyan}[1]{#1}
\newcommand{\magenta}[1]{#1}
\newcommand{\yellow}[1]{#1} 
}
}{
\newcommand{\blue}[1]{}
\newcommand{\green}[1]{}
\newcommand{\red}[1]{}
\newcommand{\cyan}[1]{}
\newcommand{\magenta}[1]{}
\newcommand{\yellow}[1]{} 
}

\newcommand{\claimref}[1]{Claim \ref{#1}}
\newcommand{\thmref}[1]{Theorem \ref{#1}}
\newcommand{\propref}[1]{Proposition \ref{#1}}
\newcommand{\lemref}[1]{Lemma \ref{#1}}
\newcommand{\coref}[1]{Corollary \ref{#1}}
\newcommand{\remref}[1]{Remark \ref{#1}}
\newcommand{\conjref}[1]{Conjecture \ref{#1}}
\newcommand{\questionref}[1]{Question \ref{#1}}
\newcommand{\defnref}[1]{Definition \ref{#1}}
\newcommand{\secref}[1]{\S \ref{#1}}
\newcommand{\ssecref}[1]{\ref{#1}}
\newcommand{\sssecref}[1]{\ref{#1}}

\newcommand{\RED}{{\mathrm{red}}}
\newcommand{\tors}{{\mathrm{tors}}}
\newcommand{\eq}{\Leftrightarrow}

\newcommand{\mapright}[1]{\smash{\mathop{\longrightarrow}\limits^{#1}}}
\newcommand{\mapleft}[1]{\smash{\mathop{\longleftarrow}\limits^{#1}}}
\newcommand{\mapdown}[1]{\Big\downarrow\rlap{$\vcenter{\hbox{$\scriptstyle#1$}}$}}
\newcommand{\smapdown}[1]{\downarrow\rlap{$\vcenter{\hbox{$\scriptstyle#1$}}$}}

\newcommand{\A}{{\mathbb A}}
\newcommand{\I}{{\mathcal I}}
\newcommand{\J}{{\mathcal J}}
\newcommand{\CO}{{\mathcal O}}
\newcommand{\C}{{\mathcal C}}
\newcommand{\BC}{{\mathbb C}}
\newcommand{\BQ}{{\mathbb Q}}
\newcommand{\m}{{\mathcal M}}
\newcommand{\h}{{\mathcal H}}
\newcommand{\Z}{{\mathcal Z}}
\newcommand{\BZ}{{\mathbb Z}}
\newcommand{\W}{{\mathcal W}}
\newcommand{\Y}{{\mathcal Y}}
\newcommand{\T}{{\mathcal T}}
\newcommand{\BP}{{\mathbb P}}
\newcommand{\CP}{{\mathcal P}}
\newcommand{\G}{{\mathbb G}}
\newcommand{\BR}{{\mathbb R}}
\newcommand{\D}{{\mathcal D}}
\newcommand{\DD}{{\mathcal D}}
\newcommand{\LL}{{\mathcal L}}
\newcommand{\f}{{\mathcal F}}
\newcommand{\E}{{\mathcal E}}
\newcommand{\BN}{{\mathbb N}}
\newcommand{\N}{{\mathcal N}}
\newcommand{\K}{{\mathcal K}}
\newcommand{\R} {{\mathbb R}}
\newcommand{\PP}{{\mathbb P}}
\newcommand{\Pp}{{\mathbb P}}
\newcommand{\BF}{{\mathbb F}}
\newcommand{\QQ}{{\mathcal Q}}
\newcommand{\VV}{{\mathcal V}}
\newcommand{\closure}[1]{\overline{#1}}
\newcommand{\EQ}{\Leftrightarrow}
\newcommand{\imply}{\Rightarrow}
\newcommand{\isom}{\cong}
\newcommand{\embed}{\hookrightarrow}
\newcommand{\tensor}{\mathop{\otimes}}
\newcommand{\wt}[1]{{\widetilde{#1}}}
\newcommand{\ol}{\overline}
\newcommand{\ul}{\underline}

\newcommand{\bs}{{\backslash}}
\newcommand{\CS}{{\mathcal S}}
\newcommand{\CA}{{\mathcal A}}
\newcommand{\Q}{{\mathbb Q}}
\newcommand{\F}{{\mathcal F}}
\newcommand{\sing}{{\text{sing}}}
\newcommand{\U} {{\mathcal U}}
\newcommand{\B}{{\mathcal B}}
\newcommand{\X}{{\mathcal X}}
\newcommand{\g}{{\mathcal G}}

\newcommand{\ECS}[1]{E_{#1}(X)}
\newcommand{\CV}[2]{{\mathcal C}_{#1,#2}(X)}

\newcommand{\rank}{\mathop{\mathrm{rank}}\nolimits}
\newcommand{\codim}{\mathop{\mathrm{codim}}\nolimits}
\newcommand{\Ord}{\mathop{\mathrm{Ord}}\nolimits}
\newcommand{\Var}{\mathop{\mathrm{Var}}\nolimits}
\newcommand{\Ext}{\mathop{\mathrm{Ext}}\nolimits}
\newcommand{\EXT}{\mathop{{\mathcal E}\mathrm{xt}}\nolimits}
\newcommand{\Pic}{\mathop{\mathrm{Pic}}\nolimits}
\newcommand{\Spec}{\mathop{\mathrm{Spec}}\nolimits}
\newcommand{\Jac}{\mathop{\mathrm{Jac}}\nolimits}
\newcommand{\Div}{\mathop{\mathrm{Div}}\nolimits}
\newcommand{\sgn}{\mathop{\mathrm{sgn}}\nolimits}
\newcommand{\supp}{\mathop{\mathrm{supp}}\nolimits}
\newcommand{\Hom}{\mathop{\mathrm{Hom}}\nolimits}
\newcommand{\Sym}{\mathop{\mathrm{Sym}}\nolimits}
\newcommand{\nilrad}{\mathop{\mathrm{nilrad}}\nolimits}
\newcommand{\Ann}{\mathop{\mathrm{Ann}}\nolimits}
\newcommand{\Proj}{\mathop{\mathrm{Proj}}\nolimits}
\newcommand{\mult}{\mathop{\mathrm{mult}}\nolimits}
\newcommand{\Bs}{\mathop{\mathrm{Bs}}\nolimits}
\newcommand{\Span}{\mathop{\mathrm{Span}}\nolimits}
\newcommand{\IM}{\mathop{\mathrm{Im}}\nolimits}
\newcommand{\im}{\mathop{\mathrm{im}}\nolimits}
\newcommand{\Hol}{\mathop{\mathrm{Hol}}\nolimits}
\newcommand{\End}{\mathop{\mathrm{End}}\nolimits}
\newcommand{\CH}{\mathop{\mathrm{CH}}\nolimits}
\newcommand{\Exec}{\mathop{\mathrm{Exec}}\nolimits}
\newcommand{\SPAN}{\mathop{\mathrm{span}}\nolimits}
\newcommand{\birat}{\mathop{\mathrm{birat}}\nolimits}
\newcommand{\cl}{\mathop{\mathrm{cl}}\nolimits}
\newcommand{\rat}{\mathop{\mathrm{rat}}\nolimits}
\newcommand{\Bir}{\mathop{\mathrm{Bir}}\nolimits}
\newcommand{\Rat}{\mathop{\mathrm{Rat}}\nolimits}
\newcommand{\aut}{\mathop{\mathrm{aut}}\nolimits}
\newcommand{\Aut}{\mathop{\mathrm{Aut}}\nolimits}
\newcommand{\eff}{\mathop{\mathrm{eff}}\nolimits}
\newcommand{\nef}{\mathop{\mathrm{nef}}\nolimits}
\newcommand{\amp}{\mathop{\mathrm{amp}}\nolimits}
\newcommand{\DIV}{\mathop{\mathrm{Div}}\nolimits}
\newcommand{\Bl}{\mathop{\mathrm{Bl}}\nolimits}
\newcommand{\Cox}{\mathop{\mathrm{Cox}}\nolimits}
\newcommand{\NE}{\mathop{\mathrm{NE}}\nolimits}
\newcommand{\NM}{\mathop{\mathrm{NM}}\nolimits}
\newcommand{\Gal}{\mathop{\mathrm{Gal}}\nolimits}
\newcommand{\coker}{\mathop{\mathrm{coker}}\nolimits}
\newcommand{\ch}{\mathop{\mathrm{ch}}\nolimits}

\title[Rationally Inequivalent Points]{Rationally Inequivalent Points on Hypersurfaces in $\PP^n$}

\author{Xi Chen}
\address{632 Central Academic Building\\
University of Alberta\\
Edmonton, Alberta T6G 2G1, CANADA}
\email{xichen@math.ualberta.ca}

\author{James D. Lewis}
\address{632 Central Academic Building\\ University of Alberta\\ Edmonton, Alberta T6G 2G1, CANADA}
\email{lewisjd@ualberta.ca}

\author{Mao Sheng}
\address{School of Mathematical Sciences\\ University of Science and Technology of China\\ Hefei, 230026, CHINA}
\email{msheng@ustc.edu.cn}

\date{March 26, 2021}


\thanks{Research of the first two authors is partially supported by Discovery Grants from the
Natural Sciences and Engineering Research Council of Canada. Research of the third named author is
partially supported by National Natural Science Foundation of China (Grant No. 11721101), National Key Research and Development Project SQ2020YFA070080, Chinese Universities Scientific Fund (CUSF), and Anhui Initiative in Quantum Information Technologies (AHY150200)}
\begin{abstract}
We prove a conjecture of Voisin that any two distinct points on a very general hypersurface of degree $2n+2$ in $\PP^{n+1}$ are rationally inequivalent.
\end{abstract}

\maketitle
\tableofcontents

\section{Introduction}
In this paper, we work exclusively over $\BC$. Recall that Claire Voisin proved the following (\cite[Theorem 3.1]{V1} for hypersurfaces and \cite[Theorem 0.6]{V2} for complete intersections)

\begin{thm}[C. Voisin]\label{VOISINTHM000}
Let $X$ be a very general complete intersection in $\PP^{n+k}$ of type $(d_1,d_2,...,d_k)$.
\begin{itemize}
\item If $\sum (d_i-1) \ge 2n+2$, no two distinct points on $X$ are $\BQ$-rationally equivalent.
\item If $(n,k,d_1) = (2,1,6)$, there are at most countably many points on $X$ that are $\BQ$-rationally equivalent to a fixed point $p$ for all $p\in X$.
\end{itemize}
\end{thm}

The main purpose of the paper is to
generalize this result in two directions. First,
we will make a minor improvement by replacing rational equivalence by
Ro\u{\i}tman's $\Gamma$-equivalence \cite{R1}: fixing a smooth projective curve
$\Gamma$ and two points $0\ne \infty\in \Gamma$, for every algebraic cycle
$\xi\in {\mathcal Z}^k(X\times \Gamma)$ with $\supp(\xi)$ flat
over $\Gamma$, the fibers $\xi_0$ and $\xi_\infty$
of $\xi$ over $0$ and $\infty$
are $\Gamma$-equivalent, written as
$\xi_0 \sim_\Gamma \xi_\infty$.

We will prove

\begin{thm}\label{COMVTHMGAMMAEQUIV}
For a fixed smooth projective curve $\Gamma$ with two fixed points $0\ne \infty$,
no two distinct points on a very general complete intersection $X$ in $\PP^{n+k}$ of type $(d_1,d_2,...,d_k)$ are
$\Gamma$-equivalent over $\BQ$ if $\sum (d_i-1) \ge 2n+2$.
\end{thm}
Although the replacement of rational equivalence by $\Gamma$-equivalence is only a mild improvement of Voisin's result, it does lead to the following interesting consequence:

\begin{cor}\label{COMVCOR000}
Fixing a quasi-projective variety $T$, there is no nonconstant rational map from $T$ to a very general complete intersection $X$ in $\PP^{n+k}$ of type $(d_1,d_2,...,d_k)$ if $\sum (d_i-1) \ge 2n+2$ or
$k=1$ and $d_1 = 2n+2$.
\end{cor}
The second case in the corollary follows from our solution to Voisin's conjecture (see Theorem \ref{COMVTHMHYPERSURFACE} below).

Second, we will try to find the optimal bound for $d_i$ where the result holds. Our most optimistic expectation is

\begin{conj}\label{COMVCONJ000}
For a very general complete intersection $X\subset \PP^{n+k}$ of type $(d_1,d_2,...,d_k)$ and every point $p\in X$,
\begin{equation}\label{COMVE000}
\dim R_{X,p,\Gamma} \le 2n - \sum_{i=1}^k (d_i - 1)
\end{equation}
where $R_{X,p,\Gamma} = \{ q\ne p\in X: N(p-q)\sim_\Gamma 0
\text{ for some } N\in \BZ^+\}$ and $\Gamma$ is a fixed smooth projective curve with two fixed points $0\ne \infty$.
\end{conj}

Note that $R_{X,p,\Gamma}$ is a locally Noetherian scheme.

The case $\sum (d_i - 1) = n+1$ follows from Ro\u{\i}tman's generalization of Mumford's famous
theorem (\cite{Mu}, \cite{R1} and \cite{R2}). Of course, Voisin proved
\begin{equation}\label{COMVE004}
\dim R_{X,p,\PP^1} \le 2n+1 - \sum_{i=1}^k (d_i - 1)
\end{equation}
for $\sum (d_i - 1)\ge 2n+2$ or $(n,k,d_1) = (2,1,6)$.
Theorem \ref{COMVTHMGAMMAEQUIV} shows that \eqref{COMVE000} holds
for $\sum (d_i - 1)\ge 2n+2$.

If our conjecture holds, $R_{X,p,\Gamma} = \emptyset$
when $\sum (d_i - 1)\ge 2n+1$. So the ``boundary'' case is
$\sum (d_i - 1) = 2n+1$. For example, it is expected that $R_{X,p,\Gamma} = \emptyset$ for a very general
sextic surface $X\subset \PP^3$. Voisin's theorem shows that $\dim R_{X,p,\PP^1} = 0$ for such surfaces $X$.
This boundary case is quite challenging, even only for sextic surfaces.
We claim the following:

\begin{thm}\label{COMVTHMHYPERSURFACE}
No two distinct points are $\Gamma$-equivalent over $\BQ$ on a very general hypersurface $X\subset \PP^{n+1}$
of degree $2n+2$ for a fixed smooth projective curve $\Gamma$ with two fixed points $0\ne \infty$. That is, \eqref{COMVE000} holds for
$k=1$ and $d_1 = 2n+2$.
\end{thm}

Note that the bound $d \ge 2n+2$ is optimal for hypersurfaces of degree $d$ in $\PP^{n+1}$: For a general hypersurface
$X$ of degree $d\le 2n+1$ in $\PP^{n+1}$, there exist two lines $L_1$ and $L_2$ in $\PP^{n+1}$ such that each $L_i$
meets $X$ at a unique point $p_i$ with $p_1\ne p_2$.
\magenta{Then
$dp_1\sim_{\PP^1} L_1 . X \sim_{\PP^1} L_2.X \sim_{\PP^1} dp_2$ and hence $p_1$ and $p_2$ are rationally equivalent over $\BQ$. Please see
Appendix \ref{COMVSECSF} for details.}

Recently, Eric Riedl and David Yang proved Conjecture \ref{COMVCONJ000} for $k=1$ and $d_1 \le 2n+1$ \cite{RiedlYangAGTH}. So together with Theorem
\ref{COMVTHMHYPERSURFACE}, we know that the conjecture holds for hypersurfaces.

A logarithmic version of Theorem \ref{COMVTHMGAMMAEQUIV}, with rational equivalence replaced by ${\mathbb A}^1$-equivalence, was given in \cite{CZ-logVoisin}.

We are only interested in the rational equivalence of two points. Voisin's method can be easily extended to the study of the linear dependence of $m$ points under rational equivalence on generic hypersurfaces. This was done in \cite{F}, where it was proved that for every $m\in \BZ^+$, there exists a number $d(m,n)$ such that $m$ distinct points on a very general hypersurface $X\subset \PP^{n+1}$ of $\deg X \ge d(m,n)$ are linearly independent in $\CH_0(X)\otimes \BQ$. However, the bound $d(m,n)$ obtained there is certainly not optimal. We are expecting that $d(m,n) = 2n+\lceil (m+1)/2\rceil$. However, we are not certain whether this is true for $m\ge 3$.

\section{Relative cycle map}\label{COMVSECRCM}

Voisin's proof consists of two major components. One is relative cycle map. For a relative Chow cycle
$Z\in \mathrm{CH}_\mathrm{hom}^n(X/B)$ for a smooth projective family $\pi: X\to B$ of relative dimension
$n$, if $\mathrm{AJ}_n(Z_b) = 0$ under the Abel-Jacobi map on each fiber $X_b$, one can define
some \red{topological} invariant $\delta Z\in H^n(R^n\pi_* \BQ)$.\comment{\red{When $R^k \pi_*\CO_X = 0$ for $k > 0$, $\delta Z$ is given by the formula \eqref{COMVE014} (see below).}
\blue{Can you prove this or quote the source of the proof? If not, maybe delete it?}}
This invariant can be defined in a Hodge-theoretical way as in \cite{V2}. Please see Appendix \ref{COMVSECNOAI} for a comprehensive treatment along this line.
Here we take a different approach: we define $\delta Z$ directly by \eqref{COMVE014} (see below)
and then we prove $\delta Z$ is invariant under
rational equivalence. This has the advantage of being elementary:
no Hodge theory is involved in the definition of $\delta Z$. In addition, we will obtain for free that $\delta Z$ is
invariant under $\Gamma$-equivalence.
Another advantage of this approach is that $\delta Z$ is well defined
for an arbitrary flat family $\pi: X\to B$ without any extra hypotheses on $X/B$.

\begin{defn}\label{COMVDEF000}
Let $\pi: X\to B$ be a flat and surjective morphism
of relative dimension $n$ from \red{a smooth variety} $X$ onto a smooth variety $B$
of $\dim B = N$. For a multi-section $Z\subset X$, we define
\begin{equation}\label{COMVE013}
\delta Z\in \Hom(\pi_* (\wedge^N \Omega_X),
\wedge^N \Omega_B)
= \Hom(\pi_* \Omega_X^N, K_B)
\end{equation}
as follows:
\begin{equation}\label{COMVE014}
\delta Z = \mathrm{Tr}_{Z/B} \circ (d\sigma):
\pi_* \Omega_X^N
\xrightarrow{d \sigma}
(\pi\circ \sigma)_* \Omega_{Z}^N
= (\pi\circ \sigma)_* K_{Z}
\xrightarrow{\mathrm{Tr}_{Z/B}}
K_B
\end{equation}
where $\mathrm{Tr}_{Z/B}$ is the trace map and
$\sigma: Z\hookrightarrow X$ is the embedding.

We can easily extend $\delta$ to the free abelian
group ${\mathcal Z}^n(X/B)$
of algebraic cycles $Z$ of pure codimension $n$
in $X$ whose support $\supp(Z)$ is flat over $B$.
For $Z = \sum m_i Z_i$ with $Z_i$ multi-sections of $\pi$,
we let $\delta Z = \sum m_i \delta Z_i$.
\end{defn}

\begin{rem}\label{COMVREMDELTAZ}
The definition \eqref{COMVE014} of $\delta Z$ might need some further explanation. The differential map $d\sigma$ is usually
$d\sigma: \sigma^* \Omega_X^N\to \Omega_Z^N$. In \eqref{COMVE014}, it is the composition of $d\sigma$ and $(\pi\circ \sigma)_*$:
\begin{equation}\label{COMVE500}
\begin{tikzcd}
\pi_* \Omega_X^N \arrow{r} & \pi_* (\Omega_X^N \otimes \CO_Z) \arrow{r}
\arrow[equal]{d}
& (\pi\circ\sigma)_* \Omega_Z^N\\
& \pi_* (\sigma_* \sigma^* \Omega_X^N).
\end{tikzcd}
\end{equation}
The trace map $\mathrm{Tr}_{Z/B}$ can be defined for $\pi_*(\wedge^m \Omega_Z) \to \wedge^m \Omega_B$
under a generically finite map $\pi: Z\to B$. Obviously, it is well defined outside of the ramification locus of $\pi$.
Since every meromorphic differential form in $\wedge^m \Omega_B$ is regular if it is regular in codimension $1$,
it suffices to show that the image of a differential $m$-form on $Z$ under the trace map can be extended to a
regular $m$-form on $B$ in codimension $1$ \cite[Proposition 5.77, p. 185]{K}. Moreover, the trace map is well defined for $B$ normal if
we follow the convention to define $\Omega_B$ to be the sheaf of differential forms regular in codimension $1$.
However, $\mathrm{Tr}_{Z/B}$ cannot be defined for $\pi_*(\Omega_Z^{\otimes m}) \to
\Omega_B^{\otimes m}$ when $m\ge 2$, which is the reason why Mumford's argument cannot be generalized
using pluri-canonical forms.
\end{rem}

\begin{lem}\label{COMVLEM001}
Let $\pi: X\to B$ be a flat and projective morphism
of relative dimension $n$ from $X$ onto a smooth variety $B$ of
$\dim B = N$ and let $Z$ be a cycle in ${\mathcal Z}^n(X/B)$.
If $\pi_* \Omega_X^N$ is locally free and
$Z_b \sim_{\Gamma} 0$ for all $b\in B$,
then $\delta Z = 0$, where $\Gamma$ is a fixed smooth projective
curve with two fixed points $0\ne \infty$.
\end{lem}

\begin{proof}
Since $\pi_* \Omega_X^N$ is locally free, $\delta Z = 0$
if and only if $\delta Z = 0$ at a general point of $B$.
\magenta{So we may shrink $B$ if necessary.}

Using a Hilbert scheme argument, we can
find a dominant and generically finite morphism $f: B'\to B$
and a cycle $Y\in {\mathcal Z}^n(X'\times \Gamma)$ such that
$\supp(Y)$ is flat over $B'\times\Gamma$ and $Y_0 - Y_\infty = f^* Z$,
where $X' = X\times_B B'$,
$Y_t$ is the fiber of $Y$ over $t\in \Gamma$ and
$f^* Z$ is the pullback of $Z$ under $f: X'\to X$ (we also use $f$ to denote the map $X'\to X$).
\magenta{Note that we can shrink $B$ to guarantee the flatness of $Y$ over $B'\times \Gamma$.}

Note that $\delta$ commutes with base change. Namely,
we have the commutative diagram
\begin{equation}\label{COMVE900}
\begin{tikzcd}
f^*\pi_* \Omega_X^N \arrow{r} \arrow[left]{d}{f^*(\delta Z)} & (\pi')_* \Omega_{X'}^N
\arrow{d}{\delta (f^*Z)}
\\
f^* K_B \arrow{r} & K_{B'}	
\end{tikzcd}
\end{equation}
where $\pi': X' \to B'$ is the projection.

Obviously, $\delta Z = 0$ if $\delta f^* Z = 0$ by \eqref{COMVE900}
and the fact that $\pi_* \Omega_{X}^N$ is locally free. So
we may simply replace $(X,B)$ by $(X',B')$.

Since $Y$ is flat over $B\times\Gamma$, it defines $\delta Y$ lying
\begin{equation}\label{COMVE067}
\begin{aligned}
\delta Y &\in \Hom(\varepsilon_* \Omega_{X\times \Gamma}^{N+1}, K_{B\times \Gamma})
\\
&= \Hom(\eta_1^* \pi_* \Omega_{X}^{N+1} \oplus \eta_1^* \pi_* \Omega_{X}^N \otimes
\eta_2^* K_\Gamma,
\eta_1^* K_{B}\otimes \eta_2^* K_\Gamma)\\
&= \Hom(\eta_1^* \pi_* \Omega_{X}^{N+1},
\eta_1^* K_{B}\otimes \eta_2^* K_\Gamma) \oplus
\Hom(\eta_1^* \pi_* \Omega_{X}^N, \eta_1^* K_{B})
\end{aligned}
\end{equation}
where $\varepsilon$, $\eta_1$ and $\eta_2$ are the projections
$\varepsilon: X\times \Gamma \to B\times \Gamma$,
$\eta_1: B\times \Gamma \to B$
and
$\eta_2: B\times \Gamma \to \Gamma$, respectively.
Let $\rho$ be the projection
\begin{equation}\label{COMVE008}
\begin{tikzcd}
\Hom(\varepsilon_* \Omega_{X\times \Gamma}^{N+1}, K_{B\times \Gamma})
\ar{r}{\rho} &
\Hom(\eta_1^* \pi_* \Omega_{X}^N,
\eta_1^* K_{B})
\ar[equal]{d}\\
 & H^0(\eta_1^* ((\pi_* \Omega_{X}^N)^\vee
\otimes K_{B}))
\end{tikzcd}
\end{equation}
given in \eqref{COMVE067}.

For every coherent sheaf $V$ on $B$, we have
$$
H^0(\eta_1^* V) =  H^0((\eta_1)_* \eta_1^*V) = H^0(V)
$$
since $\Gamma$ is projective. In other words, for every $s\in H^0(\eta_1^* V)$,
its restriction $s_t$ to $t\in \Gamma$ is a constant section in $H^0(V)$. Therefore, $\rho(\delta Y)_t$ is constant.

For every $t\in \Gamma$, we clearly have $\delta Y_t = \rho(\delta Y)_t$.
Therefore, $\delta Y_t$ is constant.
It follows that
$\delta Z = \delta Y_0 - \delta Y_\infty = 0$. We are done.
\end{proof}

We say that a closed subscheme $Z\subset X$ or its ideal sheaf $I_Z\subset \CO_X$
imposes independent conditions on a coherent sheaf $\f$ or its global sections $H^0(\f)$
(resp. a linear series $\D \subset H^0(\f)$) on $X$ if
$H^0(\f) \to H^0(\f\otimes \CO_Z)$ (resp. $\D \to H^0(\f\otimes \CO_Z)$) is surjective.

A coherent sheaf $\f$ on a variety $X$ is {\em globally generated}
(resp. {\em very ample}) if every $0$-dimensional
subscheme $Z\subset X$ of length $h^0(\CO_Z) = 1$ (resp. $2$) imposes independent conditions on $\f$. We also say that $\f$ is {\em weakly very ample} if this holds for all $Z = \{p,q\}$ consisting of two distinct points
$p\ne q$.

To show that $\sigma_1(b)\not\sim_\Gamma\sigma_2(b)$ over $\BQ$ at a general point $b\in B$ for two sections $\sigma_i:B\hookrightarrow X$ of $X/B$,
we only need to find $s\in H^0(U, \pi_* \Omega_X^N)$ satisfying
\begin{equation}\label{COMVE421}
(d\sigma_1) \sigma_1^* s - (d\sigma_2) \sigma_2^* s \ne 0
\end{equation}
over some open dense subset $U\subset B$. The existence of such $s$ is guaranteed if $\sigma_i(b)$ impose independent conditions
on $H^0(X_b, \Omega_X^N)$ for $b\in B$ general. This observation leads to the following:

\begin{prop}\label{COMVPROPVA}
Let $\pi: X\to B$ be a smooth and projective morphism
from \red{a smooth} $X$ onto a smooth variety $B$ of
$\dim B = N$. Suppose that $\Omega_X^N \otimes \CO_{X_b}$
is weakly very ample on $X_b$ for $b\in B$ general. Then $R_{X_b,p,\Gamma} = \emptyset$ for $b\in B$
very general and all $p\in X_b$, where $\Gamma$ is a fixed smooth projective \red{curve} with two fixed points $0\ne \infty$.
More generally,
\begin{equation}\label{COMVE005}
\begin{aligned}
R_{X_b,p,\Gamma} &\subset
\Big\{
q\in X_b: q\ne p \text{ and } \{p,q\} \text{ does not impose independent}\\
&\hspace{114pt} \text{conditions on } H^0(X_b, \Omega_X^N)
\Big\}
\end{aligned}
\end{equation}
for $b\in B$ very general.
\end{prop}

\begin{proof}
Suppose that there are a pair of points $p\ne q$ on a general fiber
$X_b$ such that $p \sim_\Gamma q$ over
$\BQ$ and $\{p,q\}$ imposes independent conditions on
$H^0(X_b, \Omega_X^N)$.
By a base change and shrinking $B$ to an affine variety, we may assume that
\begin{itemize}
\item
there exists two disjoint sections $P$ and $Q\subset X$ of $\pi: X\to B$ such that
$m(P_b - Q_b) \sim_\Gamma 0$ for some $m\in \BZ^+$ and all $b\in B$,
\item
$h^0(X_b, \Omega_X^N)$ is constant for all $b\in B$ and
\item
$P\sqcup Q$ imposes independent conditions on $H^0(\Omega_X^N)$.
\end{itemize}

Since $P\sqcup Q$ imposes independent conditions on $H^0(\Omega_X^N)$ and $\Omega_X^N$ is locally free, the map
\begin{equation}\label{COMVE001}
\begin{tikzcd}
H^0(\Omega_X^N\otimes I_P) \arrow[two heads]{r}{\sigma_Q^*} & H^0(\sigma_Q^* \Omega_X^N)
\end{tikzcd}
\end{equation}
is a surjection, where $\sigma_P$ and $\sigma_Q: B\hookrightarrow X$ are
the embeddings of $P$ and $Q$ to $X$, respectively.
Combining \eqref{COMVE001} with the pullback map of
$\sigma_Q: B\hookrightarrow X$ on differentials, we have a composition of two surjections
\begin{equation}\label{COMVE003}
\begin{tikzcd}
H^0(\Omega_X^N\otimes I_P) \arrow[two heads]{r}{\sigma_Q^*} & H^0(\sigma_Q^* \Omega_X^N)
\arrow[two heads]{r}{d\sigma_Q} & H^0(\Omega_B^N)
\end{tikzcd}
\end{equation}
where $d \sigma_P$ and $d \sigma_Q$ are the pullback maps induced by $\sigma_P$ and $\sigma_Q$ on the differentials, respectively.
Therefore, there exists
$s\in H^0(\Omega_X^N)$ such that
\begin{equation}\label{COMVE002}
\sigma_P^* s = 0 \text{ and } (d\sigma_Q) \sigma_Q^* s \ne 0.
\end{equation}
It follows that
\begin{equation}\label{COMVE015}
\langle \delta Z, s\rangle = (d\sigma_P) \sigma_P^* s - (d\sigma_Q) \sigma_Q^* s
= - (d\sigma_Q) \sigma_Q^* s \ne 0
\end{equation}
for $Z = P - Q$. On the other hand, $\delta Z = 0$ by
Lemma \ref{COMVLEM001}. Contradiction.

The above argument shows that no irreducible component of
\begin{equation}\label{COMVE006}
\begin{aligned}
S_{X,\Gamma} &=
\Big\{
(b,p,q): b\in B \text{ and } p\ne q\in X_b \text{ satisfy that }
p \sim_{\Gamma} q \text{ over } \BQ\\
&\hspace{66pt} \text{and } \{p,q\} \text{ imposes independent conditions}
\\
&\hspace{120pt} \text{on } H^0(X_b, \Omega_X^N)
\Big\}
\end{aligned}
\end{equation}
dominates $B$ via the projection $\xi: S_{X,\Gamma}\to B$.
Note that $S_{X,\Gamma}$ is a locally Noetherian subscheme of $X\times_B X$. Therefore,
for $b\in B\backslash \xi(S_{X,\Gamma})$, \eqref{COMVE005} holds.
\end{proof}

\begin{rem}\label{COMVREMPROPVA}
Note that the right hand side (RHS) of
\eqref{COMVE005} is a subscheme that does not depend on the choice of the triple $(\Gamma, 0, \infty)$.
\end{rem}

\begin{proof}[Proof of Theorem \ref{COMVTHMGAMMAEQUIV}]
Let $X\subset B\times \PP^{n+k}$ be the universal family of
complete intersections in $\PP^{n+k}$ of type $(d_1,d_2,...,d_k)$. By \cite{V1},
$\Omega_X^N$ is very ample on $X_b$ for $b\in B$ general
when $\sum (d_i-1) \ge 2n+2$. Then it follows from Proposition
\ref{COMVPROPVA} that no two distinct points on $X_b$
are $\Gamma$-equivalent over $\BQ$
for $b\in B$ very general.
\end{proof}

The rest of the paper will be devoted to the case $k=1$ and $d_1 = 2n+2$.

\section{Positivity of the sheaf of holomorphic $N$-forms}

Throughout this section, unless otherwise stated, we let $\pi: X\to B$ be a smooth projective morphism of relative dimension $n$ from smooth $X$ onto a smooth variety $B$ of $\dim B = N$.

\subsection{Ampleness of $\Omega_X^N$}

By Proposition \ref{COMVPROPVA}, we can prove that no two distinct points are $\Gamma$-equivalent on a very general fiber $X_b$ of $X/B$ if $\Omega_X^N$ is weakly very ample when restricted to $X_b$. We observe that
\begin{equation}\label{COMVE823}
\Omega_X^N \otimes \pi^* K_B^{-1} \cong T_X^n \otimes K_{X/B}
\end{equation}
using the pairing $\Omega_X^N\otimes \Omega_X^n \to K_X$.

When $X\subset \PP^{n+k}\times B$
is a family of complete intersections in $\PP^{n+k}$ of type $(d_1,d_2,...,d_k)$,
we have
\begin{equation}\label{COMVE216}
\begin{aligned}
\Omega_X^N \otimes \pi^* K_B^{-1} &\cong T_X^n \otimes \CO_X\left(
\sum (d_i-1) - (n+1)
\right)\\
&\cong T_X^n(n) \otimes \CO_X\left(
\sum (d_i-1) - (2n+1)
\right)
\end{aligned}
\end{equation}
where $\CO_X(1)$ is the pullback of the hyperplane bundle on $\PP^{n+k}$
and
$$
T_X^n(n) = \wedge^n T_X\otimes \CO_X(n) = \wedge^n (T_X(1)).
$$

It is very easy to verify the following simple facts:
\begin{enumerate}
\item The quotients of globally generated/(weakly) very ample coherent sheaves are also globally generated/(weakly) very ample.
\item If $\f$ is globally generated and $\g$ is globally generated/(weakly) very ample, then $\f\otimes \g$ is also globally generated/(weakly) very ample.
\item Combining (1) and (2), we see that the wedge and symmetric products of a
globally generated/(weakly) very ample coherent sheaf are also globally generated/(weakly) very ample.
\end{enumerate}

When $X\subset \PP^{n+k}\times B$ is a versal family of complete intersections,
Voisin proved that $T_X(1) = T_X\otimes \CO_X(1)$ is globally generated
on $X_b$, generalizing Herbert Clemens' result for hypersurfaces \cite{C} (see also \cite{E1} and \cite{E2}). Thus, we can conclude that
$\Omega_X^N$ is very ample on $X_b$ when $\sum (d_i - 1) > 2n+1$,
using the simple facts on global generation and very ampleness as above.

When $\sum (d_i - 1) = 2n+1$, we naturally ask whether
Voisin's result can be improved to show that
$\Omega_X^N \otimes \pi^* K_B^{-1} = T_X^n(n)$ is weakly very ample on $X_b$.
Voisin actually had a stronger conjecture
that $T_X^2(1) = (\wedge^2 T_X) \otimes \CO_X(1)$
is globally generated on $X_b$
\cite[Question 2.1]{V2}.

Unfortunately, $T_X^2(1)$ fails to be globally generated \cite{Voisin98}. Even the weaker assertion that $T_X^n(n)$ is weakly very ample is very likely to fail. Thus, in order to tackle the case $\sum (d_i - 1) = 2n+1$,
we cannot rely on the very ampleness of $\Omega_X^N$. Instead, we need to develop some refined criteria to show that $\delta (\sigma_1(B) - \sigma_2(B)) \ne 0$ for two sections $\sigma_i$ of $X/B$. These criteria, while still depending on certain positivity of $\Omega_X^N$, do not require it to be weakly very ample.

\subsection{Differential map $d\sigma$}\label{COMVSSDMP}

A closer examination of the proof of Proposition \ref{COMVPROPVA} shows that we do not really
need $\Omega_X^N$ to be weakly very ample on $X_b$. We only need find $s\in H^0(U, \pi_* \Omega_X^N)$ satisfying
\eqref{COMVE421}.
This is much weaker than the requirement that $p_1 = \sigma_1(b)$ and $p_2 = \sigma_2(b)$ impose
independent conditions on $H^0(X_b, \Omega_X^N)$ for $b$ general. For one thing,
$(d\sigma_1) \sigma_1^* s - (d\sigma_2) \sigma_2^* s = 0$ imposes only one condition on $\Gamma_b(\Omega_X^N) = H^0(X_b, \Omega_X^N)$.

Let $\Gamma_b(d\sigma_i)$ be the map induced by $d\sigma_i$ on
$\Gamma_b(\Omega_X^N)$ as in
\begin{equation}\label{COMVE145}
\begin{tikzcd}[column sep=36pt]
\Gamma_b(T_X^n\otimes K_X)\arrow[equal]{d}\\
\Gamma_b(\Omega_X^N) \arrow{r}{d\sigma_1 \oplus d\sigma_2} &
\Gamma_b\left(K_{\sigma_1(B)}\right)\oplus \Gamma_b\left(K_{\sigma_2(B)}\right).
\end{tikzcd}
\end{equation}
Clearly, \eqref{COMVE421} holds for some $s\in H^0(U, \pi_* \Omega_X^N)$ if
\begin{equation}\label{COMVE482}
\ker(\Gamma_b(d\sigma_1)) \ne \ker(\Gamma_b(d\sigma_2))
\end{equation}
holds at a general point $b\in B$. More precisely, as long as \eqref{COMVE482}
holds at a point $b\in B$ such that $h^0(X_t, \Omega_X^N)$ is locally constant for $t$ in an open
neighborhood of $b$, we can find a section $s_b\in \Gamma_b(\Omega_X^N)$ with the property
\begin{equation}\label{COMVE007}
(d\sigma_1) s_b - (d\sigma_2) s_b \ne 0
\end{equation}
and this $s_b$ can be extended to a section $s\in H^0(U, \pi_* \Omega_X^N)$ over an open neighborhood $U$ of $b$ satisfying \eqref{COMVE421}.

Therefore, to show that $\sigma_i(b)$ are not $\Gamma$-equivalent over $\BQ$ on a general fiber $X_b$, we just have to prove \eqref{COMVE482}.
Let us formalize this observation in the following proposition:

\begin{prop}\label{COMVPROPUNI}
Let $X$ be a smooth projective family of varieties over a smooth variety $B$ of $\dim B = N$ and let $\sigma_i: B\to X$ be two disjoint sections of $X/B$ for $i=1,2$.
Then $\sigma_1(b)$ and $\sigma_2(b)$ are not $\Gamma$-equivalent over $\BQ$ on $X_b$ for $b\in B$ general if \eqref{COMVE482} holds
at a point $b$ where $h^0(X_t, \Omega_X^N)$ is locally constant in $t$.
\end{prop}

\subsection{Criterion for two fixed sections}

In the rest of this section, we assume that $X\subset B\times P$ is a smooth family of subvarieties in some smooth projective variety $P$. Usually, we take $P$ to be some projective space. But some results apply to arbitrary $P$.
Here we strive for the greatest generality, although we just need to apply these results to hypersurfaces in $\PP^n$.

To apply Proposition \ref{COMVPROPUNI}, we need an explicit description of the differential maps $d\sigma_i$.
They can be made very explicit if $X\subset Y = B\times P$ is a family of varieties in a projective space $P$
passing through two fixed points $p_i\in P$ and $\sigma_i(b)\equiv p_i$ for $i=1,2$. On the other hand, for an
arbitrary family $X\subset Y$ with two sections $\sigma_i$ over $B$, we can always apply an automorphism
$\lambda\in B\times \Aut(P)$, after a base change, fiberwise to $Y/B$ such that $\lambda\circ \sigma_i(b) \equiv p_i$
for two fixed points $p_i\in P$; thus, to test \eqref{COMVE482} for a general fiber $X_b$ of $X/B$, it suffices to test it
for a general fiber $\widehat{X}_b$ of $\widehat{X}/B$, $\widehat{X} = \lambda(X)$ and $\widehat{\sigma}_i =
\lambda\circ \sigma_i$. Let us first consider families $X\subset B\times P$ with two fixed sections $\sigma_i(b) \equiv p_i$.

To set it up, we let $P = \PP^{r}$ and fix two points $p_1\ne p_2$ in $P$. We let $X\subset Y = B\times P$ be a
closed subvariety of $Y$ that is flat over $B$ with fibers $X_b$
containing $p_1$ and $p_2$ for all $b\in B$.
We assume that $X$ and $B$ are smooth of $\dim X = N + n$ and
$\dim B = N$, respectively.
We have two sections $\sigma_i: B\to X$ sending $\sigma_i(b) = p_i$ for all $b\in B$ and $i=1,2$.

To state our next proposition on the differential map $d\sigma$,
we need to introduce the filtration $F^\bullet \Omega_X$ associated to the fibration $X/B$.

For a surjective morphism $f: W\to B$ with $B$ smooth, we have \red{a Leray} filtration
\begin{equation}\label{COMVE118}
\begin{aligned}
\Omega_W^m &= F^0 \Omega_W^m \supset F^1\Omega_W^m \supset ... \supset F^{m+1}\Omega_W^m = 0\\
&\text{ with } \text{Gr}_F^p \Omega_W^m = \frac{F^p \Omega_W^m}{F^{p+1} \Omega_W^m}
= f^*\left(\wedge^p \Omega_B\right) \otimes \wedge^{m-p} \Omega_{W/B}
\end{aligned}
\end{equation}
for $\Omega_W^m = \wedge^m \Omega_W$ derived from the short exact sequence
\begin{equation}\label{COMVE111}
\begin{tikzcd}
0 \arrow{r} & f^*\Omega_B\arrow{r} & \Omega_W\arrow{r} & \Omega_{W/B}\arrow{r} & 0.
\end{tikzcd}
\end{equation}
Note that $F^p$ is an exact functor.

For $\pi_B: Y\to B$ with $Y=B\times P$, $F^p \Omega_Y^m$ is simply
\begin{equation}\label{COMVE064}
F^p \Omega_Y^m = \bigoplus_{i\ge p} \pi_B^* \Omega_B^i \otimes \pi_P^* \Omega_P^{m-i}
\end{equation}
and we have natural projections $\Omega_Y^m\to F^p\Omega_Y^m$, where $\pi_B$ and $\pi_P$ are the projections of $Y$ to $B$ and $P$, respectively.

We have the so-called adjunction sequence
\begin{equation}\label{COMVE069}
\begin{tikzcd}
0 \arrow{r} & T_{X} \arrow{r} & T_{Y} \otimes \CO_X \arrow{r}{\eta} & \N_{X}
	\arrow{r} & 0
\end{tikzcd}
\end{equation}
associated to $X\subset Y$, where $\N_X$ is the normal bundle of $X$ in $Y$.
From this, we obtain a left exact sequence
\begin{equation}\label{COMVE412}
\begin{tikzcd}
0 \arrow{r} & T_X^m \arrow{r} & T_Y^m \arrow{r}{\eta_m} &
T_Y^{m-1}\otimes \N_{X}
\end{tikzcd}
\end{equation}
on $X_b$, where $\eta_m$ is actually the map in the generalized Koszul complex
\begin{equation}\label{COMVE415}
\begin{aligned}
\wedge^m T_Y \otimes \CO_X & \xrightarrow{\eta_m}
\wedge^{m-1} T_Y\otimes \N_X \xrightarrow{}
\wedge^{m-2} T_Y\otimes \Sym^2 \N_X\\
& \xrightarrow{} ... \xrightarrow{} T_Y \otimes \Sym^{m-1} \N_X
\xrightarrow{} \Sym^{m} \N_X \xrightarrow{} 0
\end{aligned}
\end{equation}
of $\wedge^{m-\bullet} T_Y\otimes \Sym^\bullet \N_X$ induced by $\eta$.

Setting $m=n$ in \eqref{COMVE412}, we have
\begin{equation}\label{COMVE107}
\begin{tikzcd}[column sep=13pt]
T_X^n\otimes K_X\arrow[equal]{d}\arrow[hook]{r} & T_Y^n\otimes K_X \arrow[equal]{d}\arrow{r}{\eta_n} &
T_Y^{n-1}\otimes K_X\otimes \N_{X}\arrow[equal]{d}\\
\Omega_X^N \arrow[hook]{r} & \Omega_Y^{N+k}\otimes \det(\N_X) \arrow{r} &
\Omega_Y^{N+1+k}\otimes \det(\N_X)\otimes \N_{X}
\end{tikzcd}
\end{equation}
where $\det(\N_X) = \wedge^{k} \N_X$ for $k = \dim Y - \dim X$. Let us first prove:

\begin{prop}\label{COMVPROPKEY}
Let $X\subset Y = B\times P$ be a smooth projective family of varieties in a smooth projective variety $P$
passing through a fixed point $p\in P$ over a smooth variety $B$ with the section $\sigma: B\to X$ given by
$\sigma(b) = p$ for $b\in B$.
Then the diagram
\begin{equation}\label{COMVE414}
\begin{tikzcd}[column sep=18pt]
\Omega_X^N \arrow[hook]{r} \arrow{d}{d\sigma} & \Omega_Y^{N+k}\otimes \det(\N_X) \arrow{r}\arrow{d} &
\Omega_Y^{N+k+1}\otimes \det(\N_X)\otimes \N_{X}\arrow{d}\\
\Omega_{\sigma(B)}^N \arrow[hook]{r} &
F^N \Omega_Y^{N+k}\otimes \det(\N_X)\Big|_{\sigma(B)} \arrow{r} &
F^N \Omega_Y^{N+k+1}\otimes \det(\N_X)\otimes \N_{X} \Big|_{\sigma(B)}
\end{tikzcd}
\end{equation}
commutes and has left exact rows, where $N=\dim B$, $k = \dim Y - \dim X$, $\det(\N_X) = \wedge^k \N_X$ and
the vertical maps in the second and third columns are induced by the projections $\Omega_Y^\bullet \to
F^N \Omega_Y^\bullet$ followed by the restrictions to $\sigma(B)$.
\end{prop}

\begin{proof}
The rows of \eqref{COMVE414} are induced by Koszul complex \eqref{COMVE415} and hence left exact.

We want to point out that the diagram
\begin{equation}\label{COMVE493}
\begin{tikzcd}[column sep=18pt]
\Omega_Y^m\arrow{r}{\eta}\arrow{d} &
\Omega_Y^{m+1}\otimes \N_{X}\arrow{d}\\
F^l \Omega_Y^{m}\arrow{r} &
F^l \Omega_Y^{m+1}\otimes \N_{X}
\end{tikzcd}
\end{equation}
{\em does not} commute in general. However, it commutes when we restrict the bottom row to
$\sigma(B)$. That is, we claim that the diagram
\begin{equation}\label{COMVE494}
\begin{tikzcd}[column sep=18pt]
\Omega_Y^m\arrow{r}{\eta}\arrow{d}[left]{\rho_{m}} &
\Omega_Y^{m+1}\otimes \N_{X}\arrow{d}{\rho_{m+1}}\\
F^l \Omega_Y^{m}\Big|_{\sigma(B)}\arrow{r}{\eta_\sigma} &
F^l \Omega_Y^{m+1}\otimes \N_{X}\Big|_{\sigma(B)}
\end{tikzcd}
\end{equation}
commutes. Of course, this implies that the right square of \eqref{COMVE414} commute.

Let $(x_1,x_2,...,x_r)$ and $(t_1,t_2,...,t_N)$
be the local coordinates of $P$ and $B$, respectively. Let
$p = \{x_1=x_2=...=x_r=0\}$ and
\begin{equation}\label{COMVE495}
X=\{f_1(x,t)=f_2(x,t)=...=f_k(x,t)=0\}.
\end{equation}
Then $\eta$ is given by
\begin{equation}\label{COMVE496}
\eta(\omega) =(\omega\wedge df_1, \omega\wedge df_2, ...,
\omega\wedge df_k).
\end{equation}
Since $p\in X_b$ for all $b\in B$, we have $f_i(0,t)\equiv 0$. Hence
\begin{equation}\label{COMVE497}
\left.\frac{\partial f_i}{\partial t_j}\right|_{x=0} = 0
\end{equation}
for all $t$, $i=1,2,...,k$ and $j=1,2,...,N$. It follows that
\begin{equation}\label{COMVE498}
\begin{aligned}
&\quad \rho_{m+1}\circ \eta(\omega_1) = \rho_{m+1}(\omega_1\wedge df_1, \omega_1\wedge df_2, ...,
\omega_1\wedge df_k)\\
&=(\rho_{m+1}(\omega_1\wedge df_1), \rho_{m+1}(\omega_1\wedge df_2), ...,
\rho_{m+1}(\omega_1\wedge df_k))\\
&= 0 = \eta_\sigma \circ \rho_m(\omega_1)
\end{aligned}
\end{equation}
for all local sections
\begin{equation}\label{COMVE499}
\omega_1\in H^0(U,\bigoplus_{i < l} \pi_B^* \Omega_B^i \otimes \pi_P^* \Omega_P^{m-i}) \subset  H^0(U, \Omega_Y^{m}),
\end{equation}
where $U$ is an open subset of $Y$. Every $\omega\in H^0(U, \Omega_Y^{m})$
can be written as
\begin{equation}\label{COMVE550}
\omega = \omega_1 + \omega_2
\end{equation}
with $\omega_1$ given in \eqref{COMVE499} and $\omega_2\in H^0(U, F^l \Omega_Y^m)$. It is clear that
\begin{equation}\label{COMVE551}
\rho_{m+1}\circ \eta(\omega_2)
= \eta_\sigma \circ \rho_m(\omega_2).
\end{equation}
Combining \eqref{COMVE498} and \eqref{COMVE551}, we conclude that
\begin{equation}\label{COMVE552}
\rho_{m+1}\circ \eta(\omega)
= \eta_\sigma \circ \rho_m(\omega)
\end{equation}
and hence the diagram \eqref{COMVE494} commutes. It remains to prove that the left square of \eqref{COMVE414} commutes.

Note that $\Omega_X^N$ can be identified with the image of the map
\begin{equation}\label{COMVE553}
\begin{tikzcd}
\Omega_Y^N\otimes \CO_X \ar{r}{\theta} & \Omega_Y^{N+k}\otimes \det(\N_X)
\end{tikzcd}
\end{equation}
given by
\begin{equation}\label{COMVE554}
\theta(\omega) = \omega \wedge df_1\wedge df_2\wedge \cdots\wedge df_k\red{\otimes\frac{\partial}{\partial f_1}\wedge\cdots\wedge \frac{\partial}{\partial f_k}}.
\end{equation}
By \eqref{COMVE497} again, we see that the diagram
\begin{equation}\label{COMVE555}
\begin{tikzcd}
\Omega_Y^N\otimes \CO_X \ar{r}{\theta}\ar{d}{d\sigma} & \Omega_Y^{N+k}\otimes \det(\N_X)\ar{d}\\
F^N\Omega_Y^N\otimes \CO_X\Big|_{\sigma(B)} \ar{r} &
F^N \Omega_Y^{N+k}\otimes \det(\N_X)\Big|_{\sigma(B)}
\end{tikzcd}
\end{equation}
commutes. Thus, the diagram
\begin{equation}\label{COMVE556}
\begin{tikzcd}
\Omega_Y^N\otimes \CO_X \ar[bend left=20]{rr}{\theta}\ar{d}{d\sigma} \ar[two heads]{r} &
\Omega_X^N \ar[hook]{r} \ar{d}{d\sigma} &
 \Omega_Y^{N+k}\otimes \det(\N_X)\ar{d}\\
F^N\Omega_Y^N\otimes \CO_X\Big|_{\sigma(B)} \ar[equal]{r} &
\Omega_{\sigma(B)}^N \ar[hook]{r} &
F^N \Omega_Y^{N+k}\otimes \det(\N_X)\Big|_{\sigma(B)}
\end{tikzcd}
\end{equation}
commutes.
\end{proof}

Note that
\begin{equation}\label{COMVE420}
F^N \Omega_Y^{N+k} = \pi_B^* \Omega_B^N \otimes \pi_P^* \Omega_P^k
\end{equation}
Combining \eqref{COMVE107}, \eqref{COMVE414} and \eqref{COMVE420}, we obtain commutative diagrams
\begin{equation}\label{COMVE417}
\begin{tikzcd}[column sep=small]
T_X^n\otimes K_X\arrow{d}{d\sigma_i}\arrow[hook]{r} & T_Y^n\otimes K_X \arrow{d}{\alpha_{n,i}}\arrow{r}{\eta_n} &
T_Y^{n-1}\otimes K_X\otimes \N_{X}\arrow{d}{\alpha_{n-1,i}}\\
K_{\sigma_i(B)} \arrow[hook]{r} &
\pi_P^* T_P^n \otimes K_X \Big|_{\sigma_i(B)}\arrow{r} &
\pi_P^* T_P^{n-1} \otimes K_X \otimes \N_{X} \Big|_{\sigma_i(B)}
\end{tikzcd}
\end{equation}
with left exact rows for $i=1,2$. By the above diagram, we have
\begin{equation}\label{COMVE483}
\ker(\Gamma_b(d\sigma_i)) = \ker(\Gamma_b(\alpha_{n,i}))
\cap \ker(\Gamma_b(\eta_n))
\end{equation}
for $i=1,2$. Therefore,
\eqref{COMVE482} is equivalent to
\begin{equation}\label{COMVE416}
\boxed{\ker(\Gamma_b(\alpha_{n,1}))
\cap \ker(\Gamma_b(\eta_n))
\ne
\ker(\Gamma_b(\alpha_{n,2}))
\cap \ker(\Gamma_b(\eta_n))}.
\end{equation}

More explicitly, we can write $\Gamma_b(T_Y^n\otimes K_X)$ as
\begin{equation}\label{COMVE457}
\Gamma_b(T_Y^n\otimes K_X) = \Gamma_b(\pi_P^* T_P^n\otimes K_X) \oplus
\sum_{j<n} \Gamma_b(\pi_P^* T_P^j\otimes K_X)\otimes T_{B,b}^{n-j}.
\end{equation}
Then the kernel of $\Gamma_b(\alpha_{n,i})$ is
\begin{equation}\label{COMVE471}
\begin{aligned}
\ker(\Gamma_b(\alpha_{n,i})) &= \Gamma_b(\pi_P^*T_P^n \otimes K_X(-p_i))
\\
&\quad\oplus
\sum_{j<n} \Gamma_b(\pi_P^* T_P^{j}\otimes K_X)\otimes T_{B,b}^{n-j}
\end{aligned}
\end{equation}
for $i=1,2$, where $K_X(-p_i) = K_X\otimes I_{p_i}$ for $I_{p_i}$ the ideal sheaf of $p_i$.
So \eqref{COMVE416} is equivalent to
\begin{equation}\label{COMVE475}
\boxed{\begin{aligned}
&\quad \ker(\Gamma_b(\eta_n)) \cap (\Gamma_b(\pi_P^*T_P^n \otimes K_X(-p_1))
\oplus
\sum_{j<n} \Gamma_b(\pi_P^* T_P^{j}\otimes K_X)\otimes T_{B,b}^{n-j})
\\
&\ne
\ker(\Gamma_b(\eta_n)) \cap (\Gamma_b(\pi_P^*T_P^n \otimes K_X(-p_2))
\oplus
\sum_{j<n} \Gamma_b(\pi_P^* T_P^{j}\otimes K_X)\otimes T_{B,b}^{n-j})
\end{aligned}}
\end{equation}

Combining it with Proposition \ref{COMVPROPUNI}, we obtain the following criterion:

\begin{prop}\label{COMVPROPKEYCRITERION}
Let $X\subset Y = B\times P$ be a smooth projective family of $n$-dimensional varieties in a projective space
$P$ passing through two fixed points $p_1\ne p_2\in P$ over a smooth variety $B$. Then $p_1$ and $p_2$ are
not $\Gamma$-equivalent over $\BQ$ on $X_b$ for $b\in B$ general if \eqref{COMVE475} holds at a point $b$
where $h^0(X_t, T_X^n\otimes K_X)$ is locally constant in $t$.
\end{prop}

\begin{rem}\label{COMVREM000}
Since $X\subset B\times P$ is a family of varieties in $P$ passing through $p_i$,
$\eta(\mathbf v)$ is a section in $H^0(N_{X_b})$ vanishing at $p_i$ for all tangent vectors $\mathbf v\in T_{B,b}$ and $i=1,2$. It follows that
\begin{equation}\label{COMVE472}
\begin{aligned}
&\quad \eta_n\left(\sum_{j=0}^{n-1} \Gamma_b(\pi_P^* T_P^j\otimes K_X)\otimes T_{B,b}^{n-j}\right)
\\
&\subset \bigcap_{i=1}^2 \ker(\Gamma_b(\alpha_{n-1,i}))
= \Gamma_b(\pi_P^*T_P^{n-1} \otimes K_X\otimes \N_X(-p_1-p_2))\\
&\quad\quad\quad\quad\quad\quad\quad\quad\quad\quad
\oplus\sum_{j=0}^{n-2} \Gamma_b(\pi_P^* T_P^{j}\otimes K_X\otimes \N_X)\otimes T_{B,b}^{n-1-j}.
\end{aligned}
\end{equation}
\end{rem}

Let us apply Proposition \ref{COMVPROPKEYCRITERION} to complete intersections in $P=\PP^{n+k}$ of type $(d_1,d_2,...,d_k)$.
When $\sum (d_i-1) = 2n + 1$, we have
$K_X = \CO_X(n)$.
More general, let us consider a smooth projective family $X\subset Y=B\times P$ of varieties of dimension $n$ in $P$
with $K_X(-n)$ globally generated on each fiber $X_b$.
In this case, we have the following corollary of Proposition \ref{COMVPROPKEYCRITERION}.

\begin{cor}\label{COMVCORKEYCRITERION}
Let $X\subset Y = B\times P$ be a smooth projective family of $n$-dimensional varieties in a projective space $P$
passing through two fixed points $p_1\ne p_2\in P$ over a smooth variety $B$ and let $W_{X,b}$ be the
subspace of $\Gamma_b(T_P(1))$ defined by
\begin{equation}\label{COMVE461}
W_{X,b} = \Big\{ \omega\in \Gamma_b(T_P(1)):
\eta(\omega) \in \eta\big(\Gamma_b(\CO(1))\otimes T_{B,b}\big)
\Big\},
\end{equation}
where the map $\eta$ on $\Gamma_b(T_P(1))$ and $\Gamma_b(\CO(1))\otimes T_{B,b}$ are given by the diagram
\begin{equation}\label{COMVE476}
\begin{tikzcd}
\Gamma_b(T_Y(1))\ar{d}[left]{\eta}\ar[equal]{r} & \Gamma_b(T_P(1)) \oplus \Gamma_b(\CO(1))\otimes T_{B,b}\ar{dl}\\
\Gamma_b(\N_X(1))
\end{tikzcd}
\end{equation}
Suppose that there exists a point $b\in B$ such that
$h^0(X_t, T_X^n\otimes K_X)$ is constant for $t$ in an open neighborhood of $b$,
each point $p_i$ imposes independent conditions on
both $K_{X_b}(-n)$ and $T_X(1)\otimes \CO_{X_b}$, i.e., the maps
\begin{equation}\label{COMVE478}
\begin{tikzcd}
\Gamma_b(K_X(-n))\ar[two heads]{r} & K_X(-n)\otimes \CO_{p_i}
\end{tikzcd}
\text{ and}
\end{equation}
\begin{equation}\label{COMVE479}
\begin{tikzcd}
\Gamma_b(T_X(1)) \ar[two heads]{r} & T_X(1)\otimes \CO_{p_i}
\end{tikzcd}
\end{equation}
are surjective for $i=1,2$ and
\begin{equation}\label{COMVE473}
\big\{\omega\in W_{X,b} : \omega(p_1) = 0\big\} \ne
\big\{\omega\in W_{X,b} : \omega(p_2) = 0\big\}.
\end{equation}
Then $p_1$ and $p_2$ are not $\Gamma$-equivalent over $\BQ$ on $X_b$ for $b\in B$ general.
\end{cor}

\begin{proof}
By \eqref{COMVE473}, there exists $\omega\in W_{X,b}$ such that $\omega(p_i) = 0$ and $\omega(p_{3-i}) \ne 0$
for $i=1$ or $2$. Without loss of generality, let us assume that
$\omega_1(p_1) = 0$ and $\omega_1(p_2) \ne 0$ for some $\omega_1\in W_{X,b}$.

It is easy to see that $W_{X,b}$ is the image of the projection from $\Gamma_b(T_X(1))$ to
$\Gamma_b(T_P(1))$ via the diagram
\begin{equation}\label{COMVE492}
\begin{tikzcd}
\Gamma_b(T_X(1))\ar[hook]{r} \ar{dr} & \Gamma_b(T_Y(1))\ar{r}{\eta}\ar{d}
& \Gamma_b(\N_X(1))\\
& \Gamma_b(T_P(1))
\end{tikzcd}
\end{equation}
where $\Gamma_b(T_X(1))$ can be identified with $\ker(\eta)$.
In other words, for every $\omega\in W_{X,b}$, there exists $\tau\in \Gamma_b(\CO(1))\otimes T_{B,b}$
such that $\eta(\omega + \tau) = 0$ and hence
$\omega + \tau \in \Gamma_b(T_X(1))$.

By \eqref{COMVE479}, $\Gamma_b(T_X(1))$ generates the vector space $T_X(1)\otimes \CO_{p_2}$. On the other hand,
by the diagram
\begin{equation}\label{COMV465}
\begin{tikzcd}
T_{X_b}(1)\otimes \CO_{p_2} \ar[hook]{r} \ar{d} &
T_X(1)\otimes \CO_{p_2} \ar{r} \ar{d} & T_{B,b}\ar[equal]{d}\\
T_{Y_b}(1) \otimes \CO_{p_2} \ar[equal]{rd}\ar[hook]{r} & T_Y(1)\otimes \CO_{p_2} \ar{r} \ar{d} & T_{B,b}\\
& T_P(1)\otimes \CO_{p_2}
\end{tikzcd}
\end{equation}
we see that the image of the projection $T_X(1)\otimes \CO_{p_2}\to T_P(1)\otimes \CO_{p_2}$ is the same as the
 image of the map $T_{X_b}(1)\otimes \CO_{p_2}\to T_{Y_b}(1)\otimes \CO_{p_2}$ and thus has dimension $n$. Therefore,
\begin{equation}\label{COMVE432}
\dim \{ \omega(p_2): \omega \in W_{X,b}\} = n.
\end{equation}
And since $\omega_1(p_2) \ne 0$, we can find $\omega_2,...,\omega_n\in W_{X,b}$ such that $\{\omega_j(p_2)\}$ are
linearly independent. On the other hand, $\omega_1(p_1) = 0$ and hence $\{\omega_j(p_1)\}$ are linearly dependent. In other words,
\begin{equation}\label{COMVE481}
\left\{\begin{aligned}
\omega_1(p_1)\wedge\omega_2(p_1)\wedge ... \wedge \omega_n(p_1) &= 0\\
\omega_1(p_2)\wedge\omega_2(p_2)\wedge ... \wedge \omega_n(p_2) &\ne 0.
\end{aligned}
\right.
\end{equation}
Let $\eta(\omega_j + \tau_j) = 0$ for some $\tau_j\in \Gamma_b(\CO(1))\otimes T_{B,b}$ and $j=1,2,...,n$.
Then
\begin{equation}\label{COMVE462}
\bigwedge_{j=1}^n (\omega_j + \tau_j) \otimes s\in \ker(\Gamma_b(\eta_n))
\end{equation}
for all $s\in \Gamma_b(K_X(-n))$. By \eqref{COMVE481}, we have
\begin{equation}\label{COMVE470}
\begin{aligned}
\bigwedge_{j=1}^n \omega_j \otimes s & \in \Gamma_b(\pi_P^*T_P^n \otimes K_X(-p_1)) \text{ and}\\
\bigwedge_{j=1}^n \omega_j \otimes s &\not\in \Gamma_b(\pi_P^*T_P^n \otimes K_X(-p_2))
\end{aligned}
\end{equation}
provided that $s(p_2) \ne 0$. The combination of \eqref{COMVE462} and \eqref{COMVE470} yields \eqref{COMVE475}.
\end{proof}

Since the validity of \eqref{COMVE473} is determined by the restriction of $W_{X,b}$ to $Z = \{p_1,p_2\}$, we may let
$W_{X,b,Z}$ be the subspace of $H^0(Z, T_P(1))$ given by
\begin{equation}\label{COMVE074}
\begin{aligned}
W_{X,b,Z} &= W_{X,b}\otimes H^0(\CO_Z)\\
&= \Big\{ \omega\big|_Z: \omega\in \Gamma_b(T_P(1)) \text{ and }
\eta(\omega) \in \eta\big(\Gamma_b(\CO(1))\otimes T_{B,b}\big)
\Big\}
\end{aligned}
\end{equation}
and reformulate \eqref{COMVE473} as
\begin{equation}\label{COMVE073}
W_{X,b,Z} \cap H^0(Z, T_P(1) \otimes I_p) \ne 0
\end{equation}
for some $p\in \supp(Z) = \{p_1,p_2\}$.

\subsection{Criterion for two varying sections}

So far we have obtained the key criterion, Corollary \ref{COMVCORKEYCRITERION}, for the $\Gamma$-equivalence
of two fixed sections of $X/B$ in the ambient space $P$. To apply it to two arbitrary sections of $X/B$, we need to
use an automorphism $\lambda\in \Aut(Y/B)$ to move these two sections to two fixed points in $P$, as pointed
out before. This line of argument leads to the following:

\begin{prop}\label{COMVPROPVSA}
Let $X\subset Y = B\times P$ be a smooth projective family of $n$-dimensional varieties in a projective
space $P$ over the $N$-dimensional polydisk $B = \Spec \BC[[t_j]]$ and let $\sigma_i: B\to X$ be two disjoint sections of $X/B$ with
$p_i = \sigma_i(b)$ at the origin $b\in B$ for $i=1,2$.
Let $\lambda\in B\times \Aut(P)$ be an automorphism of $Y$ preserving the base $B$, satisfying that
$\lambda_b = \text{id}$ and
$\lambda(\sigma_i(t)) \equiv p_i$ for $i=1,2$ and all $t\in B$ and
given by
\begin{equation}\label{COMVE454}
\lambda \begin{bmatrix}
x_0\\
x_1\\
\vdots\\
x_{r}
\end{bmatrix}
= \Lambda \begin{bmatrix}
x_0\\
x_1\\
\vdots\\
x_{r}
\end{bmatrix},
\end{equation}
where $(x_0,x_1,...,x_r)$ are the homogeneous coordinates of $P$ and
$\Lambda = \Lambda(t)$ is an $(r+1)\times (r+1)$ matrix
over $\BC[[t_j]]$ satisfying $\Lambda(0) = I$.
Let $W_{X,b,Z,\lambda}$ be the subspace of $H^0(Z, T_P(1))$ defined by
\begin{equation}\label{COMVE491}
\begin{aligned}
W_{X,b,Z,\lambda}
&= \Big\{
\omega\big|_Z + L_\lambda(\tau): \omega \in \Gamma_b(T_P(1)),
\tau\in \Gamma_b(\CO(1)) \otimes T_{B,b},\\
&\hspace{90pt} \eta(\omega + \tau) = 0
\Big\}
\end{aligned}
\end{equation}
for $Z = \{p_1,p_2\}$, where $L_\lambda: \pi_B^* T_{B,b}\to T_P\otimes \CO_Z$
is the map given by
\begin{equation}\label{COMVE090}
L_\lambda\left(\frac{\partial}{\partial t_j}\right)
=
\begin{bmatrix}
x_0 & x_1 & ... & x_{r}
\end{bmatrix}
\left.\frac{\partial\Lambda^T}{\partial t_j}\right|_{t=0} \begin{bmatrix}
\partial/\partial x_0\\
\partial/\partial x_1\\
\vdots\\
\partial/\partial x_{r}
\end{bmatrix}.
\end{equation}

Suppose that
\begin{itemize}
\item $h^0(X_t, T_X^n\otimes K_X)$ is constant over $B$,
\item each point $p_i$ imposes independent conditions on
$$
K_{X_b}(-n) \text{ and } T_X(1)\otimes \CO_{X_b}
$$
for $i=1,2$,
\item and
\begin{equation}\label{COMVE600}
W_{X,b,Z,\lambda} \cap H^0(Z, T_P(1)\otimes I_p) \ne 0
\end{equation}
for some $p\in \supp(Z)$.
\end{itemize}
Then $\sigma_1(t)$ and $\sigma_2(t)$ are not $\Gamma$-equivalent over $\BQ$ on $X_t$ for $t\in B$ general.
\end{prop}

\begin{proof}
Note that $W_{X,b,Z,\lambda} = W_{X,b,Z}$ if $L_\lambda = 0$, i.e., $\sigma_i(t) \equiv p_i$.

Let $\widehat{X} = \lambda(X) \subset Y= B \times P$. Obviously,
$\widehat{X}$ is a smooth projective family of $n$-dimensional varieties in $P$ over $B$ passing through the two fixed points $p_1\ne p_2$.

We define the map $\widehat{\eta}: T_Y \otimes \CO_{\widehat{X}}\to \N_{\widehat{X}}$ and the space
$W_{\widehat{X},b}\subset \Gamma_b(T_P(1))$
for $\widehat{X}\subset Y = B\times P$ in the same way as $\eta$ and $W_{X,b}$. Note that since $\lambda_b = \text{id}$,
$X_b = \widehat{X}_b$ and we may use $\Gamma_b(\bullet)$ to refer both $H^0(X_b, \bullet)$ and $H^0(\widehat{X}_b, \bullet)$.

Let us consider the commutative diagram:
\begin{equation}\label{COMVE458}
\begin{tikzcd}
\Gamma_b(T_X(1)) \ar{d}{d \lambda}[left]{\cong} \ar[hook]{r} & \Gamma_b(T_Y(1)) \ar{d}{d \lambda}[left]{\cong}
\ar{r}{\eta} & \Gamma_b(\N_X(1))\ar{d}\\
\Gamma_b(T_{\widehat{X}}(1)) \ar{dr}\ar[hook]{r} & \Gamma_b(T_Y(1)) \ar{d}{\pi_{P,*}} \ar{r}{\widehat{\eta}}
& \Gamma_b(\N_{\widehat{X}}(1))\\
& \Gamma_b(T_P(1))
\end{tikzcd}
\end{equation}
As pointed out in the proof of Corollary \ref{COMVCORKEYCRITERION}, $W_{X,b}$ is simply the image of the
projection from $\Gamma_b(T_X(1))$ to $\Gamma_b(T_P(1))$ when $\Gamma_b(T_X(1))$ is identified with the kernel of
$\eta: \Gamma_b(T_Y(1)) \to \Gamma_b(\N_X(1))$.
The same holds for $\widehat{X}$. That is,
$W_{\widehat{X},b}$ is simply the image of the projection from $\Gamma_b(T_{\widehat{X}}(1))$ to $\Gamma_b(T_P(1))$
when $\Gamma_b(T_{\widehat{X}}(1))$ is identified with the kernel of
$\widehat{\eta}: \Gamma_b(T_Y(1)) \to \Gamma_b(\N_{\widehat{X}}(1))$.

We may regard $W_{\widehat{X},b}$ as the image of $\Gamma_b(T_X(1))$
under the map $\pi_{P,*} \circ d \lambda$ in the above diagram. Note that $\pi_{P,*} \circ d \lambda$ is not the same as the projection
$\pi_{P,*}: \Gamma_b(T_Y(1))\to \Gamma_b(T_P(1))$, i.e.,
\begin{equation}\label{COMVE463}
\pi_{P,*} \circ d \lambda \ne \pi_{P,*}.
\end{equation}
Indeed, we have
\begin{equation}\label{COMVE456}
(d\lambda) (\omega + \tau) = (\omega + \widehat{L}_\lambda(\tau)) + \tau
\end{equation}
for $\omega \in \Gamma_b(T_P(1))$ and $\tau\in \Gamma_b(\CO(1))\otimes T_{B,b}$, where
\begin{equation}\label{COMV487}
\widehat{L}_\lambda:
\begin{tikzcd}
\pi_B^* T_B \ar{r} & \pi_P^* T_P
\end{tikzcd}
\end{equation}
is a homomorphism induced by $d\lambda: T_Y\to T_Y$. Thus,
\begin{equation}\label{COMVE465}
\pi_{P,*} \circ (d \lambda)(\omega + \tau)
= \omega + \widehat{L}_\lambda(\tau) \ne \omega = \pi_{P,*} (\omega + \tau).
\end{equation}
It follows that
\begin{equation}\label{COMVE460}
\begin{aligned}
W_{\widehat{X},b} &= \pi_{P,*} \circ d \lambda
(\Gamma_b(T_X(1)))\\
&= \big\{
\omega + \widehat{L}_\lambda(\tau): \omega \in \Gamma_b(T_P(1)),
\tau\in \Gamma_b(\CO(1))\otimes T_{B,b},\\
&\hspace{78pt} \eta(\omega + \tau) = 0
\big\}.
\end{aligned}
\end{equation}
We claim that $L_\lambda$ and $W_{X,b,Z,\lambda}$ are exactly
the restrictions of $\widehat{L}_\lambda$ and $W_{\widehat{X},b}$ to $Z$, respectively.
Indeed, the differential map $d\lambda: T_Y\to T_Y$ is given by
\begin{equation}\label{COMVE490}
\begin{aligned}
(d\lambda)\left(\frac{\partial}{\partial x_i}\right)
&= \frac{\partial}{\partial x_i}\\
(d\lambda)\left(\frac{\partial}{\partial t_j}\right)
&= \frac{\partial}{\partial t_j} + \widehat{L}_\lambda\left(\frac{\partial}{\partial t_j}
\right)
\\
&= \frac{\partial}{\partial t_j} +
\begin{bmatrix}
x_0 & x_1 & ... & x_{r}
\end{bmatrix}
\frac{\partial\Lambda^T}{\partial t_j} \begin{bmatrix}
\partial/\partial x_0\\
\partial/\partial x_1\\
\vdots\\
\partial/\partial x_{r}
\end{bmatrix}
\end{aligned}
\end{equation}
at $b$. Therefore, $L_\lambda$ is the restriction of $\widehat{L}_\lambda$ to $Z$ and hence $W_{\widehat{X},b,Z} = W_{X,b,Z,\lambda}$.

In conclusion, the hypothesis \eqref{COMVE600} on $W_{X,b,Z,\lambda}$
translates to
\begin{equation}\label{COMVE132}
\Big\{\omega\in W_{\widehat{X},b} : \omega(p_1) = 0 \Big\} \ne
\Big\{\omega\in W_{\widehat{X},b} : \omega(p_2) = 0 \Big\}.
\end{equation}
Then by Corollary \ref{COMVCORKEYCRITERION}, $\sigma_1(t)$ and $\sigma_2(t)$ are not $\Gamma$-equivalent
over $\BQ$ on a general fiber $X_t$ of $X/B$.
\end{proof}

\begin{rem}\label{COMVREMPROPVSA}
In the above proof, it is easy to see that
\begin{equation}\label{COMV486}
\widehat{\eta}\left(\frac{\partial}{\partial x_i}\right) = \eta\left(\frac{\partial}{\partial x_i}\right)
\text{ and }
\widehat{\eta}\left(\frac{\partial}{\partial t_j}\right) = \eta\left(\frac{\partial}{\partial t_j}
- \widehat{L}_\lambda\left(\frac{\partial}{\partial t_j}\right)
\right).
\end{equation}
Since $\widehat{X}_t$ passes through $p_1$ and $p_2$,
$\widehat{\eta}(\tau)$ vanishes at $p_i$ and hence
$L_\lambda$ satisfies
\begin{equation}\label{COMVE467}
\eta(L_\lambda(\tau)) = \eta(\tau)\Big|_Z
\text{ for all } \tau\in T_{B,b}.
\end{equation}
\end{rem}

There is a more intrinsic way to define $L_\lambda$:
for every $t\in B$, we consider the line joining the two points $\sigma_i(t)$; we may regard $\sigma_i(t)$ as the image of two
fixed points on $\PP^1$ mapped to this line and thus interpret $L_\lambda$ in terms of the deformation of this map $\PP^1\to P$.
We can put the above proposition in the following equivalent form.

\begin{prop}\label{COMVPROPVSB}
Let $X\subset Y = B\times P$ be a smooth projective family of $n$-dimensional varieties in a projective
space $P$ over a smooth variety $B$ and let $v: S = B\times \PP^1\hookrightarrow Y$ be a closed immersion
preserving the base $B$ such that
$v^* \CO_Y(1) = \CO_S(1)$ and there are two fixed points $p_1\ne p_2$ on $\PP^1$ with $v_b(p_i) \in X_b$ for all $b\in B$.
Let $W_{X,b,Z,\lambda}$ be the subspace of $H^0(Z, v_b^* T_P(1))$ defined by
\begin{equation}\label{COMVE601}
\begin{aligned}
W_{X,b,Z,\lambda}
&= \Big\{
v_b^* \omega\big|_Z - L_\lambda(v_b^* \tau): \omega \in \Gamma_b(T_P(1)),
\\
&\hspace{108pt} \tau\in \Gamma_b(\CO(1)) \otimes T_{B,b},\\
&\hspace{108pt} \eta(\omega + \tau) = 0
\Big\}
\end{aligned}
\end{equation}
for $Z = \{p_1,p_2\}$, where $L_\lambda: \pi_{S,B}^* T_{B,b}\to v_b^* T_P\otimes \CO_Z$ is the map induced by $T_S\to v^* T_Y$
with $\pi_{S,B}$ the projection $S\to B$.

Let $b$ be a general point of $B$.
Suppose that each point $v_b(p_i)$ imposes independent conditions
on $K_{X_b}(-n)$ and $T_X(1)\otimes \CO_{X_b}$ for $i=1,2$ and
\begin{equation}\label{COMVE091}
W_{X,b,Z,\lambda} \cap H^0(Z, v_b^* T_P(1)\otimes I_p) \ne 0
\end{equation}
for some $p\in Z$. Then $v_b(p_1)$ and $v_b(p_2)$ are not $\Gamma$-equivalent over $\BQ$ on $X_b$.
\end{prop}

\begin{proof}
Note that the hypothesis $v^* \CO_Y(1) = \CO_S(1)$ simply means that $v$ maps $S/B$ fiberwise to lines in $P$.

Basically, we want to show that the two spaces $W_{X,b,Z,\lambda}$ defined by \eqref{COMVE491} and \eqref{COMVE601} are identical. In turn, this comes down to showing that the maps $L_\lambda$ are ``essentially'' the same up to a sign.

To be more precise, we fix a general point $b$, replace $B$ by an analytic neighborhood of $b$ and choose $\lambda\in \Aut(Y/B)$ to be an automorphism such that $\lambda(v_t(p_i))
\equiv v_b(p_i)$ for $i=1,2$ and all $t\in B$. Then we have a map
$L_\lambda$ defined by \ref{COMVE090}.
Let us rename this map to $\overline{L}_\lambda: \pi_B^* T_{B,b}\to T_P\otimes \CO_{v_b(Z)}$.

Let $\Delta_i = \pi_{S,\PP^1}^{-1}(p_i)$ be the two sections of $S/B$ given by $p_i$ for $i=1,2$, where $\pi_{S,\PP^1}$ is the projection $S\to \PP^1$.
Since $\pi_P \circ \lambda \circ v$ is constant on $\Delta_i$,
we see that
$$L_{\lambda\circ v} \equiv 0$$
from the commutative diagram
$$
\begin{tikzcd}
\pi_{S,B}^* T_B \otimes \CO_{\Delta_i} \ar[equal]{r} \ar{dr} & T_{\Delta_i}\ar[equal]{r} \ar{d} & (\lambda \circ v)^* T_{v(\Delta_i)}\ar{d}\\
& T_S\otimes \CO_{\Delta_i} \ar{r} & (\lambda \circ v)^* T_Y \otimes \CO_{\Delta_i}
\end{tikzcd}
$$
where the map
$L_{\lambda\circ v}: \pi_{S,B}^* T_{B,b}\to (\lambda \circ v_b)^* T_P\otimes \CO_Z$ is defined in an analogous way to $L_\lambda$.

Since
$$
d (\lambda \circ v) = (d \lambda) \circ (d v)
$$
we derive that
$$
L_{\lambda\circ v}(v_b^* \tau) = L_\lambda(v_b^* \tau) +  v_b^*\left(\overline{L}_\lambda(\tau)\right)
$$
using \eqref{COMVE456}. Then
$$
v_b^*\left(\overline{L}_\lambda(\tau)\right) = -L_\lambda(v_b^* \tau)
$$
and the proposition follows.
\end{proof}

Using Proposition \ref{COMVPROPVSA} or \ref{COMVPROPVSB}, we obtain the following criterion for the $\Gamma$-inequivalence of all pairs of distinct points on $X_b$.

\begin{cor}\label{COMVCORKEYCRITERIONREAL}
Let $X\subset Y = B\times P$ be a smooth projective family of $n$-dimensional varieties in a projective space $P$ over a smooth variety $B$ and
let $W_{X,b,Z,\lambda}$ be the subspace of $H^0(Z, T_P(1))$ defined by
\eqref{COMVE491} for a $0$-dimensional subscheme $Z\subset X_b$ and
$L_\lambda\in \Hom(\pi_B^* T_{B,b}, T_P\otimes \CO_Z)$.

Let $b$ be a very general point of $B$.
Suppose that
\begin{itemize}
\item
$K_{X_b}(-n)$ and $T_X(1)\otimes \CO_{X_b}$ are globally generated on $X_b$ and
\item
\eqref{COMVE600} holds for \underline{all} pairs $Z= \{p_1,p_2\}$ of distinct
points $p_1\ne p_2$ on $X_b$, \underline{some} $p\in \supp(Z)$ and \underline{all} $L_\lambda\in \Hom(\pi_B^* T_{B,b}, T_P
\otimes \CO_Z)$ satisfying \eqref{COMVE467}.
\end{itemize}
Then no two distinct points on $X_b$ are $\Gamma$-equivalent over $\BQ$.
\end{cor}

We believe that the above corollary will find applications in the future.
However, we will not use it to prove our main theorem \ref{COMVTHMHYPERSURFACE}; instead, we will apply Proposition \ref{COMVPROPVSA} directly to families $X\subset B\times \PP^{n+1}$ of hypersurfaces of degree $2n+2$ in $\PP^{n+1}$.

\section{Hypersurfaces of degree $2n+2$ in $\PP^{n+1}$}

In this section, we are going to prove our main theorem \ref{COMVTHMHYPERSURFACE} using the
criteria developed in the previous section. Here is an outline of the proof.

To start, let us choose a versal family $X \subset Y = B\times \PP^{n+1}$ of hypersurfaces of degree $2n+2$ in $\PP^{n+1}$.
Suppose that the theorem fails. Then using a Hilbert scheme argument, we can find two disjoint sections $\sigma_i: U\to X$
in an analytic (or \'etale) open neighborhood $U$ of a fixed general point $b\in B$ such that $\sigma_1(t)$ and $\sigma_2(t)$ are
$\Gamma$-equivalent over $\BQ$ for all $t$ in $U$. \magenta{We shall apply Proposition \ref{COMVPROPVSA} to deduce a contradiction. Shrinking $B$ is necessary, the first item above \eqref{COMVE600} is clearly satisfied. Since $K_{X_b}(-n) = \CO_{X_b}$ and $T_X(1)$ is globally generated on $X_b$ for $X/B$ being versal \cite{V2}, it remains to verify the crucial \eqref{COMVE600}. Let $\lambda\in U\times \Aut(P)$ be an automorphism of $Y/U$ such that $\lambda_b = \text{id}$ and $\lambda(\sigma_i(t)) \equiv p_i = \sigma_i(b)$ for $i=1,2$. We shall prove that either \eqref{COMVE600} holds or the line joining $p_1$ and $p_2$ meets $X_b$ only at these two points. However, using a standard argument, we shall further show that the latter case does not occur. This concludes the proof.}

\subsection{Versal deformation of the Fermat hypersurface}

Let us choose $X\subset Y = B\times P$ to be the family of hypersurfaces of degree $d$ in $P = \PP^{n+1}$ given by
\begin{equation}\label{COMVE484}
F(x_0,x_1,...,x_n, x_{n+1}, t_f) = x_0^d + x_1^d + ... + x_{n+1}^d +
\sum_{f\in J_d} t_f f = 0,
\end{equation}
where $(x_0,x_1,...,x_{n+1})$ are the homogeneous coordinates of $\PP^{n+1}$, $J_d$ is the set of monomials in $x_i$ given by
\begin{equation}\label{COMVE486}
\begin{aligned}
J_d &= \big\{
x_0^{m_0} x_1^{m_1} ... x_{n+1}^{m_{n+1}}:
m_0, m_1, ..., m_{n+1}\in \BN,\\
&\hspace{105pt} m_0 + m_1 + ... + m_{n+1} = d\text{ and}\\
&\hspace{105pt} m_0,m_1,...,m_{n+1}\le d-2
\big\}
\end{aligned}
\end{equation}
and $(t_f)$ are the coordinates of the affine space $B =
\Span_\BC J_d \cong \A^N$ for
\begin{equation}\label{COMV488}
N = h^0(\CO_P(d)) - h^0(T_P) - 1 = \binom{d+n+1}{n+1} - (n+2)^2.
\end{equation}
We may regard $X/B$ as a versal deformation of the Fermat hypersurface.

At a general point $b\in B$, $X/B$ is obviously versal, i.e., the Kodaira-Spencer map
\begin{equation}\label{COMVE134}
\begin{tikzcd}
T_{B,b} \arrow{r}{\sim} & H^0(\N_{X_b}) / \eta(H^0(X_b, T_P))\ar[hook]{d}\\
& H^1(T_{X_b})
\end{tikzcd}
\end{equation}
is an isomorphism, where $\eta$ is the map in
\begin{equation}\label{COMVE464}
\begin{tikzcd}
0 \ar{r} & T_X \ar{r} & T_Y\otimes \CO_X \ar{r}{\eta} & \N_X \ar{r} & 0.
\end{tikzcd}
\end{equation}
More explicitly, \eqref{COMVE134} is equivalent to saying
\begin{equation}\label{COMVE485}
\Span\left\{x_i \frac{\partial F}{\partial x_j}\right\} \oplus
\Span J_d = H^0(\N_{X_b}) = H^0(X_b, \CO(d))
\end{equation}
for $b\in B$ general.

Let $\E = \CO_P(1)^{\oplus n+2}$ be the Euler bundle on $P$. Then
\begin{equation}\label{COMVE455}
H^0(T_P) \cong
\frac{H^0(\E)}{(\alpha)} =
\Span\left\{
x_i\frac{\partial}{\partial x_j}
\right\}/(\alpha)
\end{equation}
by the Euler sequence
\begin{equation}\label{COMVE418}
\begin{tikzcd}
0 \arrow{r} & \CO_P \arrow{r} & \CO_P(1)^{\oplus (\red{n+2})}\arrow[equal]{d} \arrow{r} & T_P \arrow{r} & 0\\
&& \E
\end{tikzcd}
\end{equation}
and
\begin{equation}\label{COMVE459}
\eta\left(\red{x_i}\frac{\partial}{\partial x_{j}}\right)
= \red{x_i}\frac{\partial F}{\partial x_j} \text{ and }
\eta\left(\frac{\partial}{\partial t_f}\right)
= \frac{\partial F}{\partial t_f} = f
\end{equation}
for $j = 0,1,2,...,n+1$ and $f\in J_d$, where
\begin{equation}\label{COMVE009}
\alpha = \sum_{i=0}^{n+1} x_i \frac{\partial}{\partial x_i}.
\end{equation}

We are going to show that no two distinct points on a very general fiber $X_b$ of $X/B$ are $\Gamma$-equivalent over $\BQ$ when $d=2n+2\ge 6$. First a definition.

\begin{defn}\label{COMVDEF2SCHEME}
Let $Z$ be a $0$-dimensional scheme of length $2$ in $P = \PP^{n+1}$ with homogeneous coordinates
$(x_0,x_1,...,x_{n+1})$.
We call $Z$ {\em generic} with respect to the homogeneous coordinates $(x_i)$ if
\begin{equation}\label{COMVE523}
H^0(\CO_Z(1)) = \Span \{x_j: j\ne i\} \text{ for every } i=0,1,\ldots,n+1.
\end{equation}
Otherwise, we call $Z$ {\em special} with respect to $(x_i)$. We call $Z$ {\em very special} with respect to $(x_i)$ if
\begin{equation}\label{COMVE058}
\# \{x_i: x_i\in H^0(I_Z(1))\} = n = h^0(\CO_P(1)) - 2
\end{equation}
where $I_Z$ is the ideal sheaf of $Z$ in $P$.
\magenta{Geometrically, $Z = \{p\ne q\}$ being special means that $Z$ is projected to one point under the projection sending $(x_0,x_1,...,x_{n+1})$ to $(x_0,x_1,...,\widehat{x}_i,...,x_{n+1})$ for some $i$ and being very special means that $Z$ is contained in a line cut out by $n$ coordinate hyperplanes.}
\end{defn}

\begin{rem}\label{COMVREMDEF2SCHEME}
Clearly, these notions depend on the choice of homogeneous coordinates of $P$. More generally, we can define
these terms with respect to a basis of $H^0(L)$ for an arbitrary very ample line bundle $L$ on $P$.

When the choice of homogeneous coordinates is clear, we simply say $Z$ is generic (resp. special/very special).

Obviously, being very special implies being special.

There always exist $i\ne j$ such that
$x_i$ and $x_j$ span $H^0(\CO_Z(1))$ since $\CO_P(1)$ is very ample.
Without loss of generality, we usually make the assumption that $(i,j) = (0,1)$, i.e.,
\begin{equation}\label{COMVE522}
H^0(\CO_Z(1)) = \Span \{x_0, x_1\}.
\end{equation}
Under the hypothesis of \eqref{COMVE522}\magenta{\ and swapping $0,1$ if needed}, $Z$ is special if and only if
\begin{equation}\label{COMVE022}
\Span\{x_0, x_1\} = H^0(\CO_Z(1)) \supsetneq \Span \{x_1, x_2, ..., x_{n+1}\}.
\end{equation}
Furthermore, by re-arranging $x_2, ..., x_{n+1}$, we may assume that there exists
$1\le a\le n+1$ such that
\begin{equation}\label{COMVE059}
x_1,..., x_a\not\in H^0(I_Z(1))
\text{ and }
x_{a+1}, ..., x_{n+1}\in H^0(I_Z(1)).
\end{equation}
Of course, $Z$ is very special if and only if $a = 1$.
\end{rem}

We are considering two cases: with respect to $(x_j)$\magenta{, for a very general point $b\in B$},
\begin{description}
\item[Generic case]
$Z = \{\sigma_1(b), \sigma_2(b)\} = \{p_1, p_2\}$ is generic or
\item[Special case]
$Z = \{\sigma_1(b), \sigma_2(b)\}$ is special.
\end{description}

\subsection{A basis for $W_{X,b}$}

For convenience, we identify the tangent space $T_{B,b}$ with $\Span J_d$. Then
$\eta(f) = f$ for all $f\in \Span J_d$.

We start the verification of \eqref{COMVE600} by studying the space $W_{X,b}$
defined by \eqref{COMVE461}. It has a basis given by:

\begin{lem}\label{COMVLEMVXB}
Let $P = \PP^{n+1}$ and $X\subset Y = B\times P$ be the family of hypersurfaces in $P$ given by
\eqref{COMVE484} over $B = \Span J_d$ for $d\ge 3$. Then
\begin{equation}\label{COMVE510}
\begin{aligned}
\W_{X,b} &= \big\{ \omega \in H^0(X_b, \E(1)): \eta(\omega)\in \Span J_{d+1}
\big\}
\\
&=
\Span\big\{
\omega_{ijk}: 0\le i,j,k\le n+1,\ i\le j \text{ and } i,j\ne k
\big\}
\end{aligned}
\end{equation}
has dimension
\begin{equation}\label{COMVE025}
\dim \W_{X,b} = (n+2) \binom{n+2}2
\end{equation}
for $b=(t_f)$ in an open neighborhood of $0$, where
\begin{equation}\label{COMVE514}
\begin{aligned}
\omega_{ijk} &= x_i x_j\frac{\partial}{\partial x_k} \text{ for } i\ne j\ne k \text{ and}\\
\omega_{iik} &= x_i^2 \frac{\partial}{\partial x_k} - \sum_{j\ne i}
c_{ijk}
x_i x_j
\frac{\partial}{\partial x_i}
\text{ for } i\ne k
\end{aligned}
\end{equation}
with
\begin{equation}\label{COMVE026}
c_{ijk} = \frac{d-1}{d!}\left(\frac{\partial^d F}{\partial x_i^{d-2} \partial x_j \partial x_k}\right)
= \begin{cases}
2 d^{-1} t_f & \text{if } i\ne j=k\\
d^{-1} t_f & \text{if } i\ne j\ne k
\end{cases}
\end{equation}
for $f= x_i^{d-2} x_j x_k$.
Here we consider $\eta$ as a map
$H^0(\E(1))\to H^0(\CO(d+1))$ given by \eqref{COMVE459}.
\end{lem}

\begin{proof}
We have
\begin{equation}\label{COMVE513}
\eta\left(x_i x_j\frac{\partial}{\partial x_k}\right)
= x_i x_j \frac{\partial F}{\partial x_k} = d x_i x_j x_k^{d-1}
+ \sum_{f\in J_d} t_f x_i x_j \frac{\partial f}{\partial x_k}.
\end{equation}
It is easy to check that
\begin{equation}\label{COMVE511}
\eta(\omega_{ijk}) = x_i x_j\frac{\partial F}{\partial x_k} \in \Span J_{d+1}
\end{equation}
for $i\ne j\ne k$ and
\begin{equation}\label{COMVE024}
\eta(\omega_{iik}) = x_i^2\frac{\partial F}{\partial x_k}
-  \sum_{j\ne i}\frac{d-1}{d!}
x_i x_j
\left(
\frac{\partial^d F}{\partial x_i^{d-2} \partial x_j \partial x_k}
\right)
\frac{\partial F}{\partial x_i} \in \Span J_{d+1}
\end{equation}
for $i\ne k$. Hence $\omega_{ijk}\in \W_{X,b}$ for all $i,j\ne k$.

To show that $\{\omega_{ijk}: i\le j \text{ and } i,j\ne k\}$
forms a basis of $\W_{X,b}$ in an open neighborhood of $0$, it suffices to verify this for $b=0$: clearly,
\begin{equation}\label{COMVE512}
\Big\{\omega_{ijk}\Big|_{b=0}: i\le j \text{ and } i,j\ne k\Big\} = \left\{x_i x_j\frac{\partial}{\partial x_k}: i\le j \text{ and } i,j\ne k\right\}
\end{equation}
is a basis of $\W_{X,0}$. Therefore, \eqref{COMVE510} and \eqref{COMVE025} follow.
\end{proof}

Clearly, $W_{X,b}$ is the image of $\W_{X,b}$ under the map
\begin{equation}\label{COMVE302}
\begin{tikzcd}
H^0(X_b, \E(1))\ar[two heads]{r} & H^0(X_b, T_P(1)).
\end{tikzcd}
\end{equation}
More precisely, let \magenta{$\widetilde{W}_{X,b}$} be the lift of $W_{X,b}$ in
$H^0(X_b, \E(1))$. Then
\begin{equation}\label{COMVE303}
\magenta{\widetilde{W}_{X,b}} = \W_{X,b} \oplus \alpha \otimes H^0(\CO(1))
\end{equation}
where $\W_{X,b} \cap \alpha \otimes H^0(\CO(1)) = 0$ because
\begin{equation}\label{COMVE304}
\Span J_{d+1} \cap \eta\left(\alpha \otimes H^0(\CO(1))\right)
= \Span J_{d+1} \cap F\otimes H^0(\CO(1)) = 0.
\end{equation}

\subsection{A key observation on $L_\lambda$}

We observe the following:

\begin{lem}\label{COMVLEMLLAMBDA}
Let $P = \PP^{n+1}$ and $X\subset Y = B\times P$ be the family of hypersurfaces in $P$ given by \eqref{COMVE484} over $B = \Span J_d$.
For $b\in B$, a $0$-dimensional subscheme $Z\subset X_b$ of length $2$ and
$L_\lambda\in\Hom(\pi_B^* T_{B,b}, T_P\otimes \CO_Z)$,
if
\begin{equation}\label{COMVE480}
L_\lambda\left(f\right) \ne 0
\text{ for some } f\in H^0(I_Z(1)) \otimes \Span J_{d-1} \subset \Span J_d,
\end{equation}
then \eqref{COMVE600} holds.
\end{lem}

\begin{proof}
Obviously, \magenta{under the hypothesis, }\eqref{COMVE480} holds
for some $f=lg$ with $l\in H^0(I_Z(1))$ and $g\in J_{d-1}$.

For each point $p\in \supp(Z)$, we choose $l_p\in H^0(\CO_P(1))$ such that $l_p(p) = 0$ and $l_p\not\in H^0(I_Z(1))$ and let
\begin{equation}\label{COMVE489}
\tau_p = l_p \otimes f - l \otimes l_p g
\in H^0(\CO_{X_b}(1)) \otimes T_{B,b}.
\end{equation}
Then $\eta(\tau_p) = 0$ so $L_\lambda(\tau_p) \in W_{X,b,Z,\lambda}$. Clearly,
\begin{equation}\label{COMVE503}
L_\lambda(\tau_p) = l_p L_\lambda\left(f\right) - l L_\lambda\left(l_p g\right) = l_p L_\lambda\left(f\right)
\end{equation}
since $l\in H^0(I_Z(1))$. Then by our choice of $l_p$,
$L_\lambda(\tau_p)$ vanishes at $p$.

If $L_\lambda(\tau_p)\ne 0$\magenta{\ for some $p \in \supp(Z)$}, then \eqref{COMVE600} follows. Otherwise,
\begin{equation}\label{COMVE504}
l_p L_\lambda\left(f\right) = 0\magenta{\text{ for all } p \in \supp(Z)}.
\end{equation}
Since $l_p\not\in H^0(I_Z(1))$,
\eqref{COMVE504} implies that $L_\lambda\left(f\right)$ vanishes at all $p\in \supp(Z)$.

If $Z$ consists of two distinct points, then
we must have
\begin{equation}\label{COMVE505}
L_\lambda\left(f\right) = 0,
\end{equation}
which contradicts our hypothesis \eqref{COMVE480}.
\magenta{Although we only need the lemma for this case, we will prove it for all $Z$ for the sake of completeness.}

If $Z$ is supported at a single point $p$, then $L_\lambda\left(f\right)$ vanishes at $p$. Applying the same argument to
$\tau_q = l_q \otimes f - l \otimes l_q g$ for some $l_q\in H^0(\CO_P(1))$ satisfying $l_q(p) \ne 0$, we have
\begin{equation}\label{COMVE301}
L_\lambda(\tau_q) = l_q L_\lambda\left(f\right) - l L_\lambda\left(l_q g\right) = l_q L_\lambda\left(f\right) \in W_{X,b,Z,\lambda}
\end{equation}
vanishing at $p$. Again, we have either \eqref{COMVE600} or \eqref{COMVE505} since $l_q(p)\ne 0$.
\end{proof}

Let us assume that \eqref{COMVE505} holds for all $f\in H^0(I_Z(1)) \otimes \Span J_{d-1}$.
Otherwise, we are done by the above lemma.
Then $L_\lambda: T_{B,b}\to H^0(Z, T_P)$ factors through
\begin{equation}\label{COMVE488}
\frac{\Span J_d}{H^0(I_Z(1)) \otimes \Span J_{d-1}}
\end{equation}
and it can be regarded as a map
\begin{equation}\label{COMVE509}
\begin{tikzcd}
\displaystyle{
	\frac{\Span J_d}{H^0(I_Z(1)) \otimes \Span J_{d-1}}
} \ar{r}{L_\lambda} & H^0(Z, T_P).
\end{tikzcd}
\end{equation}

\subsection{The space $H^0(I_Z(1)) \otimes \Span J_{d-1}$}

Let us figure out the space \eqref{COMVE488}.
Obviously,
\begin{equation}\label{COMVE450}
H^0(I_Z(1)) \otimes \Span J_{d-1} \subset \Span J_d
\cap H^0(I_Z(1)) \otimes H^0(\CO_P(d-1)).
\end{equation}
Furthermore, since $H^0(I_Z(1)) \otimes H^0(\CO_P(d-1))$ is the kernel of the map
\begin{equation}\label{COMVE527}
\begin{tikzcd}
H^0(\CO_P(d)) \ar[two heads]{r}{\xi} \ar[equal]{d} & \Sym^d H^0(\CO_Z(1))\\
\Sym^d H^0(\CO_P(1)) \ar{ur}
\end{tikzcd}
\end{equation}
we may write \eqref{COMVE450} as
\begin{equation}\label{COMVE023}
H^0(I_Z(1)) \otimes \Span J_{d-1} \subset \Span J_d\cap \ker(\xi).
\end{equation}
Actually, this inclusion is an equality for $Z$ generic:

\begin{lem}\label{COMVLEMJD}
Let $P = \PP^{n+1}$, $J_d$ be defined in \eqref{COMVE486}
and $Z$ be a $0$-dimensional subscheme of $P$ of length $2$.
If \magenta{$d\ge 5$} and $Z$ is generic with respect to $(x_i)$,
then
\begin{equation}\label{COMVE519}
\begin{aligned}
H^0(I_Z(1)) \otimes \Span J_{d-1} &= \Span J_{d} \cap
H^0(I_Z(1)) \otimes H^0(\CO_P(d-1))\\
&= \Span J_{d} \cap \ker(\xi).
\end{aligned}
\end{equation}
Or equivalently, $H^0(I_Z(1)) \otimes \Span J_{d-1}$ is the kernel of the map
\begin{equation}\label{COMVE506}
\begin{tikzcd}
\Span J_d \ar{r}[above]{\xi} & \Sym^d H^0(\CO_Z(1)).
\end{tikzcd}
\end{equation}
In addition,
\begin{equation}\label{COMVE028}
\begin{tikzcd}
\displaystyle{\frac{\Span J_d}{H^0(I_Z(1)) \otimes \Span J_{d-1}}} \ar{r}[above]{\xi}[below]{\sim} & \Sym^d H^0(\CO_Z(1))
\end{tikzcd}
\end{equation}
is an isomorphism.
\end{lem}

\begin{proof}
To prove \eqref{COMVE519},
it suffices to find a subset $S\subset J_{d}$ such that
\begin{equation}\label{COMVE515}
\Span J_{d} = H^0(I_Z(1)) \otimes \Span J_{d-1} +
\Span(S)
\end{equation}
and
\begin{equation}\label{COMVE521}
H^0(I_Z(1)) \otimes H^0(\CO_P(d-1))\cap
\Span (S) = 0.
\end{equation}

Let us assume \eqref{COMVE522}. By \eqref{COMVE523}, $H^0(\CO_Z(1)) = \Span \{x_1,x_2,...,x_{n+1}\}$ and hence there exists $i\ne 0,1$ such that
\begin{equation}\label{COMVE524}
H^0(\CO_Z(1)) = \Span \{x_1, x_i\}.
\end{equation}
Similarly, we have $H^0(\CO_Z(1)) = \Span \{x_0,x_2,...,x_{n+1}\}$ and hence there exists $j\ne 0,1$ such that
\begin{equation}\label{COMVE525}
H^0(\CO_Z(1)) = \Span \{x_0, x_j\}.
\end{equation}
Then we let
\begin{equation}\label{COMVE518}
\begin{aligned}
S &= \Big\{x_0^{d-3} x_i^3, x_0^{d-3} x_i^2 x_1, x_0^{d-3} x_i x_1^2,\\
&\hspace{24pt} x_0^{d-3} x_1^3, x_0^{d-4} x_1^4, ..., x_0^3 x_1^{d-3},\\
&\hspace{24pt} x_0^2 x_1^{d-3} x_j, x_0 x_1^{d-3} x_j^2, x_1^{d-3} x_j^3 \Big\}.
\end{aligned}
\end{equation}
By \eqref{COMVE522}, \eqref{COMVE524} and \eqref{COMVE525},
for every $k$,
\begin{equation}\label{COMVE520}
\begin{aligned}
x_k &\in H^0(I_Z(1)) + \Span \{x_0, x_1\},\\
x_k &\in H^0(I_Z(1)) + \Span \{x_1, x_i\},\text{ and}\\
x_k &\in H^0(I_Z(1)) + \Span \{x_0, x_j\}.
\end{aligned}
\end{equation}
\magenta{To see \eqref{COMVE515}, it suffices to prove that every monomial in $J_d$ lies in the vector space spanned by $H^0(I_Z(1)) \otimes \Span J_{d-1}$ and $S$. For a monomial
$$
x_0^{a_0} x_1^{a_1} x_2^{a_2} ... x_{n+1}^{a_{n+1}} \in J_d
$$
satisfying $a_0 + a_1 \le d-3$, we choose
$a_k = \max(a_2,a_3,...,a_{n+1})$
and write $x_k = l_1 + l_2$ for some $l_1\in H^0(I_Z(1))$ and $l_2\in \Span\{x_0, x_1\}$. Then
$$
\begin{aligned}
&\quad x_0^{a_0} x_1^{a_1} x_2^{a_2} ... x_k^{a_k}... x_{n+1}^{a_{n+1}} =
x_0^{a_0} x_1^{a_1} x_2^{a_2} ... x_k^{a_k-1} (l_1 + l_2) ... x_{n+1}^{a_{n+1}}\\
&\in H^0(I_Z(1)) \otimes \Span J_{d-1} + \Span \Big\{
x_0^{a_0+1} x_1^{a_1} x_2^{a_2} ... x_k^{a_k-1} ... x_{n+1}^{a_{n+1}},
\\
&\hspace{168pt} x_0^{a_0} x_1^{a_1+1} x_2^{a_2} ... x_k^{a_k-1} ... x_{n+1}^{a_{n+1}}
\Big\}
\end{aligned}
$$
Repeating this process, we see that
$$
\begin{aligned}
J_d
\subset H^0(I_Z(1)) \otimes \Span J_{d-1} +
\Span \Big\{
x_0^{a_0} x_1^{a_1} \prod_{k=2}^{n+1}
x_k^{a_k} \in J_d: a_0 + a_1 \ge d-2
\Big\}
\end{aligned}
$$
So it remains to verify that monomials of type
$$
x_0^{a_0} x_1^{a_1} x_2^{a_2} ... x_{n+1}^{a_{n+1}}\in J_d
$$
with $a_0 + a_1 \ge d-2$ lie in $H^0(I_Z(1)) \otimes \Span J_{d-1} + \Span(S)$:

\begin{enumerate}
\item For a monomial $x_0^{d-3} x_a x_b x_c$ with $b,c\ne 0$, we can write
$$
\begin{aligned}
x_a &= l_1 + l_2,\hspace{12pt} \text{for } l_1\in H^0(I_Z(1))
\text{ and } l_2\in \Span \{x_1, x_i\}\\
x_b &= l_3 + l_4,\hspace{12pt} \text{for } l_3\in H^0(I_Z(1))
\text{ and } l_4\in \Span \{x_1, x_i\}\\
x_c &= l_5 + l_6,\hspace{12pt} \text{for } l_5\in H^0(I_Z(1))
\text{ and } l_6\in \Span \{x_1, x_i\}
\end{aligned}
$$
Then
$$
\begin{aligned}
&\quad x_0^{d-3} x_a x_b x_c = x_0^{d-3} (l_1 + l_2) x_b x_c = l_1 x_0^{d-3} x_b x_c + l_2 x_0^{d-3} x_b x_c\\
&\in H^0(I_Z(1)) \otimes \Span J_{d-1} + \Span \{
x_0^{d-3} x_1 x_b x_c, x_0^{d-3} x_i x_b x_c \}\\
& \subset H^0(I_Z(1)) \otimes \Span J_{d-1} + \Span \{
x_0^{d-3} x_1 (l_3 + l_4) x_c, x_0^{d-3} x_i (l_3 + l_4) x_c \}\\
& \subset H^0(I_Z(1)) \otimes \Span J_{d-1} + \Span \{
x_0^{d-3} x_1^2 x_c, x_0^{d-3} x_1 x_i x_c, x_0^{d-3} x_i^2 x_c\}\\
& \subset H^0(I_Z(1)) \otimes \Span J_{d-1} + \Span \Big\{
x_0^{d-3} x_1^2 (l_5 + l_6), x_0^{d-3} x_1 x_i (l_5 + l_6),\\
&\hspace{168pt} x_0^{d-3} x_i^2 (l_5 + l_6)\Big\}\\
&\subset H^0(I_Z(1)) \otimes \Span J_{d-1} + \Span \{
x_0^{d-3} x_1^3, x_0^{d-3} x_1^2 x_i, x_0^{d-3} x_1 x_i^2, x_0^{d-3} x_i^3 \}\\
&\subset H^0(I_Z(1)) \otimes \Span J_{d-1} + \Span(S).
\end{aligned}
$$
\item For a monomial $x_1^{d-3} x_a x_b x_c$ with $b,c\ne 1$, we can write $$
\begin{aligned}
x_a &= l_1 + l_2,\hspace{12pt} \text{for } l_1\in H^0(I_Z(1))
\text{ and } l_2\in \Span \{x_0, x_j\}\\
x_b &= l_3 + l_4,\hspace{12pt} \text{for } l_3\in H^0(I_Z(1))
\text{ and } l_4\in \Span \{x_0, x_j\}\\
x_c &= l_5 + l_6,\hspace{12pt} \text{for } l_5\in H^0(I_Z(1))
\text{ and } l_6\in \Span \{x_0, x_j\}
\end{aligned}
$$
Then by the same argument as in (1), we obtain
$$
\begin{aligned}
x_1^{d-3} x_a x_b x_c
&\in H^0(I_Z(1)) \otimes \Span J_{d-1} + \Span \Big\{
x_1^{d-3} x_0^3, x_1^{d-3} x_0^2 x_j,\\
&\hspace{168pt} x_1^{d-3} x_0 x_j^2, x_1^{d-3} x_j^3 \Big\}\\
&\subset H^0(I_Z(1)) \otimes \Span J_{d-1} + \Span(S).
\end{aligned}
$$
\item For a monomial
$$
x_0^{a_0} x_1^{a_1} x_2^{a_2} ... x_{n+1}^{a_{n+1}}\in J_d
$$
with $a_0, a_1\ge 2$, by substituting every $x_k$ for $k\ge 2$ with
$x_k = l_1 + l_2$ for some $l_1\in H^0(I_Z(1))$ and $l_2\in \Span \{x_0, x_1\}$, we obtain
$$
\begin{aligned}
x_0^{a_0} x_1^{a_1} x_2^{a_2} ... x_{n+1}^{a_{n+1}} &\in H^0(I_Z(1)) \otimes \Span J_{d-1} + \Span \{
x_0^{r} x_1^{d-r}: 2\le r\le d-2 \}\\
&\subset H^0(I_Z(1)) \otimes \Span J_{d-1} + \Span(S)
\end{aligned}
$$
where we have shown that
$$
x_0^{d-2}x_1^2, x_0^2 x_1^{d-2} \in H^0(I_Z(1)) \otimes \Span J_{d-1} + \Span(S)
$$
in (1) and (2).
\end{enumerate}
}

To see \eqref{COMVE521}, we just have to show that $\ker(\xi)\cap \Span(S) = 0$, which is equivalent to
\begin{equation}\label{COMVE474}
\xi\left(\Span(S)\right) = \Sym^{d} H^0(\CO_Z(1))
\end{equation}
since $|S| = \dim \Sym^{d} H^0(\CO_Z(1)) = d+1$. Again it is easy to see from \eqref{COMVE522}, \eqref{COMVE524} and \eqref{COMVE525} that
\begin{equation}\label{COMVE469}
\begin{aligned}
\xi\left(\Span(S)\right) &= \xi \left(
\Span\{ x_0^{d-k} x_1^{k}: k = 0,1,\ldots, d\}
\right)\\
&= \Sym^{d} H^0(\CO_Z(1)).
\end{aligned}
\end{equation}
This also proves that \eqref{COMVE028} is an isomorphism.
\end{proof}

When $Z$ is special, $H^0(I_Z(1)) \otimes \Span J_{d-1}$ is no longer the kernel of the map \eqref{COMVE506}.
Instead, we have the following result when $Z$ is special but not very special.

\begin{lem}\label{COMVLEMJD3}
Let $P = \PP^{n+1}$, $J_d$ be defined in \eqref{COMVE486} and $Z$ be a $0$-dimensional subscheme of $P$ of length $2$.
Suppose that $d\ge 4$, $Z$ satisfies \eqref{COMVE022}
and $\{x_2, ...,x_{n+1}\}\not\subset H^0(I_Z(1))$. Then
\begin{equation}\label{COMVE034}
\begin{aligned}
\Span J_{d} \cap \ker(\xi) &= H^0(I_Z(1)) \otimes \Span J_{d-1}
\\
&\quad + \Span \Big\{
x_0^{d-2}  x_i (x_j - c_j x_1): i\ge 1, j\ge 2\\
&\hspace{72pt} \text{and }
x_j - c_j x_1\in H^0(I_Z(1))
\Big\}.
\end{aligned}
\end{equation}
\end{lem}

\begin{proof}
We leave the verification of \eqref{COMVE034} to the readers.
\end{proof}

\subsection{Special case}

Let us first prove \eqref{COMVE600} when $Z$ is special for all $b$. Without loss of generality, let us assume that
$Z = \{p_1, p_2\}$ satisfies \eqref{COMVE022} and \eqref{COMVE059} for $b$ general and some $a$.

We claim that $L_\lambda: \pi_B^* T_{B,b}\to T_P\otimes \CO_Z$ factors through
a sub-sheaf $\g_Z$ of $T_P\otimes \CO_Z$, i.e.,
$L_\lambda\in \Hom(\pi_B^* T_{B,b}, \g_Z)$ for the sub-sheaf $\g_Z$ of $T_P\otimes \CO_Z$ generated by the global sections
\begin{equation}\label{COMVE061}
H^0(\g_Z) = \Span\left\{
x_i \frac{\partial}{\partial x_j}: j = 0 \text{ or } 2 \le i,j \le a
\right\}.
\end{equation}
In addition, if $x_0$ vanishes at one of $p_i$ for $b$ general,
$\g_Z$ is generated by
\begin{equation}\label{COMVE089}
H^0(\g_Z) = \Span\left\{
x_i \frac{\partial}{\partial x_j}: i = j = 0 \text{ or } 2 \le i,j \le a
\right\}.
\end{equation}

To see this, we notice that $(1,0,...,0)\not\in X_b$ for all $b$. So
\begin{equation}\label{COMVE314}
x_1(p_i)\ne 0 \text{ for } i=1,2.
\end{equation}
Otherwise, if $x_1 = 0$ at some $p\in Z$, then
$x_2 = x_3 = ... = x_{n+1} = 0$ at $p$ by \eqref{COMVE022} and $p=(1,0,...,0)$.

Thus, we may choose $\lambda$ to be given by
\begin{equation}\label{COMVE063}
\lambda \begin{bmatrix}
x_0\\
x_1\\
\vdots\\
x_{n+1}
\end{bmatrix}
=
\begin{bmatrix}
	g_1(t) & g_2(t)\\
	& 1\\
	&& A(t)\\
	&&& I_{n-a+1}
	\end{bmatrix}
\begin{bmatrix}
x_0\\
x_1\\
\vdots\\
x_{n+1}
\end{bmatrix}
\end{equation}
locally at $b$, for some $g_1(t)$, $g_2(t)$ and $A(t)$ satisfying $g_1(b) = 1$, $g_2(b) = 0$ and $A(b) = I_{a-1}$,
where $I_m$ is the $m\times m$ identity matrix.
Then by \eqref{COMVE090},
\begin{equation}\label{COMVE062}
L_\lambda\left(
\tau
\right) \in H^0(\g_Z)
\end{equation}
for all $\tau\in T_{B,b}$ with $\g_Z$ generated by \eqref{COMVE061}.

When $x_0$ vanishes at one of $p_i$ for $b$ general, $g_2(t) \equiv 0$ in \eqref{COMVE063} and thus we have \eqref{COMVE089}.
This proves our claim that $L_\lambda$ factors through $\g_Z$ given by
\eqref{COMVE061} or \eqref{COMVE089}.

Let $\Lambda\subset P$ be the line joining $p_1$ and $p_2$. Then
the map $\xi$ in \eqref{COMVE527} is simply the restriction to $\Lambda$
as in
\begin{equation}\label{COMVE010}
\begin{tikzcd}
H^0(\CO_P(m)) \ar{r}{\xi} \ar{dr}[below]{\xi} &
H^0(\CO_\Lambda(m)) \ar[equal]{d}\\
& \Sym^{m} H^0(\CO_Z(1))
\end{tikzcd}
\end{equation}
for $m\in \BN$.
We will use
$\Sym^{m} H^0(\CO_Z(1))$ and $H^0(\CO_\Lambda(m))$ interchangeably under this setting. We also use $\xi$ to denote the induced map
\begin{equation}\label{COMVE666}
\begin{tikzcd}
H^0(\CO_{X_b}(m)) \ar{r}{\xi} &
\displaystyle{\frac{H^0(\CO_\Lambda(m))}{\xi(F)\otimes H^0(\CO_\Lambda(m-d))}}
\end{tikzcd}
\end{equation}
where quotient by $\xi(F)$ is necessary; otherwise it is not well defined as $\xi(F)$ is not zero in
$H^0(\CO_\Lambda(d))$ unless $X_b$ contains the line $\Lambda$.

We further abuse the notation by using $\xi$ for the maps induced by
the restriction $H^0(\E(m))\to H^0(\Lambda, \E(m))$:
\begin{equation}\label{COMVE011}
\begin{tikzcd}[column sep=small]
H^0(\E(m)) \ar[equal]{r} \ar{dd}{\eta}
& H^0(X_b, \E(m))\ar{d} \ar{r}{\xi} & H^0(\Lambda, \E(m)) \ar{d}\\
& H^0(X_b, T_P(m)) \ar{r}{\xi} \ar{d}{\eta}
& H^0(\Lambda, T_P(m))
\ar{d}{\eta}
\\
H^0(\CO(m+d)) \ar{r} & H^0(\CO_{X_b}(m+d)) \ar{r}{\xi} &
\displaystyle{\frac{H^0(\CO_{\Lambda}(m+d))}{\xi(F)\otimes H^0(\CO_{\Lambda}(m))}}
\end{tikzcd}
\end{equation}
for $m\le d-2$, where we also abuse the notation $\eta$ by using it for three different maps, all defined by \eqref{COMVE459}.

Next, let us consider the images of the spaces $\W_{X,b}\subset H^0(X_b, \E(1))$ and $W_{X,b}\subset H^0(X_b, T_P(1))$ under $\xi$, where
$\xi(\W_{X,b})$ and $\xi(W_{X,b})$ are considered as the subspaces
of $H^0(\Lambda, \E(1))$ and
$H^0(\Lambda, T_P(1))$, respectively.

\begin{lem}\label{COMVLEMVXBSPECIAL}
Let $P = \PP^{n+1}$ and $X\subset Y = B\times P$ be the family of hypersurfaces in $P$ given by \eqref{COMVE484} over
$B = \Span J_d$ for $n\ge 2$ and $d\ge 4$.
For $b\in B$ general and \underline{all} $0$-dimensional subschemes $Z\subset X_b$ of length $2$ satisfying \eqref{COMVE022},
\begin{equation}\label{COMVE316}
\xi(\W_{X,b}) \supset \left\{
x_1^2 \frac{\partial}{\partial x_i}
\right\} \cup \left\{
x_0 x_1 \frac{\partial}{\partial x_j}: j\ge 1
\right\}
\end{equation}
and
\begin{equation}\label{COMVE052}
\xi\left(W_{X,b}\right) = H^0(\Lambda, T_P(1))
\end{equation}
if $\{x_2,...,x_{n+1}\}\not\subset H^0(I_Z(1))$ and
\begin{equation}\label{COMVE053}
\begin{aligned}
\xi\left(\W_{X,b}\right)
&= \Span \left\{
x_0 x_1 \frac{\partial}{\partial x_k} : k\ne 0,1
\right\}\\
&\hspace{36pt}
\cup \left\{
x_0^2 \frac{\partial}{\partial x_k} - c_{01k} x_0 x_1 \frac{\partial}{\partial x_0}: k\ne 0
\right\}\\
&\hspace{36pt}
\cup \left\{
x_1^2 \frac{\partial}{\partial x_k} - c_{10k} x_0 x_1 \frac{\partial}{\partial x_1}: k\ne 1
\right\}\subset H^0(\Lambda, \E(1))
\end{aligned}
\end{equation}
if $\{x_2,...,x_{n+1}\}\subset H^0(I_Z(1))$,
where $\Lambda\subset P$ is the line cutting out $Z$ on $X_b$, $\xi$ is the map defined in \eqref{COMVE011} and
$c_{ijk}$ are the numbers given by \eqref{COMVE026}.
\end{lem}

\begin{proof}
Let us first deal with the case that $\{x_2,...,x_{n+1}\}\not\subset H^0(I_Z(1))$, i.e., $Z$ is special but not very special. Note that under the hypothesis of \eqref{COMVE022}, all $x_2,...,x_{n+1}$ are multiples of $x_1$ in $H^0(\CO_\Lambda(1))$.

We write $u_1 \equiv u_2$ if $\xi(u_1 - u_2)\in \xi(\W_{X,b})$. Of course, $\omega_{ijk} \equiv 0$
for $\omega_{ijk}$ given by \eqref{COMVE514}. Under this notation,
\eqref{COMVE316} is equivalent to
\begin{equation}\label{COMVE315}
\begin{aligned}
x_1^2 \frac{\partial}{\partial x_0}
&\equiv x_1^2 \frac{\partial}{\partial x_1}
\equiv x_1^2 \frac{\partial}{\partial x_2}
\equiv ... \equiv x_1^2 \frac{\partial}{\partial x_{n+1}}
\\
&\equiv x_0 x_1 \frac{\partial}{\partial x_1}
\equiv x_0 x_1 \frac{\partial}{\partial x_2}
\equiv ... \equiv x_0 x_1 \frac{\partial}{\partial x_{n+1}} \equiv 0.
\end{aligned}
\end{equation}

Without loss of generality, let us assume that $x_2\not\in H^0(I_Z(1))$. Then $x_2 = a x_1$ in $H^0(\CO_\Lambda(1))$ for some $a\ne 0$. Therefore,
\begin{equation}\label{COMVE057}
\begin{aligned}
\omega_{01k}\equiv \omega_{12k} \equiv 0
&\Rightarrow
x_0x_1 \frac{\partial}{\partial x_k} \equiv x_1x_2 \frac{\partial}{\partial x_k}
\equiv 0
\\
&\Rightarrow x_0x_1 \frac{\partial}{\partial x_k} \equiv x_1^2 \frac{\partial}{\partial x_k}
\equiv 0
\end{aligned}
\end{equation}
for $k\ge 3$ and
\begin{equation}\label{COMVE055}
\begin{aligned}
\omega_{120} \equiv \omega_{201}
\equiv \omega_{012}
\equiv 0
&\Rightarrow x_1x_2 \frac{\partial}{\partial x_0} \equiv x_2 x_0 \frac{\partial}{\partial x_1} \equiv
x_0 x_1 \frac{\partial}{\partial x_2} \equiv 0
\\
&\Rightarrow
x_1^2 \frac{\partial}{\partial x_0} \equiv x_0 x_1 \frac{\partial}{\partial x_1} \equiv
x_0 x_1 \frac{\partial}{\partial x_2} \equiv 0.
\end{aligned}
\end{equation}

We claim that \eqref{COMVE057} holds for all $k\ge 1$, i.e.,
\begin{equation}\label{COMVE056}
\begin{aligned}
x_0x_1\frac{\partial}{\partial x_k} \equiv x_1^2 \frac{\partial}{\partial x_k} &\equiv 0 \text{ for all } k\ge 1
\text{ or equivalently}
\\
x_i x_j \frac{\partial}{\partial x_k} &\equiv 0
\text{ for all } j, k\ge 1.
\end{aligned}
\end{equation}

If $\{x_3, ..., x_{n+1}\} \not\subset H^0(I_Z(1))$, say $x_3\not\in H^0(I_Z(1))$, then
\begin{equation}\label{COMVE065}
\begin{aligned}
\omega_{231} \equiv \omega_{132} \equiv 0
&\Rightarrow
x_2x_3 \frac{\partial}{\partial x_1}\equiv x_1x_3
\frac{\partial}{\partial x_2}
\equiv 0
\\
&\Rightarrow
x_1^2 \frac{\partial}{\partial x_1}\equiv x_1^2
\frac{\partial}{\partial x_2}
\equiv 0
\end{aligned}
\end{equation}
and together with \eqref{COMVE057} and \eqref{COMVE055},
we see that \eqref{COMVE056} follows.

Otherwise, $\{x_3, ..., x_{n+1}\} \subset H^0(I_Z(1))$. Then by
\begin{equation}\label{COMVE066}
\begin{aligned}
\omega_{113}\equiv 0\Rightarrow
x_1^2 \frac{\partial}{\partial x_3}
- (c_{103} x_0 + c_{123} x_2) x_1 \frac{\partial}{\partial x_1}
&\equiv 0\\
\red{\Rightarrow}x_1^2 \frac{\partial}{\partial x_3}\equiv x_0x_1 \frac{\partial}{\partial x_1} &\equiv 0
\end{aligned}
\end{equation}
we conclude that
\begin{equation}\label{COMVE068}
x_1 x_2\frac{\partial}{\partial x_1}\equiv 0
\Rightarrow x_1^2\frac{\partial}{\partial x_1}
\equiv 0
\end{equation}
as long as $c_{123}\ne 0$, which is obvious for $b\in B$ general.
Similarly, by considering $\omega_{223}$, we obtain
\begin{equation}\label{COMVE070}
x_1^2\frac{\partial}{\partial x_2}
\equiv 0.
\end{equation}
This concludes the proof of \eqref{COMVE056},
which, combined with \eqref{COMVE055}, yields \eqref{COMVE315} and hence \eqref{COMVE316}.

Next, let us prove \eqref{COMVE052}. Note that by \eqref{COMVE303}, we have the diagram
\begin{equation}\label{COMVE317}
\begin{tikzcd}
\W_{X,b} \ar{r}{\xi} \ar[two heads]{d}& H^0(\Lambda, \E(1))\ar[two heads]{d}\\
W_{X,b} \ar{r}{\xi} & H^0(\Lambda, T_P(1))
\end{tikzcd}
\end{equation}
and hence
\begin{equation}\label{COMVE318}
\xi(W_{X,b}) = \frac{\xi(\W_{X,b})}{
\alpha \otimes H^0(\CO_\Lambda(1))}
\end{equation}
for $\alpha$ given by \eqref{COMVE009}.

Let us write $u_1 \equiv u_2\ (\text{mod } \alpha)$
if $u_1 - u_2 \in \xi(W_{X,b})$. Then \eqref{COMVE052} is equivalent to
\begin{equation}\label{COMVE319}
x_i x_j \frac{\partial}{\partial x_k}
\equiv 0\ (\text{mod } \alpha)
\end{equation}
for all $i,j,k$. Since $H^0(\CO_\Lambda(1)) = \Span\{x_0,x_1\}$, it is enough to prove \eqref{COMVE319} for $0\le i,j\le 1$.

Obviously,
\begin{equation}\label{COMVE072}
x_i\alpha \equiv 0\ (\text{mod } \alpha) \Rightarrow x_i x_0 \frac{\partial}{\partial x_0}
\equiv - x_i \sum_{j=1}^{n+1} x_j \frac{\partial}{\partial x_j}\ (\text{mod } \alpha)
\end{equation}
for all $i$. Combining \eqref{COMVE055},
\eqref{COMVE056} and \eqref{COMVE072},
we obtain
\begin{equation}\label{COMVE075}
x_0^2 \frac{\partial}{\partial x_0}
\equiv
x_0x_1 \frac{\partial}{\partial x_0}
\equiv
x_1^2 \frac{\partial}{\partial x_0}
\equiv 0\ (\text{mod } \alpha) \Rightarrow x_i x_j \frac{\partial}{\partial x_0} \equiv 0\ (\text{mod } \alpha)
\end{equation}
for all $i,j$.

Finally, by \eqref{COMVE075},
\begin{equation}\label{COMVE077}
\omega_{00k} \equiv 0
\Rightarrow
x_0^2 \frac{\partial}{\partial x_k}
- \sum_{j=1}^{n+1} c_{0jk} x_0x_j \frac{\partial}{\partial x_0}
\equiv 0\Rightarrow
x_0^2 \frac{\partial}{\partial x_k} \equiv 0
\ (\text{mod } \alpha)
\end{equation}
for all $k\ge 1$. Combining \eqref{COMVE056}, \eqref{COMVE075}
and \eqref{COMVE077}, we conclude \eqref{COMVE052}.

When $\{x_2,...,x_{n+1}\}\subset H^0(I_Z(1))$, i.e., $Z$ is very special, \eqref{COMVE053} follows directly from the fact that
$\xi(\W_{X,b}) = \Span \left\{\xi(\omega_{ijk})\right\}$.
\end{proof}

We want to call attention to the subtle difference and relation between $\xi(\W_{X,b})$ and $\xi(W_{X,b})$
in the above lemma and also Lemma \ref{COMVLEMVXBGENERIC} below. By \eqref{COMVE317}, $\xi(W_{X,b})$ is the image of $\xi(\W_{X,b})$ under
$H^0(\Lambda, \E(1))\twoheadrightarrow H^0(\Lambda, T_P(1))$.
However, $\xi(\W_{X,b})$ is not necessarily the lift of
$\xi(W_{X,b})$ in $H^0(\Lambda, \E(1))$. In particular,
when $Z$ is special but not very special,
we have \eqref{COMVE052} but it is easy to check that
$\xi(\W_{X,b})\ne H^0(\Lambda, \E(1))$.

Let us go back to the proof of \eqref{COMVE600} for $Z$ special.
Since $x_0$ and $x_1$ span $H^0(\CO_Z(1))$, we can choose $p\in Z$ such that $x_0\ne 0$ at $p$.
To prove \eqref{COMVE600}, let us consider $\omega\in \W_{X,b}$ such that $\omega(p) = 0$. Note that
$\eta(\omega)\in \Span J_{d+1}$ by the definition of $\W_{X,b}$ and $\eta(\omega)$ also vanishes at $p$. We claim that
\begin{equation}\label{COMVE076}
\eta(\omega) \in H^0(I_p(1))\otimes \Span J_d.
\end{equation}
This follows from the lemma below.

\begin{lem}\label{COMVLEMJD2}
Let $P = \PP^{n+1}$ and $J_d$ be defined in \eqref{COMVE486}
for $d\ge 3$.
Then
\begin{equation}\label{COMVE027}
\Span J_{d+1} \cap H(I_p(d+1))
= H^0(I_p(1)) \otimes \Span J_{d}
\end{equation}
for every point $p\in P$ satisfying
\begin{equation}\label{COMVE060}
p\not\in \{
(1,0,...,0), (0,1,0,...,0), ..., (0,...,0,1)
\}.
\end{equation}
Furthermore, for every $0$-dimensional subscheme $Z\subset P$ of length $2$,
a point $p\in \supp(Z)$ satisfying \eqref{COMVE060} and
$s\in H^0(I_p(1)) \backslash H^0(I_Z(1))$,
\begin{equation}\label{COMVE098}
\Span J_{d+1} \cap H\red{^0}(I_p(d+1))
= H^0(I_Z(1)) \otimes \Span J_{d} + s\otimes \Span J_d.
\end{equation}
\end{lem}

\begin{proof}
By \eqref{COMVE060}, there exist $i\ne j$ such that neither $x_i$ nor $x_j$ vanishes at $p$. Without loss
of generality, let us assume that $x_0\ne 0$ and $x_1 \ne 0$ at $p$.

It is obvious that
\begin{equation}\label{COMVE029}
\begin{aligned}
\Span J_{d+1} \cap H(I_p(d+1))
&\supset H^0(I_p(1)) \otimes \Span J_{d}
\text{ and}\\
\dim\left(\Span J_{d+1} \cap H(I_p(d+1))\right)
&= \dim \Span J_{d+1} - 1.
\end{aligned}
\end{equation}
Therefore, to show \eqref{COMVE027}, it suffices to show that
\begin{equation}\label{COMVE526}
\Span J_{d+1} = H^0(I_p(1)) \otimes \Span J_{d} + \Span
\left\{
x_0^2 x_1^{d-1}
\right\}
\end{equation}
which follows from the fact that
\begin{equation}\label{COMVE078}
x_k\in H^0(I_p(1)) + \Span\left\{x_0\right\}
\text{ and }
x_k\in H^0(I_p(1)) + \Span\left\{x_1\right\}
\end{equation}
for all $k$.

To see \eqref{COMVE098}, we observe that for all $l\in H^0(I_p(1))$ and $f\in \Span J_d$, $lf$ can be written as
\begin{equation}\label{COMVE099}
lf = (l-cs)f + csf \in H^0(I_Z(1)) \otimes \Span J_{d} + s\otimes \Span J_d,
\end{equation}
where $c$ is a constant such that $l-cs\in H^0(I_Z(1))$.
\end{proof}

Note that by \eqref{COMVE484}, $p\in Z$ always satisfies \eqref{COMVE060}.

Suppose that $a = 1$ in \eqref{COMVE059}, i.e., $Z$ is very special.
By Lemma \ref{COMVLEMVXBSPECIAL},
\begin{equation}\label{COMVE081}
\begin{aligned}
&\quad\left\{\begin{aligned}
x_0^2 \frac{\partial}{\partial x_2} - c_{012} x_0 x_1 \frac{\partial}{\partial x_0} &\in \xi(\W_{X,b})\\
x_0^2 \frac{\partial}{\partial x_3} - c_{013} x_0 x_1 \frac{\partial}{\partial x_0} &\in \xi(\W_{X,b})
\end{aligned}\right.
\\
&
\Rightarrow c_{013} x_0^2 \frac{\partial}{\partial x_2} - c_{012} x_0^2 \frac{\partial}{\partial x_3}
\in \xi(\W_{X,b})
\end{aligned}
\end{equation}
Similarly,
\begin{equation}\label{COMVE800}
\begin{aligned}
&\quad\left\{\begin{aligned}
x_1^2 \frac{\partial}{\partial x_2} - c_{102} x_0 x_1 \frac{\partial}{\partial x_1} &\in \xi(\W_{X,b})\\
x_1^2 \frac{\partial}{\partial x_3} - c_{103} x_0 x_1 \frac{\partial}{\partial x_1} &\in \xi(\W_{X,b})
\end{aligned}\right.
\\
&
\Rightarrow c_{103} x_1^2 \frac{\partial}{\partial x_2} - c_{102} x_1^2 \frac{\partial}{\partial x_3}
\in \xi(\W_{X,b})
\end{aligned}
\end{equation}
and hence
\begin{equation}\label{COMVE801}
(c_{013} x_0^2 + c_{103} x_1^2)\frac{\partial}{\partial x_2} -
(c_{012} x_0^2 + c_{102} x_1^2) \frac{\partial}{\partial x_3}
\in \xi(\W_{X,b}).
\end{equation}
Since $x_2=...=x_{n+1}=0$ at $p\ne (1,0,...,0), (0,1,0,...,0)$, neither $x_0$ nor $x_1$ vanishes at $p$.
Hence there exist numbers $r_k$ such that
$c_{01k} x_0^2 + c_{10k} x_1^2 + r_k x_0 x_1$ vanishes at $p$
for $k=2,3$. For $b$ general, \red{by \ref{COMVE026} }the
numbers $c_{ijk}$ are general with the only relations $c_{ijk} = c_{ikj}$. In particular,
\begin{equation}\label{COMVE082}
\det \begin{bmatrix}
c_{012} & c_{102}\\
c_{013} & c_{103}
\end{bmatrix} \ne 0.
\end{equation}
Therefore, at least one of $c_{012} x_0^2 + c_{102} x_1^2 + r_2 x_0 x_1$
and $c_{013} x_0^2 + c_{103} x_1^2 + r_3 x_0 x_1$ does not vanish on $Z$.
Consequently,
\begin{equation}\label{COMVE083}
\begin{aligned}
&\quad (c_{013} x_0^2 + c_{103} x_1^2 + r_3 x_0 x_1)\frac{\partial}{\partial x_2}\\
&-
(c_{012} x_0^2 + c_{102} x_1^2 + r_2 x_0 x_1) \frac{\partial}{\partial x_3}
\in \xi(\W_{X,b})
\end{aligned}
\end{equation}
vanishes at $p$ but not on $Z$.
So we may choose $\omega\in \W_{X,b}$ such that
\begin{equation}\label{COMVE084}
\begin{aligned}
\xi(\omega) &= (c_{013} x_0^2 + c_{103} x_1^2 + r_3 x_0 x_1)\frac{\partial}{\partial x_2}\\
&\quad -
(c_{012} x_0^2 + c_{102} x_1^2 + r_2 x_0 x_1) \frac{\partial}{\partial x_3},\\
\omega(p) &= 0 \text{ and } \omega\Big|_Z \ne 0.
\end{aligned}
\end{equation}
Let us write
\begin{equation}\label{COMVE085}
\xi(\omega) = s_1 \left(s_2 \frac{\partial}{\partial x_2}
+ s_3 \frac{\partial}{\partial x_3}\right)
\end{equation}
with $s_i\in H^0(\CO_P(1))$ satisfying $s_1(p) = 0$
and either $s_1s_2 \ne 0$ or $s_1s_3\ne 0$ on $Z$.

Since $\omega(p) = 0$, $\tau = \eta(\omega)$ vanishes at $p$ as well. So by Lemma \ref{COMVLEMJD2},
$\tau\in H^0(I_p(1)) \otimes \Span J_{d}$.
When we regard $\tau$ as a vector in $H^0(I_p(1)) \otimes T_{B,b}$,
we have
\begin{equation}\label{COMVE080}
L_\lambda(\tau) = s_1 \gamma
\end{equation}
for some
\begin{equation}\label{COMVE086}
\gamma\in H^0(\g_Z) = \Span \left\{
x_0 \frac{\partial}{\partial x_0}, x_1 \frac{\partial}{\partial x_0}
\right\}
\end{equation}
by \eqref{COMVE061} \red{with $a=1$}. Then
\begin{equation}\label{COMVE087}
\red{\omega|_Z} - L_\lambda(\tau) = s_1\left(
s_2 \frac{\partial}{\partial x_2} + s_3 \frac{\partial}{\partial x_3} - \gamma
\right)\in W_{X,b,Z,\lambda}.
\end{equation}
Obviously, $\red{\omega|_Z} - L_\lambda(\tau)$ vanishes at $p$. But since one of $s_1s_2$ and $s_1s_3$ does not vanish on $Z$ and $\gamma$
lies in the subspace \eqref{COMVE086} of $H^0(Z, T_P)$, it is easy to see that $\omega - L_\lambda(\tau)$
does not vanish in $H^0(Z, T_P(1))$. This finishes the proof for \eqref{COMVE600} when $Z$ is very special.

Suppose that $2\le a\le n$ in \eqref{COMVE059}. Then by
\eqref{COMVE052}, $\xi$ maps $W_{X,b}$
surjectively onto $H^0(\Lambda, T_P(1))$. So we can choose $\omega\in \W_{X,b}$ such that
\begin{equation}\label{COMVE079}
\xi(\omega) = s x_1 \frac{\partial}{\partial x_{n+1}}
\end{equation}
in $H^0(\Lambda, T_P(1))$ for some $s\in H^0(I_p(1))\backslash H^0(I_Z(1))$.
Note that $x_1$ does not vanish on either $p_i\in Z$, as explained for \eqref{COMVE314}.

By the same argument as before, we have
\begin{equation}\label{COMVE088}
\omega - L_\lambda(\tau) = s\left(
x_1 \frac{\partial}{\partial x_{n+1}} - \gamma
\right)\in W_{X,b,Z,\lambda}
\end{equation}
for some $\gamma\in H^0(\g_Z)$. Again, $\omega - L_\lambda(\tau)$ vanishes at $p$ and does not vanish in
$H^0(Z, T_P(1))$ for $a\le n$ by \eqref{COMVE061}.
This finishes the proof for \eqref{COMVE600} when $a\le n$.

Suppose that $a\ge 2$ and $x_0$ vanishes at one of $p_i$ for $b$ general. Then we choose $\omega\in \W_{X,b}$ such that
\begin{equation}\label{COMVE092}
\xi(\omega) = s x_1 \frac{\partial}{\partial x_0}
\end{equation}
in $H^0(\Lambda, T_P(1))$ for some $s\in H^0(I_p(1))\backslash H^0(I_Z(1))$.

By the same argument as before, we have
\begin{equation}\label{COMVE093}
\omega - L_\lambda(\tau) = s\left(
x_1 \frac{\partial}{\partial x_0} - \gamma
\right)\in W_{X,b,Z,\lambda}
\end{equation}
for some $\gamma\in H^0(\g_Z)$. Again, $\omega - L_\lambda(\tau)$ vanishes at $p$. Note that we choose $p$ such that
$x_0\ne 0$ at $p$. So $x_0$ must vanish at $Z\backslash \{p\}$. By \eqref{COMVE089},
\begin{equation}\label{COMVE094}
\gamma\in H^0(\g_Z) = \Span \left\{
x_0 \frac{\partial}{\partial x_0},
x_1 \frac{\partial}{\partial x_2}, ...,
x_1 \frac{\partial}{\partial x_a}
\right\}.
\end{equation}
It follows that $\omega - L_\lambda(\tau) \ne 0$ in $H^0(Z, T_P(1))$.
This finishes the proof for \eqref{COMVE600} when $a\ge 2$ and
$x_0$ vanishes at one of $p_i$.

It remains to verify \eqref{COMVE600} when $a = n+1$ in \eqref{COMVE059} and $x_0\ne 0$ at both $p_i$. In this case,
\begin{equation}\label{COMVE557}
\begin{aligned}
H^0(\g_Z) &= \Span \left\{
x_i \frac{\partial}{\partial x_j}: j = 0 \text{ or } 2 \le i,j \le n+1
\right\}
\\
&= \Span \left\{
x_1 \frac{\partial}{\partial x_j}: j = 0, 1,...,n+1
\right\}
\end{aligned}
\end{equation}
by \eqref{COMVE061}\magenta{, where the term
$x_1 (\partial/\partial x_1)$ comes from using the relation $\alpha = 0$ given by the Euler
vector field for $\alpha$ defined by \eqref{COMVE009}}.

Let us choose $s_1 = x_0 - r_1 x_1$ and $s_2 = x_0 - r_2 x_1$ for some constants $r_i$ such that
$s_i(p_i) \ne 0$ and $s_i(p_{3-i}) = 0$ for $i=1,2$. Clearly, $r_1\ne r_2\ne 0$.

Fixing $1\le k \le n+1$, we let
\begin{equation}\label{COMVE320}
u_k = x_0 \frac{\partial}{\partial x_k}
- \sum_{j=1}^{n+1} c_{0jk} x_j \frac{\partial}{\partial x_0}.
\end{equation}
Since $\omega_{00k} = x_0 u_k$, $\xi(x_0u_k)\in \xi(\W_{X,b})$. And by \eqref{COMVE316}, $\xi(x_1u_k)\in \xi(\W_{X,b})$. Therefore,
$\xi(su_k)\in \xi(\W_{X,b})$ for all $s\in H^0(\CO_P(1))$. In particular,
there exist $w_{ik}\in \W_{X,b}$ such that
\begin{equation}\label{COMVE095}
w_{ik}\Big|_\Lambda = s_i u_k\Big|_\Lambda
\end{equation}
in $H^0(\Lambda, \E(1))$ for $i=1,2$. Then by Lemma \ref{COMVLEMJD2},
\begin{equation}\label{COMVE096}
\eta(w_{ik}) - s_i \gamma_{ik} \in H^0(I_Z(1)) \otimes \Span J_{d}
\end{equation}
for some $\gamma_{ik} \in \Span J_d$ and $i=1,2$. We may write
\begin{equation}\label{COMVE321}
\eta(w_{ik}) - s_i \gamma_{ik} = \sum l_j \tau_j
\end{equation}
with $l_j\in H^0(\CO_P(1))$ and $\tau_j\in
H^0(I_Z(1)) \otimes \Span J_{d-1}$. Then
\begin{equation}\label{COMVE018}
w_{ik} - s_i L_\lambda(\gamma_{ik})
- \sum l_j L_\lambda(\tau_j)
=
s_i \left( u_k - L_\lambda\left(\gamma_{ik} \right) \right)\in W_{X,b,Z,\lambda}.
\end{equation}
when restricted to $Z$,
since $L_\lambda$ vanishes on $H^0(I_Z(1)) \otimes \Span J_{d-1}$.

By the same argument as before, we conclude that
\begin{equation}\label{COMVE097}
\left(u_k - L_\lambda\left(\gamma_{ik} \right)\right)\Big|_{p_i} = 0
\end{equation}
for $i=1,2$; otherwise, \eqref{COMVE600} follows.
By our choice of $s_i$ and $r_i$, we see that
\begin{equation}\label{COMVE019}
L_\lambda\left(\gamma_{ik} \right) = r_{3-i} x_1 \frac{\partial}{\partial x_k}
- \sum_{j=1}^{n+1} c_{0jk} x_j \frac{\partial}{\partial x_0}
\end{equation}
for $i=1,2$. In particular,
\begin{equation}\label{COMVE017}
L_\lambda\left(\gamma_{1k} - \gamma_{2k} \right) = (r_2 - r_1) x_1 \frac{\partial}{\partial x_k} \ne 0.
\end{equation}
So $x_1 (\partial /\partial x_k)$ lies in the image of
$L_\lambda$ for all $k=1,2,...,n+1$.

By \eqref{COMVE096}, $\eta(w_{ik}) - s_i \gamma_{ik} = 0$ in $H^0(\CO_\Lambda(d+1))$ and hence
\begin{equation}\label{COMVE466}
\begin{aligned}
&\quad \xi(s_i \gamma_{ik})  =  \xi(\eta(w_{ik})) = \xi(\eta(s_i u_k))
= \xi(s_i \eta(u_k))
\\
&\Rightarrow s_i (\gamma_{ik}-\eta(u_k))\Big|_\Lambda
= 0
\Rightarrow (\gamma_{ik}-\eta(u_k))\Big|_\Lambda = 0
\\
&\Rightarrow \xi(\gamma_{ik}) = \xi(\eta(u_k))
\end{aligned}
\end{equation}
for $i=1,2$.
Therefore, $\xi(\gamma_{1k} - \gamma_{2k}) = 0$ and hence
\begin{equation}\label{COMVE325}
\gamma_{1k} - \gamma_{2k} \in \Span J_{d} \cap \ker(\xi).
\end{equation}
Combining \eqref{COMVE017} and \eqref{COMVE325}, we conclude that
\begin{equation}\label{COMVE322}
\begin{aligned}
&\text{for each } 1\le k\le n+1,
\text{ there exists } \gamma_k \in \Span J_{d} \cap \ker(\xi)\\
&\hspace{132pt}\text{such that }
L_\lambda(\gamma_k) = x_1 \frac{\partial}{\partial x_k}.
\end{aligned}
\end{equation}
On the other hand, we know that
\begin{equation}\label{COMVE487}
\Span J_{d} \cap \ker(\xi) = H^0(I_Z(1)) \otimes \Span J_{d-1} + V
\end{equation}
by \eqref{COMVE034} in Lemma \ref{COMVLEMJD3} for
\begin{equation}\label{COMVE030}
\begin{aligned}
V &= \Span \Big\{
x_0^{d-2}  x_i (x_j - c_j x_1): i\ge 1, j\ge 2 \text{ and}\\
&\hspace{138pt} x_j - c_j x_1\in H^0(I_Z(1))
\Big\}.
\end{aligned}
\end{equation}
And since $L_\lambda$ vanishes on
$H^0(I_Z(1)) \otimes \Span J_{d-1}$, \eqref{COMVE322} is equivalent to saying that
\begin{equation}\label{COMVE323}
\left\{
x_1 \frac{\partial}{\partial x_k}: k\ge 1
\right\}\subset L_\lambda(V).
\end{equation}

Note that for $x_0^{d-2}  x_i (x_j -
c_j x_1)\in V$,
\begin{equation}\label{COMVE508}
\eta\left(
x_1 \otimes x_0^{d-2}  x_i (x_j - c_j x_1)
- (x_j - c_j x_1) \otimes x_0^{d-2}  x_1 x_i
\right) = 0
\end{equation}
and hence
\begin{equation}\label{COMVE100}
\begin{aligned}
&\quad L_\lambda\left(
x_1 \otimes x_0^{d-2}  x_i (x_j - c_j x_1)
- (x_j - c_j x_1) \otimes x_0^{d-2}  x_1 x_i
\right)\\
&= x_1 L_\lambda \left(x_0^{d-2}  x_i (x_j - c_j x_1)\right)
- (x_j - c_j x_1) L_\lambda\left( x_0^{d-2}  x_1 x_i
\right)\\
&= x_1 L_\lambda \left(x_0^{d-2}  x_i (x_j - c_j x_1)\right) \in W_{X,b,Z,\lambda}.
\end{aligned}
\end{equation}
It follows that $x_1 L_\lambda(\gamma) \in W_{X,b,Z,\lambda}$ for all $\gamma\in V$.
Consequently,
\begin{equation}\label{COMVE102}
\Span\left\{x_1^2 \frac{\partial}{\partial x_k}: k\ge 1\right\}
\subset W_{X,b,Z,\lambda}
\end{equation}
by \eqref{COMVE323}.

It remains to find $u\in H^0(\Lambda, \E)$ satisfying
\begin{equation}\label{COMVE103}
u \in \Span\left\{x_1 \frac{\partial}{\partial x_k}: k\ge 1\right\},\ u\ne 0 \text{ and } x_0 u\in W_{X,b,Z,\lambda}.
\end{equation}
If such $u$ exists, $u\ne 0$ at both $p_i$. Then combining \eqref{COMVE102} and \eqref{COMVE103}, we see
that $(x_0-r_1x_1) u\in W_{X,b,Z,\lambda}$ vanishes at $p_2$ but not $p_1$.

To construct $u$ satisfying \eqref{COMVE103}, let us consider
\begin{equation}\label{COMVE109}
\begin{aligned}
\omega &= c_{013}\left(\omega_{012} - \sum_{j=2}^{n+1} c_{02j}
\omega_{1j0}\right) - c_{012}\left(\omega_{013} - \sum_{j=2}^{n+1} c_{03j} \omega_{1j0}\right)\\
&= c_{013}\left(x_0x_1 \frac{\partial}{\partial x_2} - \sum_{j=2}^{n+1} c_{02j}
x_1 x_j \frac{\partial}{\partial x_0} \right)\\
&\quad - c_{012}\left(x_0x_1 \frac{\partial}{\partial x_3} - \sum_{j=2}^{n+1} c_{03j} x_1 x_j \frac{\partial}{\partial x_0}\right)
\end{aligned}
\end{equation}
in $\W_{X,b}$. We choose $\omega$ in such a way that
the expansion of $\eta(\omega)$ does not contain monomials in $J_{d+1}$ of degree $d-1$ in $x_0$. Thus, we can write
\begin{equation}\label{COMVE110}
\eta(\omega) = \sum_{i=1}^{n+1} x_i \tau_i
\end{equation}
for some $\tau_i\in \Span J_d$. Therefore, by the definition \eqref{COMVE491} of $W_{X,b,Z,\lambda}$,
\begin{equation}\label{COMVE112}
\omega - \sum_{i=1}^{n+1} x_i L_\lambda (\tau_i) \in W_{X,b,Z,\lambda}
\end{equation}
when restricted to $Z$. Combining it with \eqref{COMVE557} and \eqref{COMVE102}, we conclude
\begin{equation}\label{COMVE113}
x_0\left(c_{013} x_1 \frac{\partial}{\partial x_2} - c_{012} x_1 \frac{\partial}{\partial x_3}\right) -
\beta_1 x_1^2\frac{\partial}{\partial x_0}\in W_{X,b,Z,\lambda}
\end{equation}
for some constant $\beta_1$.
Similarly, we have
\begin{equation}\label{COMVE477}
x_0\left(c_{023} x_2 \frac{\partial}{\partial x_1} - c_{021} x_2 \frac{\partial}{\partial x_3}\right) -
\beta_2 x_1^2\frac{\partial}{\partial x_0}\in W_{X,b,Z,\lambda}
\end{equation}
for some constant $\beta_2$ by switching $x_1$ and $x_2$.
Hence by \eqref{COMVE113} and \eqref{COMVE477},
\begin{equation}\label{COMVE101}
\begin{aligned}
&x_0\Big(
e_1 c_{013} x_1 \frac{\partial}{\partial x_2} + e_2 c_{023} x_2 \frac{\partial}{\partial x_1}
\\
&\quad\quad - (e_1 c_{012} x_1 + e_2 c_{021} x_2) \frac{\partial}{\partial x_3}
\Big)\in W_{X,b,Z,\lambda}
\end{aligned}
\end{equation}
for constants $e_1$ and $e_2$, not all zero, satisfying $e_1\beta_1 + e_2\beta_2 = 0$.

For $b\in B$ general, $c_{013} c_{023}\ne 0$ and hence $e_1 c_{013}$ and $e_2 c_{023}$
cannot both vanish. Therefore,
\begin{equation}\label{COMVE324}
u = e_1 c_{013} x_1 \frac{\partial}{\partial x_2} + e_2 c_{023} x_2 \frac{\partial}{\partial x_1}
- (e_1 c_{012} x_1 + e_2 c_{021} x_2) \frac{\partial}{\partial x_3}
\end{equation}
satisfies \eqref{COMVE103}.

This finishes the proof of \eqref{COMVE600} for $Z$ special. Thus, if $Z = \{\sigma_1(b),\sigma_2(b)\}$ is special with respect to $(x_i)$ for all $b\in B$, then
$\sigma_1(b)$ and $\sigma_2(b)$ are not $\Gamma$-equivalent over $\BQ$ on $X_b$ for $b\in B$ general.

\subsection{Generic case}

Next we will finish the proof of our main theorem by proving \eqref{COMVE600} for $Z$ generic.
We start with a result on $\xi(\W_{X,b})$ for $Z$ generic, similar to Lemma \ref{COMVLEMVXBSPECIAL}.

\begin{lem}\label{COMVLEMVXBGENERIC}
Let $P = \PP^{n+1}$ and $X\subset Y = B\times P$ be the family of hypersurfaces in $P$ given by \eqref{COMVE484}
over $B = \Span J_d$ for $n\ge 2$ and $d\ge 4$. Then $\xi$ is surjective when restricted to $\W_{X,b}$, i.e.,
\begin{equation}\label{COMVE016}
\xi(\W_{X,b}) = H^0(\Lambda, \E(1))
\end{equation}
for $b\in B$ general and \underline{all} $0$-dimensional subschemes $Z\subset X_b$ of length $2$ that are generic
with respect to $(x_i)$, where $\Lambda\subset P$ is the line \magenta{cutting out $Z$ }on $X_b$ and $\xi$ is the restriction
$H^0(\E(1))\to H^0(\Lambda, \E(1))$.
\end{lem}
To keep our argument in sight, we postpone its proof till the end of the subsection.

By the isomorphism \eqref{COMVE028}, $L_\lambda$ actually induces a map
\begin{equation}\label{COMVE106}
\begin{tikzcd}
\displaystyle{\frac{\Span J_d}{H^0(I_Z(1)) \otimes \Span J_{d-1}}} \ar{r}[above]{L_\lambda} \ar{d}[left]{\xi}[right]{\cong} &
H^0(Z,T_P)\\
\Sym^d H^0(\CO_Z(1)) \ar[equal]{r} & H^0(\CO_{\Lambda}(d)) \ar{u}{L_\lambda}
\end{tikzcd}
\end{equation}
As before, we choose $s_i\in H^0(\CO_P(1))$ such that
$s_i(p_i)\ne 0$ and $s_i(p_{3-i}) = 0$ for $i=1,2$.
For every $u\in H^0(\E)$, by Lemma \ref{COMVLEMVXBGENERIC}, there exist
$\omega_i\in \W_{X,b}$ such that
\begin{equation}\label{COMVE031}
\xi(\omega_i) = \xi(s_i u)
\end{equation}
in $H^0(\Lambda, \E(1))$ for $i=1,2$. Then as \eqref{COMVE096}, \red{by Lemma \ref{COMVLEMJD2}} we have
\begin{equation}\label{COMVE114}
\eta(\omega_i) - s_i \gamma_i \in H^0(I_Z(1)) \otimes \Span J_{d}
\end{equation}
for some $\gamma_i \in \Span J_d$. It follows that
\begin{equation}\label{COMVE115}
s_i\left(u - L_\lambda(\gamma_i) \right) \in W_{X,b,Z,\lambda}
\end{equation}
for $i=1,2$, when restricted to $Z$. As before, we must have
\begin{equation}\label{COMVE116}
\left(u - L_\lambda(\gamma_i)\right) \Big|_{p_i} = 0
\end{equation}
for $i=1,2$; otherwise, \eqref{COMVE600} follows.

By \eqref{COMVE114}\magenta{\ and $H^0(\Lambda, I_Z(1)) = 0$}, $\xi\left(\eta(\omega_i) - s_i \gamma_i\right) = 0$ and hence
\begin{equation}\label{COMVE117}
\begin{aligned}
&\quad \xi(s_i \gamma_i) = \xi \left(\eta(\omega_i)\right) =
\xi \left(\eta(s_i u)\right)
= \xi \left(s_i \eta(u)\right)
\\
&\Rightarrow
s_i(\gamma_i - \eta(u))\Big|_\Lambda = 0
\Rightarrow (\gamma_i - \eta(u))\Big|_\Lambda = 0
\Rightarrow \xi(\gamma_i) = \xi(\eta(u))
\end{aligned}
\end{equation}
for $i=1,2$. Then $\xi(\gamma_1) = \xi(\gamma_2)$
and $\gamma_1 - \gamma_2 \in H^0(I_Z(1)) \otimes \Span J_{d-1}$
by Lemma \ref{COMVLEMJD}. Therefore, $L_\lambda(\gamma_1) = L_\lambda(\gamma_2)$.
Combining this with \eqref{COMVE116}, we conclude that
\begin{equation}\label{COMVE119}
u\magenta{\Big|_Z} = L_\lambda(\gamma_1) = L_\lambda(\gamma_2)
\end{equation}
in $H^0(Z, T_P)$. This implies that the map $L_\lambda$ in \eqref{COMVE106} is onto.
Indeed, the combination of \eqref{COMVE117} and \eqref{COMVE119} tells us exactly what $L_\lambda$ is:
\begin{equation}\label{COMVE121}
\boxed{L_\lambda(\gamma) = u\Big|_Z
	\text{ if } \gamma\Big|_\Lambda = \eta(u)
	\Big|_\Lambda}
\end{equation}
for $\gamma\in T_{B,b} = \Span J_d$ and $u\in H^0(\E)$.
Let us see the geometric implication of \eqref{COMVE121}.

Let $\widehat{\eta}$ be the map given by the commutative diagram
\begin{equation}\label{COMVE120}
\begin{tikzcd}
H^0(\E) \ar[two heads]{d}[left]{\xi} \ar{r}{\eta} & H^0(\CO_P(d)) \ar[two heads]{d}[right]{\xi}\\
H^0(\Lambda, \E) \ar{r}{\widehat{\eta}} & H^0(\CO_\Lambda(d)).
\end{tikzcd}
\end{equation}
Obviously, $\widehat{\eta}$ is the restriction of $\eta$ to $\Lambda$ and defined in the same way as $\eta$ by
\begin{equation}\label{COMVE122}
\widehat{\eta}\left(
x_i \frac{\partial}{\partial x_j}
\right) = x_i \frac{\partial F}{\partial x_j}
\end{equation}
for all $0\le i,j\le n+1$ with everything restricted to $\Lambda$.

Since
\begin{equation}\label{COMVE123}
h^0(\Lambda, \E) - h^0(\CO_{\Lambda}(d)) = 2(n + 2) - (d+1) > 0
\end{equation}
for $d = 2n+2$, there exists $u\ne 0\in h^0(\Lambda, \E)$ such that
$\widehat{\eta}(u) = 0$. By \eqref{COMVE121}, $u$ vanishes in $H^0(Z, T_P)$. That is, $u$ lies in the kernel of the map
\begin{equation}\label{COMVE104}
\begin{tikzcd}
H^0(\Lambda, \E) \ar{r}{\rho} & H^0(Z, T_P).
\end{tikzcd}
\end{equation}
Obviously, $\ker(\rho)$ is two dimensional and $\alpha\in \ker(\rho)$ for $\alpha$ given in \eqref{COMVE009}.

We can make everything very explicit. If we identify $\Lambda$ with $\PP^1$ and let $p_1 = (0,1)$, $p_2 = (1,0)$ and $y$
be the affine coordinate of $\Lambda\backslash p_2$, then
\begin{equation}\label{COMVE105}
\widehat{\eta}\left(
\ker(\rho)
\right) = \Span \{ f(y), yf'(y)\}
\end{equation}
for $f(y) = \widehat{\eta}(\alpha) = \xi(F)\in H^0(\CO_{\Lambda}(d))$. \magenta{
To see this, we let
$$
x_i = a_i y + b_i
$$
be the restriction of $x_i$ to $\Lambda$. Then
$$
\sum g_i(y) \frac{\partial}{\partial x_i}\in \ker(\rho)
\Leftrightarrow \frac{g_0(y)}{a_0 y + b_0} = \frac{g_1(y)}{a_1 y + b_1} = ... = \frac{g_{n+1}(y)}{a_{n+1} y + b_{n+1}}\text{ at } 0,\infty
$$
Therefore, $\ker(\rho)$ is spanned by
$$
\alpha = \sum (a_i y + b_i) \frac{\partial}{\partial x_i} \hspace{12pt}\text{and}
\hspace{12pt} \beta = \sum a_i y \frac{\partial}{\partial x_i}.
$$
Then
$$
\widehat{\eta}(\beta) = \sum a_i y  \left.\frac{\partial F}{\partial x_i}\right|_\Lambda = y \sum \frac{d}{dy} (a_iy + b) \left.\frac{\partial F}{\partial x_i}\right|_\Lambda = y f'(y)
$$
and \eqref{COMVE105} follows.

}Since $u\ne 0\in \ker(\rho)$ and $\widehat{\eta}(u) = 0$, we conclude that
$f(y)$ and $yf'(y)$ must be two linearly dependent polynomials in $y$.
This can only happen if $f(y) = cy^m$, i.e., $\xi(F)$ vanishes only at $p_1$ and $p_2$. Namely, $X_b$ and $\Lambda$
have no intersections other than $p_1$ and $p_2$. So we have reached our key conclusion:

\begin{prop}\label{COMVPROPKC}
If there are two points $p_1\ne p_2$ on a general hypersurface
$X\subset \PP^{n+1}$ of degree $2n+2$ that are $\Gamma$-equivalent over $\BQ$,
then the line $\Lambda$ joining $p_1$ and $p_2$ meets $X$ only at $p_1$ and $p_2$.
\end{prop}

It remains to prove the following:

\begin{prop}\label{COMVPROPLINE}
Let $P=\PP^{n+1}$, $\G(1,P)$ be the Grassmannian of lines in $P$ and $B = \PP H^0(\CO_P(d))$ be the parameter space
of hypersurfaces in $P$ of degree $d = 2n+2$. For $0 < m < d$, let $W_m$ be the incidence correspondence
\begin{equation}\label{COMVE108}
\begin{aligned}
W_m &= \Big\{ (X, \Lambda, p_1, p_2): p_1\ne p_2 \text{ and }
X.\Lambda = m p_1 + (d-m) p_2  \Big\}
\\
& \subset B\times \G(1,P)\times P\times P.
\end{aligned}
\end{equation}
Then
\begin{enumerate}
\item[(1)]
$W_m$ is irreducible.
\item[(2)]
$W_m$ is generically finite over $B$ via the projection $\pi: W_m\to B$.
\item[(3)]
For a general $X\in B$, the fiber $\pi^{-1}([X])$ contains at least two points
$(X, \Lambda_i, p_{i1}, p_{i2})$ for $i=1,2$ such that $p_{11} \ne p_{21}$ and the line joining
$p_{11}$ and $p_{21}$ meet $X$ at more than two points.
\end{enumerate}
\end{prop}

Let us see how the above proposition implies our main theorem.
We consider the incidence correspondence
\begin{equation}\label{COMVE124}
\begin{aligned}
W &= \Big\{ (X, \Lambda, p_1, p_2): p_1\ne p_2\in X\cap \Lambda \text{ and }
p_1 \sim_{\Gamma} p_2 \text{ over } \BQ \Big\}
\\
& \subset B\times \G(1,P)\times P\times P
\end{aligned}
\end{equation}
for $B = \PP H^0(\CO_P(d))$. This is a locally Noetherian scheme, a priori.

If no components of $W$ dominate $B$, we are done. Otherwise, by Proposition \ref{COMVPROPKC}
and \ref{COMVPROPLINE}, $W$ must contain some $W_m$ as an irreducible component. Then by
Proposition \ref{COMVPROPLINE} again, for $X\in B$ general, there exist
$(X, \Lambda_i, p_{i1}, p_{i2})\in W_m\subset W$ for $i=1,2$ such that $p_{11} \ne p_{21}$ and
the line joining $p_{11}$ and $p_{21}$ meet $X$ at more than two points.

Since $p_{i1}\sim_{\Gamma} p_{i2}$ over $\BQ$ and $X.\Lambda_i = m p_{i1} + (d-m) p_{i2}$, we have
\begin{equation}\label{COMVE125}
\magenta{X.\Lambda_1 \sim_{\PP^1} X.\Lambda_2 \Rightarrow X.\Lambda_1 \sim_\Gamma X.\Lambda_2\Rightarrow }
d p_{i1} \sim_\Gamma d p_{i2} \sim_\Gamma X.\Lambda_i
\end{equation}
over $\BQ$ on $X$ for $i=1,2$. It follows that all four points $p_{ij}$ are $\Gamma$-equivalent over $\BQ$.
Then by Proposition \ref{COMVPROPKC} again, the line joining $p_{11}$ and $p_{21}$ must meet $X$
only at $p_{11}$ and $p_{21}$, which is a contradiction.

It remains to prove Proposition \ref{COMVPROPLINE}.

\begin{proof}[Proof of Proposition \ref{COMVPROPLINE}]
The proof of this statement is fairly standard. To see that $W_m$ is irreducible of \red{dimension} $\dim B$, it suffices to
project it to $\G(1,P)\times P\times P$. The fiber of $W_m$ over $(\Lambda, p_1, p_2)$ for $p_1\ne p_2\in \Lambda$
is a linear subspace of $B$ of dimension $\dim B - d$. Therefore, $W_m$ is irreducible of dimension
\begin{equation}\label{COMVE126}
\begin{aligned}
\dim W_m &= \dim \big\{(\Lambda, p_1,p_2): p_1\ne p_2\in \Lambda\big\} + (\dim B - d)
\\
&= \dim \G(1,P) + 2 - d + \dim B\\
&= \dim B + (2n+2 - d) = \dim B
\end{aligned}
\end{equation}
for $d=2n+2$.

To show that $W_m$ is generically finite over $B$, it suffices to exhibit
a point $(X,\Lambda, p_1, p_2)\in W_m$ such that $\Lambda$ does not deform while preserving the
tangency conditions with $X$. By that we mean there does not exist a one-parameter family of
lines $\Lambda_t$ such that $\Lambda_0 = \Lambda$ and $\Lambda_t$ meets $X$ at
two points with multiplicities $m$ and $d-m$, respectively. Such deformation of $\Lambda$ is
governed by the standard exact sequence
\begin{equation}\label{COMVE127}
\begin{tikzcd}
0 \ar{r} & T_\Lambda(-p_1 -p_2) \ar{r} & T_P(-\log X)\Big|_\Lambda
\ar{r} & N \ar{r} & 0.
\end{tikzcd}
\end{equation}
It is easy to find $(X,\Lambda, p_1,p_2)\in W_m$ such that $H^0(N) = 0$. We leave the details to Appendix \ref{COMVSECSF}.

Finally, to show (3), it again suffices to exhibit $(X,\Lambda_i,p_{i1},p_{i2})\in W_m$ for $i=1,2$ with the
required properties and neither $\Lambda_1$ nor $\Lambda_2$ deforms while preserving the tangency
conditions with $X$. Again, it is easy to find such $X$ and $\Lambda_i$ and use the exact sequence
\eqref{COMVE127} to show that $\Lambda_i$ do not deform. \red{Again, we refer the reader to Appendix A for the details.}
\end{proof}

This finishes the proof of our main theorem \ref{COMVTHMHYPERSURFACE}. It remains to provide the following
\begin{proof}[Proof of Lemma \ref{COMVLEMVXBGENERIC}]
Let $\{\omega_{ijk}\}$ be the basis of $\W_{X,b}$ given by \eqref{COMVE514} with $c_{ijk}$ given by \eqref{COMVE026}. For $b\in B$ general, $\{c_{ijk}: 0\le i\ne j,k\le n+1\}$ is a general set of numbers satisfying $c_{ijk} = c_{ikj}$.

We write $u_1 \equiv u_2$ if $\xi(u_1 - u_2)\in \xi(\W_{X,b})$. Of course,
we have $\omega_{ijk} \equiv 0$ and
want to show that $u\equiv 0$ for all $u\in H^0(\E(1))$.

For starters, it is obvious that
\begin{equation}\label{COMVE300}
\omega_{ijk} \equiv 0 \Rightarrow x_i x_j \frac{\partial}{\partial x_k} \equiv 0
\text{ for all } i\ne j\ne k
\end{equation}
and
\begin{equation}\label{COMVE305}
\omega_{iik} \equiv 0 \Rightarrow x_i^2 \frac{\partial}{\partial x_k}
- \sum_{j\ne i} c_{ijk} x_i x_j \frac{\partial}{\partial x_i}
\equiv 0
\text{ for all } i\ne k.
\end{equation}

Without loss of generality, we assume \eqref{COMVE522}. We discuss in two cases:
\begin{enumerate}
\item
Suppose that
\begin{equation}\label{COMVE033}
\Span \{x_0,x_1\} = \Span \{x_1,x_i\} = \Span\{x_i, x_0\} = H^0(\CO_Z(1))
\end{equation}
for some $i$.
Without loss of generality, we may assume that $i=2$. Namely, we have
\begin{equation}\label{COMVE020}
\Span \{x_0,x_1\} = \Span \{x_1,x_2\} = \Span\{x_2, x_0\} = H^0(\CO_Z(1)).
\end{equation}
\item
Otherwise,
suppose that there does not exist $x_i$ satisfying \eqref{COMVE033}. Namely, for each $x_i$, either
$x_i\in \Span\{x_0\}$ or $x_i\in \Span\{x_1\}$ in $H^0(\CO_Z(1))$.
And since $Z$ is generic, there must exist $i\ne j\ne 0,1$ such that
\begin{equation}\label{COMVE044}
\Span \{x_0,x_i\} = \Span \{x_1,x_j\} = H^0(\CO_Z(1)).
\end{equation}
Without loss of generality, we may assume that $i=3$ and $j=2$. In summary, when \eqref{COMVE033} fails, we may assume that
\begin{equation}\label{COMVE045}
\begin{aligned}
\Span \{x_0,x_3\} &= \Span \{x_1,x_2\} = \Span\{x_0,x_1\} = H^0(\CO_Z(1)) \text{ and}\\
\{x_2,...,x_{n+1}\} &\subset \Span\{x_0\} \cup \Span\{x_1\}\text{ in } H^0(\CO_Z(1)).
\end{aligned}
\end{equation}
\end{enumerate}

In the first case, we assume \eqref{COMVE020}.
Then for all $k\ne 0,1,2$ and all $i,j$,
\begin{equation}\label{COMVE021}
x_0x_1 \frac{\partial}{\partial x_k} \equiv x_1x_2 \frac{\partial}{\partial x_k} \equiv x_0 x_2 \frac{\partial}{\partial x_k} \equiv 0
\end{equation}
and hence
\begin{equation}\label{COMVE037}
x_i x_j \frac{\partial}{\partial x_k} \equiv 0
\end{equation}
since $\{x_0x_1, x_1x_2, x_0x_2\}$ spans $H^0(\CO_\Lambda(2))$ by \eqref{COMVE020}.

Suppose that $x_k\ne 0$ in $H^0(\CO_Z(1))$ for some $3\le k\le n+1$.
Without loss of generality, suppose that $x_3\ne 0$ in $H^0(\CO_Z(1))$.
Then at least two pairs among $\{x_0,x_3\}$, $\{x_1,x_3\}$ and $\{x_2,x_3\}$ are linearly independent in $H^0(\CO_Z(1))$. Without loss of generality, let us assume that
\begin{equation}\label{COMVE032}
\Span \{x_0,x_1\} = \Span \{x_1,x_3\} = \Span\{x_3, x_0\} = H^0(\CO_Z(1)).
\end{equation}
Then
\begin{equation}\label{COMVE310}
x_0x_1 \frac{\partial}{\partial x_2} \equiv x_1 x_3 \frac{\partial}{\partial x_2} \equiv x_0 x_3 \frac{\partial}{\partial x_2} \equiv 0 \Rightarrow x_ix_j\frac{\partial}{\partial x_2} \equiv 0
\end{equation}
for all $i,j$. That is,
\eqref{COMVE037} holds for
$k=2$ as well. Thus, it holds for
all $k\ne 0,1$:
\begin{equation}\label{COMVE036}
x_i x_j \frac{\partial}{\partial x_k} \equiv 0 \text{ if } k\ne 0,1.
\end{equation}
It remains to prove \eqref{COMVE037}
for $k=0,1$.

\magenta{Setting $i \ne 0,1$ and $k = 0,1$ in \eqref{COMVE305}, we have
$$
x_i^2 \frac{\partial}{\partial x_k}
\equiv \sum_{j\ne i} c_{ijk} x_i x_j \frac{\partial}{\partial x_i}
\equiv 0
$$
by \eqref{COMVE036}. Switching $i$ and $k$, we can rewrite the above as
$$
x_k^2 \frac{\partial}{\partial x_i}
\equiv 0 \text{ for all } i=0,1 \text{ and } k\ne 0,1.
$$
Similarly, setting $i = 0,1$ and $k \ne 0,1$ in \eqref{COMVE305}, we have
$$
\sum_{j\ne i} c_{ijk} x_i x_j \frac{\partial}{\partial x_i}
\equiv x_i^2 \frac{\partial}{\partial x_k}
\equiv 0
$$
by \eqref{COMVE036}. In summary, by \eqref{COMVE305} and \eqref{COMVE036},}
we see that
\begin{equation}\label{COMVE035}
x_k^2 \frac{\partial}{\partial x_i} \equiv x_i\sum_{j\ne i} c_{ijk}
x_j\frac{\partial}{\partial x_i} \equiv 0
\text{ for all } i=0,1 \text{ and } k\ne 0,1.
\end{equation}
Setting $i=0$ in \eqref{COMVE035} and combining it with \eqref{COMVE300}, we have
\begin{equation}\label{COMVE038}
x_k^2 \frac{\partial}{\partial x_0}
\equiv
x_0 \sum_{j\ne 0} c_{0jk}
x_j \frac{\partial}{\partial x_0}
\equiv  x_kx_l \frac{\partial}{\partial x_0}
\equiv 0
\text{ for all } k>l\ge 1.
\end{equation}
If $\Span\{x_k,x_l\} = H^0(\CO_Z(1))$ for some $k>l\ge 2$, then
\begin{equation}\label{COMVE309}
x_k^2 \frac{\partial}{\partial x_0} \equiv x_l^2 \frac{\partial}{\partial x_0} \equiv x_k x_l \frac{\partial}{\partial x_0} \equiv 0 \Rightarrow x_ix_j\frac{\partial}{\partial x_0} \equiv 0
\end{equation}
for all $i,j$ by \eqref{COMVE038}. Otherwise,
$x_k$ and $x_l$ are \red{linearly} dependent in $H^0(\CO_Z(1))$ for all $k>l\ge 2$. This implies that
\begin{equation}\label{COMVE040}
x_3,...,x_{n+1} \in \Span\{x_2\}
\end{equation}
in $H^0(\CO_Z(1))$. Thus
\begin{equation}\label{COMVE041}
x_2^2 \frac{\partial}{\partial x_0}
\equiv x_1x_2 \frac{\partial}{\partial x_0}
\equiv 0 \Rightarrow
x_0 x_2 \frac{\partial}{\partial x_0}
\equiv 0 \Rightarrow x_0 x_j \frac{\partial}{\partial x_0}
\equiv 0 \text{ for } j\ge 2
\end{equation}
since $x_0 \in \Span\{x_1,x_2\}$.
So we may rewrite \eqref{COMVE038} as
\begin{equation}\label{COMVE039}
x_2^2 \frac{\partial}{\partial x_0}
\equiv x_1x_2 \frac{\partial}{\partial x_0}
\equiv c_{01k} x_0 x_1 \frac{\partial}{\partial x_0}
\equiv
0
\end{equation}
for all $k\ge 2$. \red{Since $c_{012} \ne 0$ for general $b$}, we have
\begin{equation}\label{COMVE311}
x_0 x_1 \frac{\partial}{\partial x_0}
\equiv 0\Rightarrow x_1^2 \frac{\partial}{\partial x_0}
\equiv 0
\end{equation}
since $x_0 = b_1 x_1 + b_2 x_2$ in $H^0(\CO_Z(1))$ for some $b_i\ne 0$ by \eqref{COMVE020}.
Combining \eqref{COMVE039} and \eqref{COMVE311}, we conclude that
$x_ix_j(\partial/\partial x_0) \equiv 0$
for all $i,j$.
This proves \eqref{COMVE037}
for $k=0$. The same argument works for $k=1$. This finishes the proof of the lemma if we have \eqref{COMVE020} and one of $x_3,...,x_{n+1}$ does not vanish in $H^0(\CO_Z(1))$.

Otherwise, while we still have \eqref{COMVE020},
$x_3=...=x_{n+1}=0$
in $H^0(\CO_Z(1))$. Then we have a system of
linear equations:
\begin{equation}\label{COMVE042}
\begin{aligned}
x_0x_1\frac{\partial}{\partial x_2}
\equiv x_1x_2\frac{\partial}{\partial x_0}
\equiv x_0x_2\frac{\partial}{\partial x_1}
&\equiv 0\\
(c_{013}x_0 x_1 + c_{023} x_0 x_2)
\frac{\partial}{\partial x_0}
&\equiv 0\\
(c_{103}x_1 x_0 + c_{123} x_1 x_2)
\frac{\partial}{\partial x_1}
&\equiv 0\\
(c_{203}x_2 x_0 + c_{213} x_2 x_1)
\frac{\partial}{\partial x_2}
&\equiv 0\\
x_0^2 \frac{\partial}{\partial x_1}
- (c_{011}x_0 x_1 + c_{021} x_0 x_2)
\frac{\partial}{\partial x_0}
&\equiv 0\\
x_0^2 \frac{\partial}{\partial x_2}
- (c_{012}x_0 x_1 + c_{022} x_0 x_2)
\frac{\partial}{\partial x_0}
&\equiv 0\\
x_1^2 \frac{\partial}{\partial x_0}
- (c_{100}x_1 x_0 + c_{120} x_1 x_2)
\frac{\partial}{\partial x_1}
&\equiv 0\\
x_1^2 \frac{\partial}{\partial x_2}
- (c_{102}x_1 x_0 + c_{122} x_1 x_2)
\frac{\partial}{\partial x_1}
&\equiv 0\\
x_2^2 \frac{\partial}{\partial x_0}
- (c_{200}x_2 x_0 + c_{210} x_2 x_1)
\frac{\partial}{\partial x_2}
&\equiv 0\\
x_2^2 \frac{\partial}{\partial x_1}
- (c_{201}x_2 x_0 + c_{211} x_2 x_1)
\frac{\partial}{\partial x_2}
&\equiv 0
\end{aligned}
\end{equation}
Suppose that
\begin{equation}\label{COMVE043}
a_0 x_0 + a_1 x_1 + a_2 x_2 = 0
\end{equation}
in $H^0(\CO_Z(1))$ for some constants
$a_0, a_1, a_2$, not all zero. By our hypothesis \eqref{COMVE020}, $a_i\ne 0$
for $i=0,1,2$.

Using \eqref{COMVE043}, we can reduce
\eqref{COMVE042} into a system of linear equations in $x_i^2(\partial/\partial x_j)$ for
$0\le i\ne j\le 2$. For example,
\begin{equation}\label{COMVE012}
\begin{aligned}
\left.\begin{aligned}
(c_{013}x_0 x_1 + c_{023} x_0 x_2)
\frac{\partial}{\partial x_0}
&\equiv 0\\
x_1x_2 \frac{\partial}{\partial x_0} &\equiv 0
\end{aligned}\right\}
&\Rightarrow\\
(a_1 x_1 + a_2 x_2)(c_{013}x_1 + c_{023} x_2)
\frac{\partial}{\partial x_0}
&\equiv a_1 c_{013} x_1^2 \frac{\partial}{\partial x_0}
+ a_2 c_{023} x_2^2 \frac{\partial}{\partial x_0}\\
&\equiv 0.
\end{aligned}
\end{equation}
In this way, we obtain a more managable system of linear equations:
\begin{equation}\label{COMVE047}
\begin{aligned}
c_{013} \left(a_1 x_1^2 \frac{\partial}{\partial x_0}\right)
+ c_{023} \left(a_2 x_2^2 \frac{\partial}{\partial x_0}\right) &\equiv 0\\
c_{103} \left(a_0 x_0^2 \frac{\partial}{\partial x_1}\right) +  c_{123} \left(a_2 x_2^2
\frac{\partial}{\partial x_1}\right)
&\equiv 0\\
c_{203} \left(a_0 x_0^2 \frac{\partial}{\partial x_2}\right) +  c_{213} \left(a_1 x_1^2
\frac{\partial}{\partial x_2}\right)
&\equiv 0
\\
a_0 x_0^2 \frac{\partial}{\partial x_1}
+ c_{011} \left(a_1 x_1^2 \frac{\partial}{\partial x_0}\right) +  c_{021} \left(a_2 x_2^2
\frac{\partial}{\partial x_0}\right)
&\equiv 0\\
a_0 x_0^2 \frac{\partial}{\partial x_2}
+
c_{012} \left(a_1 x_1^2 \frac{\partial}{\partial x_0}\right) +  c_{022} \left(a_2 x_2^2
\frac{\partial}{\partial x_0}\right)
&\equiv 0\\
a_1x_1^2 \frac{\partial}{\partial x_0}
+ c_{100} \left(a_0 x_0^2 \frac{\partial}{\partial x_1}\right) + c_{120} \left(a_2 x_2^2
\frac{\partial}{\partial x_1}\right)
&\equiv 0\\
a_1 x_1^2 \frac{\partial}{\partial x_2}
+ c_{102}\left(a_0x_0^2 \frac{\partial}{\partial x_1}\right) + c_{122} \left(a_2 x_2^2
\frac{\partial}{\partial x_1}\right)
&\equiv 0\\
a_2 x_2^2 \frac{\partial}{\partial x_0}
+ c_{200} \left(a_0 x_0^2 \frac{\partial}{\partial x_2}\right) + c_{210} \left(a_1 x_1^2
\frac{\partial}{\partial x_2}\right)
&\equiv 0\\
a_2 x_2^2 \frac{\partial}{\partial x_1}
+ c_{201} \left(a_0x_0^2 \frac{\partial}{\partial x_2}\right) + c_{211} \left(a_1 x_1^2
\frac{\partial}{\partial x_2}\right)
&\equiv 0.
\end{aligned}
\end{equation}
We may consider \eqref{COMVE047} as a system of homogeneous linear equations in $a_i x_i^2 (\partial/\partial x_j)$ for $0\le i\ne j\le 2$. It is easy to show that \eqref{COMVE047} has only the trivial solution for $c_{ijk}$ general. That is,
\begin{equation}\label{COMVE048}
a_i x_i^2 \frac{\partial}{\partial x_j}
\equiv 0 \Rightarrow x_i^2 \frac{\partial}{\partial x_j} \equiv 0
\text{ for all } i\ne j.
\end{equation}
Together with \eqref{COMVE300}, we see that
\eqref{COMVE037} holds for all $i,j,k$.
This finishes the proof of the lemma in the first case.

In the second case, we assume \eqref{COMVE045}.
Note that under this hypothesis, $\{x_0, x_2\}$ and $\{x_1, x_3\}$ are linearly dependent in $H^0(\CO_Z(1))$, respectively.
Then for all $k\ne 0,1,2,3$,
\begin{equation}\label{COMVE046}
\begin{aligned}
&\quad x_0 x_1 \frac{\partial}{\partial x_k}
\equiv x_0x_2 \frac{\partial}{\partial x_k}
\equiv x_1 x_3 \frac{\partial}{\partial x_k}
\equiv 0
\\
&\Rightarrow
x_0 x_1 \frac{\partial}{\partial x_k}
\equiv x_0^2 \frac{\partial}{\partial x_k}
\equiv x_1^2 \frac{\partial}{\partial x_k}
\equiv 0.
\end{aligned}
\end{equation}
And since $\{x_0^2, x_0x_1, x_1^2\}$ spans $H^0(\CO_\Lambda(2))$,
we see that \eqref{COMVE037} holds for all $k\ge 4$. It remains to prove
\eqref{COMVE037} for $k=0,1,2,3$.
We argue in a similar way to the first case.

Suppose that one of $x_4,...,x_{n+1}$ does not vanish in $H^0(\CO_Z(1))$.
Without loss of generality, suppose that $x_4 \ne 0$ in $H^0(\CO_Z(1))$. By \eqref{COMVE045}, $x_4$ lies in either $\Span \{x_0\}$ or $\Span\{x_1\}$. Without loss of generality, we may assume that
$x_4\ne 0\in \Span\{x_0\}$ in $H^0(\CO_Z(1))$. Then
\begin{equation}\label{COMVE312}
\begin{aligned}
&\quad x_1 x_4 \frac{\partial}{\partial x_0}
\equiv x_2x_4 \frac{\partial}{\partial x_0}
\equiv x_1 x_3 \frac{\partial}{\partial x_0}
\equiv 0
\\
&\Rightarrow
x_0 x_1 \frac{\partial}{\partial x_0}
\equiv x_0^2 \frac{\partial}{\partial x_0}
\equiv x_1^2 \frac{\partial}{\partial x_0}
\equiv 0 \text{ and}\\
&\quad x_0 x_1 \frac{\partial}{\partial x_2}
\equiv x_0x_4 \frac{\partial}{\partial x_2}
\equiv x_1 x_3 \frac{\partial}{\partial x_2}
\equiv 0
\\
&\Rightarrow
x_0 x_1 \frac{\partial}{\partial x_2}
\equiv x_0^2 \frac{\partial}{\partial x_2}
\equiv x_1^2 \frac{\partial}{\partial x_2}
\equiv 0.
\end{aligned}
\end{equation}
So \eqref{COMVE037} holds for $k=0,2$ and hence for all $k\ne 1, 3$.

Let us prove \eqref{COMVE037} for $k=1$.
If $x_k\ne 0\in \Span\{x_1\}$
in $H^0(\CO_Z(1))$ for some $k\ge 5$, then we have \eqref{COMVE037} for $k=1,3$ by the same argument as above. Otherwise, $x_k\in \Span\{x_0\}$ for all $k\ne 1,3$. Then
\begin{equation}\label{COMVE049}
x_0 x_2 \frac{\partial}{\partial x_1}
\equiv x_0 x_3 \frac{\partial}{\partial x_1}
\equiv 0 \Rightarrow
x_0^2 \frac{\partial}{\partial x_1}
\equiv x_0 x_1 \frac{\partial}{\partial x_1}
\equiv 0
\end{equation}
and
\begin{equation}\label{COMVE050}
x_1^2 \frac{\partial}{\partial x_0} - \sum_{j\ne 1} c_{1j0}x_1 x_j\frac{\partial}{\partial x_1}
\equiv 0
\Rightarrow c_{130} x_1 x_3 \frac{\partial}{\partial x_1} \equiv 0.
\end{equation}
As long as $c_{130}\ne 0$, we have
\begin{equation}\label{COMVE051}
x_1 x_3 \frac{\partial}{\partial x_1} \equiv 0 \Rightarrow x_1^2 \frac{\partial}{\partial x_1} \equiv 0
\end{equation}
which, together with \eqref{COMVE049}, implies \eqref{COMVE037} for $k=1$. The same argument works for $k=3$. This proves the lemma if we have \eqref{COMVE045} and
one of $x_4,...,x_{n+1}$ does not vanish in $H^0(\CO_Z(1))$.

The only remaining case is that we have \eqref{COMVE045} and $x_4 = ... = x_{n+1} = 0$ in $H^0(\CO_Z(1))$.
In this case, we have
\begin{equation}\label{COMVE306}
\begin{aligned}
x_1 x_2 \frac{\partial}{\partial x_0}
\equiv x_1x_3 \frac{\partial}{\partial x_0}
\equiv 0 &\Rightarrow
x_0 x_1 \frac{\partial}{\partial x_0}
\equiv x_1^2 \frac{\partial}{\partial x_0}
\equiv 0\\
x_0 x_1 \frac{\partial}{\partial x_2}
\equiv x_1 x_3 \frac{\partial}{\partial x_2}
\equiv 0
&\Rightarrow x_0 x_1 \frac{\partial}{\partial x_2}
\equiv x_1^2 \frac{\partial}{\partial x_2}
\equiv 0
\end{aligned}
\end{equation}
and
\begin{equation}\label{COMVE313}
\begin{aligned}
x_0^2 \frac{\partial}{\partial x_2}
- \sum_{j\ne 0} c_{0j2} x_0x_j \frac{\partial}{\partial x_0} \equiv 0
&\Rightarrow
x_0^2 \frac{\partial}{\partial x_2}
- c_{022} x_0x_2 \frac{\partial}{\partial x_0}
\equiv 0\\
x_2^2 \frac{\partial}{\partial x_0}
- \sum_{j\ne 2} c_{2j0} x_2x_j \frac{\partial}{\partial x_2} \equiv 0 &\Rightarrow
x_2^2 \frac{\partial}{\partial x_0}
- c_{200} x_0x_2 \frac{\partial}{\partial x_2} \equiv 0.
\end{aligned}
\end{equation}
Suppose that $x_2 = a x_0$ in $H^0(\CO_Z(1))$ for some $a\ne 0$.
Then \eqref{COMVE313} becomes
\begin{equation}\label{COMVE307}
\begin{aligned}
- a c_{022} \left(x_0^2 \frac{\partial}{\partial x_0}\right) + x_0^2 \frac{\partial}{\partial x_2}
&\equiv 0\\
a^2 \left(x_0^2 \frac{\partial}{\partial x_0}\right)
- a c_{200} \left(x_0^2 \frac{\partial}{\partial x_2}\right) &\equiv 0.
\end{aligned}
\end{equation}
For $c_{022}c_{200}\ne 1$, \eqref{COMVE307} has only the trivial solution as a system of homogeneous linear equations in
$x_0^2 (\partial/\partial x_k)$ for $k=0,2$. That is,
\begin{equation}\label{COMVE308}
x_0^2 \frac{\partial}{\partial x_0}
\equiv x_0^2 \frac{\partial}{\partial x_2}\equiv 0
\end{equation}
which, combined with \eqref{COMVE306}, implies \eqref{COMVE037} for $k=0,2$. Similarly, we can prove \eqref{COMVE037} for $k=1,3$. This finishes the proof of the lemma.
\end{proof}
\appendix

\section{Simple facts on $\Gamma$-equivalence and multi-tangent lines to hypersurfaces}\label{COMVSECSF}

In this appendix, we prove a few simple facts. These should be more or less well known to the experts. We provide the proofs for readers' convenience.

First, we claim that $\Gamma$-equivalence is indeed an equivalence on $0$-cycles. It suffices to prove that it is \red{symmetric}: Fixing a smooth projective curve $\Gamma$ and two points $p\ne q$ on $\Gamma$, two $0$-cycles $\xi_1$ and $\xi_2$ on a projective variety $X$ are equivalent under $(\Gamma, p, q)$ if and only if they are equivalent under $(\Gamma, q, p)$.

As another interpretation of $\Gamma$-equivalence, $\xi_1$ and $\xi_2$ are equivalent under $(\Gamma, p, q)$ if and only if there exists a morphism $f: \Gamma \to S^N X$ for $N$ sufficiently large such that
$f(p) = \xi_1 + \eta$ and $f(q) = \xi_2 + \eta$ for some effective zero cycle $\eta$, where $S^N X$ is the $N$-th symmetric product of $X$.

Let us first show that there exists a morphism $\phi: \Gamma\to S^d \Gamma$ for $d$ sufficiently large such that
$\phi(p) = q+D$ and $\phi(q) = p+D$ for some effective divisor $D$ on $\Gamma$. There exists a natural map $\pi: S^d \Gamma \to \Pic^d(\Gamma) \cong J(\Gamma)$ sending $(p_1,p_2,...,p_d)$ to
$p_1 + p_2 + ... + p_d$. For $d \ge 2g - 1$, this is an $\PP^{d-g}$-bundle whose fibers can be identified with the complete linear series $|L|$ for $L\in \Pic^d(\Gamma)$, where $g = g(\Gamma)$ is the genus of $\Gamma$.

Fixing a divisor $F\in \Pic^{d+1}(\Gamma)$ for $d\ge 2g$, we can embed $\Gamma$ to $\Pic^d(\Gamma)$ by $\lambda: \Gamma \to \Pic^d(\Gamma)$ sending $\lambda(s) = F - s$ for all $s\in \Gamma$.
Let $Z = \pi^{-1}(\lambda(\Gamma))$ be the fiber product of $\pi$ and $\lambda$. Then $Z$ is a $\PP^{d-g}$-bundle over $\Gamma$, whose fiber over $s\in \Gamma$ is the linear series $|F - s|$. Since
$\deg (F-p-q) \ge 2g -1$, $F-p-q$ is effective and we let $D\in |F-p-q|$. Then $(p, q+D)$ and $(q,p+D)$ are two points on two fibers
of $Z$ over $\Gamma$. So there exists a section $\phi: \Gamma\to Z\subset S^d \Gamma$ passing through these two points. That is,
$\phi(p) = q+D$ and $\phi(q) = p+D$.

Now we combine $f: \Gamma \to S^N X$ and $\phi: \Gamma\to S^d \Gamma$ to obtain a morphism $h: \Gamma\to S^d(S^N X) \to S^{dN} X$ sending
$$
h(p) = f(q+D) = \xi_2 + f(D) \text{ and } h(q) = \xi_1 + f(D).
$$
It follows that $\xi_1$ and $\xi_2$ are also equivalent under $(\Gamma,q,p)$.

Next, let us explain how the sequence \eqref{COMVE127} governs the deformation of a line with tangency conditions with a fixed hypersurface.

More generally, let us consider the deformation of a nonconstant morphism $f: C\to P$ from a smooth \red{projective} curve $C$ to a smooth projective variety $P$. Suppose that $f$ deforms over a smooth variety $B$. That is,
there exists a smooth family \red{of curves} $W$ over $B$ and a morphism $\phi: W\to P\times B$ preserving the base $B$ such that $\phi_0: W_0\to P$ is exactly $f: C\to P$ when $\phi$ is restricted to a point $0\in B$.

The Kodaira-Spencer map $T_{B,0}\to H^0(\N_f)$
associated to $\phi$ is given by
\begin{equation}\label{COMVE888}
\begin{tikzcd}[column sep=normal]
 & & \phi^*\pi_B^* T_B\ar{d}\ar{dr}\\
 0 \ar{r} & T_W \ar{r} & \phi^* T_{P\times B} \ar{r} & \N_{\phi} \ar{r} & 0
\end{tikzcd}
\end{equation}
where $\N_\phi$ and $\N_f$ are normal bundles of the maps $\phi$ and $f$, respectively, and $\pi_B$ is the projection $P\times B\to B$.
Here we restrict the map $\phi^*\pi_B^* T_B\to \N_\phi$ to $W_0$
and take the global sections to obtain the map $T_{B,0}\to H^0(\N_f)$.

Now let us impose some tangency conditions on $f(C)$ with a fixed hypersurface $X\subset P$. For simplicity, let us assume that $X$ is a divisor of simple normal crossings. Suppose that
$$
f^* X = m_1 p_1 + m_2 p_2 + ... + m_r p_r
$$
for some distinct points $p_1,p_2,...,p_r\in C$. Let us assume that $\phi$ preserves the multiplicities of $p_i$ but not necessarily the points $f(p_i)$ themselves. That is,
$$
\phi^*(X\times B) = m_1 D_1 + m_2 D_2 + ... + m_r D_r
$$
where $D_i$ are disjoint sections of $W/B$ with $D_i\cap W_0 = p_i$.

Then the map $\phi^*\pi_B^* T_B\to \N_\phi$ in \eqref{COMVE888} factors through $\N_{\phi, X}$ via the diagram
\begin{equation}\label{COMVE788}
\begin{tikzcd}[column sep=small]
& & \phi^* \pi_B^* T_B\ar{d}\ar{dr}\\
0 \ar{r} & T_W(-\log D) \ar{r}\ar{d} & \phi^* T_{P\times B}(-\log (X\times B)) \ar{r} \ar{d} & \N_{\phi,X} \ar{r} \ar{d} & 0\\
0 \ar{r} & T_W \ar{r} & \phi^* T_{P\times B} \ar{r} & \N_\phi \ar{r} & 0
\end{tikzcd}
\end{equation}
where $D = \sum D_i$. Therefore, the Kodaira-Spencer map
$T_{B,0}\to H^0(\N_f)$ factors through $H^0(\N_{f,X})$, where $\N_{f,X}$ is given by the exact sequence
\begin{equation}\label{COMVE688}
\begin{tikzcd}
0 \ar{r} & T_C(-\sum p_i) \ar{r} & f^* T_P(-\log X) \ar{r} & \N_{f,X} \ar{r} & 0
\end{tikzcd}
\end{equation}
So the versal deformation space of $f: C\to P$ preserving the tangency conditions at $f^{-1}(X)$ with a fixed hypersurface $X$ has dimension no more than
\begin{equation}\label{COMVE588}
\begin{aligned}
h^0(\N_{f,X}) &\le h^0(f^* T_P(-\log X)) - \chi(T_C(-\sum p_i))
\\
&= h^0(f^* T_P(-\log X)) + (3g(C) - 3 + r).
\end{aligned}
\end{equation}

\red{Note the above inequality holds for \emph{any} $g=g(C)$}. Now let us apply \eqref{COMVE588} to the deformation of lines $\Lambda \hookrightarrow
P$ in $P = \PP^{n+1}$ preserving tangency conditions with a fixed hypersurface $X$ of degree $d$. Our purpose is to construct pairs $(X, \Lambda)$ such that the deformation $\Lambda$ has the expected dimension $0$ given by \eqref{COMVE588}. That is, these multi-tangent lines to $X$ are rigid.

Obviously, we need to compute $H^0(\Lambda, T_P(-\log X))$. This can be done via the restriction of the exact sequence
$$
\begin{tikzcd}
0 \ar{r} & T_P(-\log X) \ar{r} & T_P \ar{r} & \N_{X/P}\ar[equal]{d} \ar{r} & 0\\
&&& \CO_X(d)
\end{tikzcd}
$$
to $\Lambda$. Since $\Lambda\not\subset X$, the above exact sequence remains exact when restricted to $\Lambda$. So $H^0(\Lambda, T_P(-\log X))$ is the kernel of the map
\begin{equation}\label{COMVE988}
\begin{tikzcd}
H^0(\Lambda, T_P) \ar{r}{\eta} & H^0(\Lambda, \CO_X(d))
\end{tikzcd}
\end{equation}
\red{Note the right hand side of the map is equivalent to $H^0(\Gamma, \CO_{\Gamma}(d))/ \langle F|_{\Gamma}\rangle$.} This map has been used throughout the paper. Once again, it is given by
$$
\eta\left(
\sum L_i \frac{\partial}{\partial z_i}
\right) = \sum L_i \frac{\partial F}{\partial z_i}
$$
with $L_i\in H^0(\CO_P(1))$ and $\partial F/\partial z_i$ restricted to $\Lambda$,
where $(z_0,z_1,...,z_{n+1})$ are the homogeneous coordinates of $P = \PP^{n+1}$ and $F$ is the defining equation of $X$.

In our first example, we will construct a smooth hypersurface $X$ of degree $d=2n+1$ and two lines $\Lambda_i$ for $i=1,2$ such that each $\Lambda_i$ meets $X$ at a unique point $p_i$ with $p_1\ne p_2$ and has rigid deformation preserving the tangency. We let
$$
\begin{aligned}
\Lambda_1 &= \{z_2=z_3 =z_4=...=z_{n+1} = 0\},\ p_1 = (1,-1,0,...,0)\\
\Lambda_2 &= \{z_1=z_3 = z_4 =...=z_{n+1} = 0\},\ p_2 = (1,0,-1,0,...,0)\\
F &= (z_0+z_1+z_2)^d + z_1z_2 G_0(z_0,z_1,z_2) + \sum_{j=3}^{n+1} z_j G_j(z_0,z_1,z_2)\\
&\quad + G(z_3,z_4,...,z_{n+1})
\end{aligned}
$$
where $G_0(z_0,z_1,z_2), G_j(z_0,z_1,z_2)$ and $G(z_3,z_4,...,z_{n+1})$ are homogeneous polynomials in
$(z_0,z_1,z_2)$ and $(z_3,z_4,...,z_{n+1})$ of degree $d-2, d-1$ and $d$, respectively. For a general choice of $(G_0, G_j,G)$, $X = \{ F = 0\}$ is smooth by Bertini.

For each $\Lambda = \Lambda_i$ and a general choice of $(G_0, G_3,...,G_{n+1})$, it is easy to check that
the space
$$
\begin{aligned}
&\quad H^0(\CO_{\Lambda}(1)) \otimes \Span\left.\left\{
\frac{\partial F}{\partial z_0}, \frac{\partial F}{\partial z_1}, ..., \frac{\partial F}{\partial z_{n+1}}
\right\}\right|_\Lambda\\
&= H^0(\CO_{\Lambda}(1))
\otimes \Span\big\{
(z_0+z_1+z_2)^{d-1}, z_i G_0(z_0,z_1,z_2),\\
&\hspace{114pt} G_3(z_0,z_1,z_2), ..., G_{n+1}(z_0,z_1,z_2)
\big\}\Big|_\Lambda
\end{aligned}
$$
surjects onto $H^0(\CO_{\Lambda}(d))$. Therefore, the map $\eta$ in \eqref{COMVE988} is surjective for each $\Lambda = \Lambda_i$. Then it is easy to compute
$$
h^0(f^* T_P(-\log X)) = h^0(\Lambda, T_P) - h^0(\Lambda, \CO_X(d)) = (2n+3) - d = 2
$$
and hence each $\Lambda = \Lambda_i$ is rigid by \eqref{COMVE588}. This construction shows that there exist a smooth hypersurface $X\subset \PP^{n+1}$ of degree $2n+1$ and two lines $\Lambda_1$ and $\Lambda_2$, each tangent to $X$ at a point $p_i$ with multiplicity $2n+1$ such that $p_1 \ne p_2$ and neither $\Lambda_i$ can deform when preserving the tangency condition with $X$. Combining this with an incidence correspondence argument as in the proof of Proposition \ref{COMVPROPLINE}, we can prove that for a general hypersurface $X\subset \PP^{n+1}$ of degree $2n+1$,
there are finitely many lines in $\PP^{n+1}$ meeting $X$ at a unique point and there exist (at least) two lines
$\Lambda_i$ for $i=1,2$, each meeting $X$ at a unique point $p_i$ with $p_1\ne p_2$. This justifies our claim that
the bound $2n+2$ in Theorem \ref{COMVTHMHYPERSURFACE} is optimal.

In our second example, we will construct $(X, \Lambda_i, p_{i1}, p_{i2})$ as claimed in the proof of Proposition \ref{COMVPROPLINE}, where
\begin{itemize}
	\item $X$ is a smooth hypersurface in $P = \PP^{n+1}$ of degree $d=2n+2$,
	\item $\Lambda_1$ and $\Lambda_2$ are two lines, each meeting $X$ at two points $p_{i1}$ and $p_{i2}$ with assigned multiplicities $m$ and $d-m$,
	\item neither $\Lambda_i$ deforms when preserving the tangency conditions with $X$,
	\item $p_{11}\ne p_{21}$ and the line $\overline{p_{11} p_{21}}$ meets $X$ at more than two points.
\end{itemize}
As in the first example, we let
$$
\begin{aligned}
\Lambda_1 &= \{z_2=z_3 =z_4=...=z_{n+1} = 0\},\\
\Lambda_2 &= \{z_1=z_3 = z_4 =...=z_{n+1} = 0\},\\
p_{11} &= (1,-1,0,...,0),\ p_{12} = (1,1,0,...,0)\\
p_{21} &= (1,0,-1,0,...,0),\ p_{22} = (1,0,1,0,...,0)\\
F &= (z_0+z_1+z_2)^m (z_0-z_1 - z_2)^{d-m} + z_1z_2 G_0(z_0,z_1,z_2)\\
&\quad  + \sum_{j=3}^{n+1} z_j G_j(z_0,z_1,z_2) + G(z_3,z_4,...,z_{n+1})
\end{aligned}
$$
where $G_0(z_0,z_1,z_2), G_j(z_0,z_1,z_2)$ and $G(z_3,z_4,...,z_{n+1})$ are homogeneous polynomials in
$(z_0,z_1,z_2)$ and $(z_3,z_4,...,z_{n+1})$ of degree $d-2, d-1$ and $d$, respectively. For a general choice of $(G_0, G_j,G)$, $X = \{ F = 0\}$ is smooth by Bertini.

For $G_0$ general, the line $\overline{p_{11} p_{21}}$ clearly meets $X$ at more than two points.

As in the first example, it is easy to check that
$$
H^0(\CO_{\Lambda}(1)) \otimes \Span\left.\left\{
\frac{\partial F}{\partial z_0}, \frac{\partial F}{\partial z_1}, ..., \frac{\partial F}{\partial z_{n+1}}
\right\}\right|_\Lambda
$$
surjects onto $H^0(\CO_{\Lambda}(d))$ for
each $\Lambda = \Lambda_i$ and a general choice of $(G_0, G_j)$. Thus,
$$
h^0(f^* T_P(-\log X)) = h^0(\Lambda, T_P) - h^0(\Lambda, \CO_X(d)) = (2n+3) - d = 1
$$
and hence each $\Lambda = \Lambda_i$ is rigid by \eqref{COMVE588}. This proves the claim we made in the proof of Proposition \ref{COMVPROPLINE}.

\section{Notes on algebraic invariants}\label{COMVSECNOAI}

\magenta{Here we give another construction of the invariants defined in \secref{COMVSECRCM} based on Hodge theory.
These algebraic invariants from Hodge theory, some of which are used in \cite{V2}, are the same thing as de Rham invariants, the latter
not involving Hodge theory.
First let us fix some notations.}
For a $\BQ$-MHS $V$, we put $\Gamma(V) := \hom_{\rm MHS}(\BQ(0),V)$
and accordingly $J(V) := \Ext^1_{\rm MHS}(\Q(0),V)$.

To arrive at the invariants of interest,
we must introduce a natural filtration on the Chow groups of $X$.
Let $\rho : \X \to S$ be a smooth and proper morphism of smooth quasi-projective
varieties over a subfield \red{$k$ of $\BC$ finitely generated over $\overline{\BQ}$,} and let
$K= k(S)$. Fix an embedding $K \hookrightarrow \BC$ over $k$, and put $X := X/\BC =
\X_{\eta_S}\times_K\BC$.

\begin{thm}[\cite{Lew1}]

Let $ X := X/\BC$ be smooth projective
of dimension $d$.
Then for all $r\geq 0$, there is a filtration, depending on $k\subset \BC$,
\[
\CH^r(X;\Q) = F^0 \supseteq F^1\supseteq \cdots \supseteq F^{\nu}
\supseteq F^{\nu +1}\supseteq
\]
\[
 \cdots \supseteq F^r\supseteq F^{r+1}
= F^{r+2}=\cdots,
\]
which satisfies the following

\medskip
\noindent
{\rm (i)} $F^1 = \CH^r_{\hom}(X;\Q)$.

\medskip
\noindent
{\rm (ii)} $F^2 \subseteq \ker AJ\otimes\Q : \CH^r_{\hom}(X;\Q)
\rightarrow J\big(H^{2r-1}(X(\BC),\Q(r))\big)$.

\medskip
\noindent
{\rm (iii)}
$F^{\nu_1}\CH^{r_1}(X;\Q)\bullet F^{\nu_2}\CH^{r_2}(X;\Q)\subset
F^{\nu_1 +\nu_2}\CH^{r_1+r_2}(X;\Q)$, where $\bullet$ is
the intersection product.

\medskip
\noindent
{\rm (iv)} $F^{\nu}$ is  preserved under the action of correspondences
between smooth projective varieties over $\BC$.

\medskip
\noindent
{\rm (v)} Let ${\rm Gr}_F^{\nu} :=  F^{\nu}/F^{\nu +1}$ and
assume that the K\"unneth components of the diagonal
class $[\Delta_X] = \oplus_{p+q = 2d}[\Delta_X(p,q)] \in
H^{2d}(X\times X,\Q(d)))$ are algebraic over $\Q$.  Then
\[
\Delta_X(2d-2r+\ell,2r-\ell)_{\ast}\big\vert_{{\rm Gr}_F^{\nu}\CH^r(X,m;\Q)}
= \delta_{\ell,\nu}\cdot \text{\rm Identity}.
\]
[If we assume the conjecture that homological and numerical equivalence coincide,
then (v) says that ${\rm Gr}_F^{\nu}$ factors through the Grothendieck motive.]

\medskip
\noindent
{\rm (vi)} Let $D^r(X) := \bigcap_{\nu}F^{\nu}$, and $k = \ol{\Q}$.
 If the Bloch-Beilinson
conjecture on the injectivity of the Abel-Jacobi map ($\otimes\Q$)
holds for smooth quasi-projective varieties defined over $\ol{\Q}$, then $D^r(X) = 0$.
\end{thm}

It is instructive to briefly explain how this filtration comes about.
 Consider a $k$-spread $\rho : \X \to {S}$, where
$\rho$ is smooth and proper.  Let $\eta$
be the generic point of ${S}/k$, and put $K := k(\eta)$. Write
$X_K := \X_\eta$. From  \cite{Lew1} we introduced a
decreasing filtration $\F^\nu\CH^r(\X;\Q)$, with the
property that   $Gr_\F^\nu\CH^r(\X;\Q) \hookrightarrow
E_{\infty}^{\nu,2r-\nu}(\rho)$, where $E_{\infty}^{\nu,2r-\nu}(\rho)$
is the $\nu$-th graded piece of the Leray filtration on
the lowest weight part $\ul{H}^{2r}_{\mathcal H}(\X,\Q(r))$
of Beilinson's absolute Hodge
cohomology $H^{2r}_{\mathcal H}(\X,\Q(r))$ associated to $\rho$.
That lowest weight part $\ul{H}^{2r}_{\mathcal H}(\X,\Q(r))
\subset {H}^{2r}_{\mathcal H}(\X,\Q(r))$ is given by the image
$H^{2r}_{\mathcal H}(\ol{\X},\Q(r)) \to H^{2r}_{\mathcal H}(\X,\Q(r))$,
where $\ol{\X}$ is a smooth compactification of $\X$. There
is a cycle class map $\CH^{r}(\X;\Q)
:= \CH^{r}(\X/k;\Q) \to
\ul{H}^{2r}_{\mathcal H}(\X,\Q(r))$, which is conjecturally
injective if $k = \ol{\Q}$ under the Bloch-Beilinson conjecture assumption,
using the fact that there is a short exact sequence:
\[
0 \to J(H^{2r-1}(\X,\Q(r)))\to  {H}^{2r}_{\mathcal H}(\X,\Q(r)) \to
\Gamma(H^{2r}(\X,\Q(r))) \to 0.
\]
(Injectivity would imply $D^{r}(X) = 0$.)
Regardless of whether or not injectivity holds, the
filtration $\F^\nu\CH^r(\X;\Q)$ is given by the
pullback of the Leray filtration on $\ul{H}^{2r}_{\mathcal H}(\X,\Q(r))$ to
$\CH^{r}(\X;\Q)$.
It is proven in \cite{Lew1} that the term $E_{\infty}^{\nu,2r-\nu}(\rho)$
fits in a short exact sequence:
$$
0\to \ul{E}_{\infty}^{\nu,2r-\nu}(\rho) \to E_{\infty}^{\nu,2r-\nu}(\rho)
\to \ul{\ul{E}}_{\infty}^{\nu,2r-\nu}(\rho) \to 0,
$$
where
\[
\ul{\ul{E}}_{\infty}^{\nu,2r-\nu}(\rho) =
\Gamma(H^\nu({S},R^{2r-\nu}\rho_{\ast}\Q(r))),
\]
\[
\ul{E}_{\infty}^{\nu,2r-\nu}(\rho) = \frac{J(
W_{-1}H^{\nu-1}({S},R^{2r-\nu}\rho_{\ast}\Q(r)))}{\Gamma(Gr_W^0H^{\nu-1}
({S},R^{2r-\nu}\rho_{\ast}\Q(r)))}
\]
\[
\subset J(H^{\nu-1}({S},R^{2r-\nu}\rho_{\ast}\Q(r))).
\]
[Here the latter inclusion is a result of the short exact sequence:
\[
0\to W_{-1}H^{\nu-1}({S},R^{2r-\nu}\rho_{\ast}\Q(r)) \to
W_0H^{\nu-1}({S},R^{2r-\nu}\rho_{\ast}\Q(r))
\]
\[
\to Gr_W^0H^{\nu-1}({S},R^{2r-\nu}\rho_{\ast}\Q(r)) \to 0.]
\]

One then has (by definition)
\[
F^{\nu}\CH^r(X_K;\Q) =\lim_{\buildrel \to\over
{U\subset S/k}}\F^\nu\CH^r(\X_U;\Q), \quad \X_{U} := \rho^{-1}(U)
\]
\[
F^{\nu}\CH^r(X_{\BC};\Q) = \lim_{\buildrel \to
\over {K\subset \BC}}F^{\nu}\CH^r(X_K;\Q)
\]
Further, since direct limits preserve exactness,
\[
Gr_F^{\nu}\CH^r(X_K;\Q) =
\lim_{\buildrel \to\over {U\subset S/k}}Gr_{\F}^\nu\CH^r(\X_{U};\Q),
\]
\[
Gr_F^{\nu}\CH^r(X_{\BC};\Q) = \lim_{\buildrel \to\over
{K\subset \BC}}Gr_F^{\nu}\CH^r(X_K;\Q)
\]

\subsection{(Generalized) normal functions}

Let us now assume that with regard to the smooth and proper map
$\rho : \X \to {S}$ over a subfield  $k \subset \BC$\magenta{, finitely generated over $\overline{\BQ}$,} and after
possibly shrinking ${S}$, that ${S}$ is affine, with $K = k({S})$.
Let $V\subset \red{{S}_{\BC}}$ be smooth, irreducible, closed
subvariety of dimension
$\nu-1$ (note that ${S}$ affine $\Rightarrow V$ affine). One has a
commutative square
\[
\begin{matrix}
\X_V&\hookrightarrow&\red{\X_{\BC}}\\
&\\
\rho_V\downarrow\quad&&\quad\downarrow\rho\\
&\\
V&\hookrightarrow&{S}_{\BC},
\end{matrix}
\]
and a commutative diagram
\[
\begin{matrix}
&&&\xi\in&Gr_{\F}^\nu\CH^r(\X;\Q)&\mapsto&Gr_{F}^{\nu}\CH^r(X_K;\Q)\\
&\\
&&&&\downarrow\\
&\\
0&\to&\ul{E}_{\infty}^{\nu,2r-\nu}(\rho)&\to&E_{\infty}^{\nu,2r-\nu}(\rho)
&\to&\ul{\ul{E}}_{\infty}^{\nu,2r-\nu}(\rho)&\to&0\\
&\\
&&\downarrow&&\downarrow&&\downarrow\\
&\\
0&\to&\ul{E}_{\infty}^{\nu,2r-\nu}(\rho_V)&\to&
E_{\infty}^{\nu,2r-\nu}(\rho_V)
&\to&\ul{\ul{E}}_{\infty}^{\nu,2r-\nu}(\rho_V)&\to&0\\
&\\
&&&&&&||\\
&\\
&&&&&&0
\end{matrix}
\]
where $\ul{\ul{E}}_{\infty}^{\nu,2r-\nu}(\rho_V) = 0$ follows from the weak
Lefschetz theorem for locally constant systems
over affine varieties. Thus for any $\xi\in Gr_{\F}^\nu\CH^r(\X;\Q)$,
we have a ``normal
function'' $\eta_\xi$ with the property that for any such smooth irreducible
closed $V\subset S_{\BC}$ of dimension $\nu-1$, we have a value $\eta_\xi(V) \in
\ul{E}_{\infty}^{\nu,2r-\nu}(\rho_V)$. Here we think of $V$
as a point on a suitable open subset of the Chow variety
of dimension $\nu-1$ subvarieties of
${S}_{\BC}$ and $\eta_\xi$ defined on that subset. For example if $\nu =1$,
then we recover the classical notion of normal functions.

\begin{defn} {\rm $\eta_{\xi}$ is called an arithmetic normal function.}
\end{defn}

\begin{example}{\rm  \red{Suppose} ${S}$ is affine of dimension $\nu -1$. Then
in this case $V = {S}$, and $\xi\in Gr_{\F}^{\nu}\CH^{r}(\X;\Q)$
induces a ``single point'' normal function
\[
\eta_{\xi}(V) = \eta_{\xi}({S})
\in J(H^{\nu-1}({S},R^{2r-\nu}\rho_{\ast}\Q(r))).
\]}
\end{example}

Now let $\xi \in \F^{\nu}\CH^r(\X;\Q)$ be given, and let
$[\xi]\in \ul{\ul{E}}_{\infty}^{\nu,2r-\nu}(\rho)$ be its image via
the composite
\[
\F^{\nu}\CH^r(\X;\Q) \to {E}_{\infty}^{\nu,2r-\nu}(\rho) \to \ul{\ul{E}}_{\infty}^{\nu,2r-\nu}(\rho).
\]

\subsection{The invariants}
\begin{thm}[see \cite{K-L}] The class $[\xi]$ depends only on $\eta_{\xi}$, and is called the
topological invariant of $\eta_{\xi}$.
\end{thm}

Let us assume that $S$ is affine. Then
\[
\CO_S\otimes_{\BC} R^{i} \rho_*\BC\xrightarrow{\nabla} \Omega^1_S\otimes
R^{i} \rho_*\BC \xrightarrow{\nabla} \cdots,
\]
is an acyclic resolution of $R^{2r-\nu} \rho_*\BC$ in the analytic topology, where $\nabla := \partial\otimes Id$ is
the Gauss-Manin connection. The corresponding
cohomology $H^{\nu}(S,R^{2r-\nu}\rho_*\BC)$
is given by $H^0(S,-)$ of the middle cohomology in:
\[
\Omega^{\nu-1}_S\otimes
R^{2r-\nu} \rho_*\BC \xrightarrow{\nabla} \Omega^{\nu}_S\otimes
R^{2r-\nu} \rho_*\BC \xrightarrow{\nabla} \Omega^{\nu+1}_S\otimes
R^{2r-\nu} \rho_*\BC,
\]
which is by definition the space of de Rham invariants, and is denoted by $\nabla DR^{r,\nu}(\X/S)$.
As the map $\ul{\ul E}_{\infty}^{\nu,2r-\nu}(\rho) \hookrightarrow \nabla DR^{r,\nu}(\X/S)$, together with the
regularity of $\nabla$, it follows that the de Rham invariant of an algebraic cycle is the same as the topological invariant. It turns out that
$H^i(S,R^j\rho_*\Q(r))$ defines a $\Q$-MHS \cite{Ar}, hence its complexification carries a descending Hodge
filtration $F^{\bullet}H^i(S,R^j\rho_*\BC)$. In particular,
\[
\ul{\ul E}_{\infty}^{\nu,2r-\nu}(\rho) \hookrightarrow  F^rH^{\nu}(S,R^{2r-\nu}\rho_*\BC),
\]
where the latter term maps to $H^0(S,-)$ of the middle cohomology in:
\begin{equation}\label{E1}
\Omega^{\nu-1}_S\otimes F^{r-\nu+1}
R^{2r-\nu} \rho_*\BC \xrightarrow{\nabla} \Omega^{\nu}_S\otimes
F^{r-\nu}R^{2r-\nu} \rho_*\BC
\end{equation}
\[ \xrightarrow{\nabla} \Omega^{\nu+1}_S\otimes
F^{r-\nu-1}R^{2r-\nu} \rho_*\BC,
\]
which is called the space of Mumford-Griffiths invariants, and is denoted by
$\nabla J^{r,\nu}(\X/S)$. Note that there is a natural ``forgetful'' map
$\nabla J^{r,\nu}(\X/S) \to \nabla DR^{r,\nu}(\X/S)$, which need not be injective.
Having said this, it is clear from the above discussion that
\[
\IM\big(\ul{\ul E}_{\infty}^{\nu,2r-\nu}(\rho) \to \nabla J^{r,\nu}(\X/S)\big)\to
\IM\big(\ul{\ul E}_{\infty}^{\nu,2r-\nu}(\rho) \to \nabla DR^{r,\nu}(\X/S)\big),
\]
is an isomorphism. Thus when it comes to the image of algebraic cycles, the
de Rham and Mumford-Griffiths  invariants coincide! (All of this is based on \cite{L-S} and \cite{MS}.) Those cycles that have trivial
Mumford-Griffiths invariant must therefore land in $\ul{E}_{\infty}^{\nu,2r-\nu}(\rho)$. In some
instances, this can be an uncountable space. Note that
\[
\Omega^{\nu-1}_S\otimes F^{r-\nu+2}
R^{2r-\nu} \rho_*\BC \xrightarrow{\nabla} \Omega^{\nu}_S\otimes
F^{r-\nu+1}R^{2r-\nu} \rho_*\BC
\]
\[
\xrightarrow{\nabla} \Omega^{\nu+1}_S\otimes
F^{r-\nu}R^{2r-\nu} \rho_*\BC,
\]
is a subcomplex of (\ref{E1}). The Mumford invariants are $H^0(S,-)$ of the middle cohomology of the cokernel complex:
\[
\Omega^{\nu-1}_S\otimes \h^{r-\nu+1,r-1}(\X/S)
 \xrightarrow{\tilde \nabla} \Omega^{\nu}_S\otimes
\h^{r-\nu,r}(\X/S)
\]
\[
\xrightarrow{\tilde\nabla} \Omega^{\nu+1}_S\otimes
\h^{r-\nu-1,r+1}(\X/S),
\]
and where $\tilde\nabla$ is induced from $\nabla$.

\begin{example}\label{EX1}
 Let us put $N := \dim S$ and $n$ the relative dimension of $\rho$, with $r=n$. In this case
 we are studying the relative $0$-cycles on each fiber of $\rho$.  This involves $\F^n\CH^n(\X;\Q)$,
where we set $\nu = n$.  Then
 \[
H^0\biggl(S, \frac{\Omega^n_S\otimes_{\CO_S}\h^{0,n}(\X/S)}{{\tilde\nabla}\big(\Omega_S^{n-1}\otimes_{\CO_S}\h^{1,n-1}(\X/S)
\big)}\biggr)
\]
is the associated space of  Mumford invariants.  If $n=2$, it also appears in \cite{V2}.
Note that in this case, we need $\xi \in \CH^2(\X;\Q)$ to be Abel-Jacobi equivalent to zero fiberwise, in order that $\xi \in \F^2\CH^2(\X;\Q)$.
\end{example}

\begin{question}\label{Q1} {\rm   (i) Can one characterize this  filtration
in terms of arithmetic normal functions?

\medskip

(ii) By choosing $V$ sufficiently general, can one characterize this  filtration
in terms of the corresponding  Abel-Jacobi map for a fixed general variety?
E.g. we know that $F^1\CH^r(X;\Q) = \CH^r_{\hom}(X;\Q)$ and
$$
F^2\CH^r(X;\Q) \subseteq \CH^r_{AJ}(X;\Q) :=  \ker AJ_X:\CH^r_{\hom}(X;\Q) \to J^r(X)_{\Q}.
$$
Is it the case that $F^2\CH^r(X;\Q) = \CH^r_{AJ}(X;\Q)$?

\medskip

(ii)$^{'}$ What about the zero (or torsion) locus of such normal functions. I.e., are they
sensitive to the field of definition of algebraic cycles?}
\end{question}

\begin{remark} {\rm $\bullet_1$ Special cases of Question \ref{Q1}(i)  are worked out in \cite{K-L}.
Further, if both $X$ and $S$  are defined over $k$, with $\X = S\times_k X$, with $\rho = {\rm Pr}_1$,
then the answer is yes, as shown in \cite{Lew2}.

\medskip
$\bullet_2$ In the case where $\nu = 1$, (ii) and (ii)$^{'}$  can be shown to be equivalent. (See for
example \cite{Lew3}.)}
\end{remark}

\subsection{Example \ref{EX1} revisited} Let us put $N := \dim S$ and $n$ the relative dimension of $\rho$.

\begin{question} Does there \red{exist} a morphism of sheaves
\[
\frac{\Omega^n_S\otimes_{\CO_S}\h^{0,n}}{\tilde{\nabla}\big(\Omega_S^{n-1}\otimes_{\CO_S}\h^{1,n-1}
\big)} \to {\h}om_{\CO_S}\big(\rho_*(\wedge^N\Omega_{\X}),\omega_S\big),
\]
induced by
\[
(a\otimes b,\rho_*(c))\in \big(\Omega_S^n\otimes\h^{0,n},\rho_*(\wedge^N\Omega_{\X})\big) \mapsto
a\wedge\rho_*(\rho^*(b),c) = a\wedge \int_{\X_t}  \rho^*(b)\wedge c\in \omega_{S,t},
\]
where $\omega_S$ is the canonical sheaf on $S$?
\end{question}

\begin{remark}
The answer is yes if $\X = S \times_k X$, for in this case
\[
\tilde{\nabla}\big(\Omega_S^{n-1}\otimes_{\CO_S}\h^{1,n-1}
\big) = 0.
\]
\end{remark}


\begin{thebibliography}{Mum68}

\bibitem[Ara05]{Ar}
Donu Arapura.
\newblock The Leray spectral sequence is motivic.
\newblock {\em Invent. Math.}, 160(3):567--579, 2005.

\bibitem[Cle86]{C}
Herbert Clemens.
\newblock Curves on generic hypersurfaces.
\newblock {\em Ann. Sci. \'Ecole Norm. Sup.}, 19(4):629--636, 1986.

\bibitem[CZ18]{CZ-logVoisin}
Xi~Chen and Yi~Zhu.
\newblock $\mathbb{A}^1$-curves on affine complete intersections.
\newblock {\em Sel. Math. New Ser.}, 24:3823--3834, 2018.
\newblock Also preprint arXiv:1610.03791.

\bibitem[Ein88]{E1}
Lawrence Ein.
\newblock Subvarieties of generic complete intersections.
\newblock {\em Invent. Math.}, 94(1):163--169, 1988.

\bibitem[Ein91]{E2}
Lawrence Ein.
\newblock Subvarieties of generic complete intersections. {II}.
\newblock {\em Math. Ann.}, 289(3):465--471, 1991.

\bibitem[Fak02]{F}
Najmuddin Fakhruddin.
\newblock Zero cycles on generic hypersurfaces of large degree.
\newblock {\em arXiv.org}, page arXiv:math/0208179, Aug 2002.

\bibitem[KL07]{K-L}
Matthew Kerr and James~D. Lewis.
\newblock The Abel-Jacobi map for higher Chow groups, {II}.
\newblock {\em Invent. Math.}, 170(2):355--420, 2007.

\bibitem[Kol98]{K}
J\'anos Koll\'ar.
\newblock {\em Birational Geometry of Algebraic Varieties}, volume 134 of {\em
  Cambridge Tracts in Mathematics}.
\newblock Cambridge University Press, 1998.

\bibitem[Lew01]{Lew1}
James~D. Lewis.
\newblock A filtration on the Chow groups of a complex projective variety.
\newblock {\em Compos. Math.}, 128(3):299--322, 2001.

\bibitem[Lew10]{Lew3}
James~D. Lewis.
\newblock Abel-Jacobi equivalence and a variant of the Beilinson-Hodge
  conjecture.
\newblock {\em J. Math. Sci. Univ. Tokyo}, 17:179--199, 2010.

\bibitem[Lew12]{Lew2}
James~D. Lewis.
\newblock Arithmetic normal functions and filtrations on Chow groups.
\newblock {\em Proc. Am. Math. Soc.}, 140(8):2663--2670, 2012.

\bibitem[LS07]{L-S}
James~D. Lewis and Shuji Saito.
\newblock Algebraic cycles and Mumford-Griffiths invariants.
\newblock {\em Amer. J. Math.}, 129(6):1449 --1499, 2007.

\bibitem[Mum68]{Mu}
David Mumford.
\newblock Rational equivalence of $0$-cycles on surfaces.
\newblock {\em J. Math. Kyoto Univ.}, 9:195--204, 1968.

\bibitem[Ro{\u{\i}}71]{R1}
A.~A. Ro{\u{\i}}tman.
\newblock {$\Gamma$}-equivalence of zero-dimensional cycles.
\newblock {\em Math. Sb. (N.S.)}, 86(128):557--570 (Russian), 1971.
\newblock English transl., Math. USSR-Sb. {\bf 15} (1971), 555--567.

\bibitem[Ro{\u{\i}}72]{R2}
A.~A. Ro{\u{\i}}tman.
\newblock Rational equivalence of zero-dimensional cycles.
\newblock {\em Math. Sb. (N.S.)}, 89(131):569--585 (Russian), 1972.
\newblock English transl., Math. USSR-Sb. {\bf 18} (1972), 571--588.

\bibitem[RY18]{RiedlYangAGTH}
Eric Riedl and David Yang.
\newblock Applications of a Grassmannian technique in hypersurfaces.
\newblock {\em arXiv.org}, page arXiv:1806.02364, Jun 2018.

\bibitem[Sai07]{MS}
Morihiko Saito.
\newblock Direct image of logarithmic complexes and infinitesimal invariants of
  cycles.
\newblock In {\em Algebraic Cycles and Motives {\bf 2}}, volume 344 of {\em
  London Mathematical Society Lecture Note Series}, pages 30--318. Cambridge
  University Press, Cambridge, 2007.

\bibitem[Voi94]{V1}
Claire Voisin.
\newblock Variations de structure de Hodge et zero-cycles sur les surfaces
  generales.
\newblock {\em Math. Ann.}, 299:77--103, 1994.

\bibitem[Voi96]{V2}
Claire Voisin.
\newblock On a conjecture of Clemens on rational curves on hypersurfaces.
\newblock {\em J. Diff. Geom.}, 44:200--214, 1996.

\bibitem[Voi98]{Voisin98}
Claire Voisin.
\newblock A correction: ``{O}n a conjecture of {C}lemens on rational curves on
  hypersurfaces'' [{J}.\ {D}ifferential {G}eom.\ {\bf 44} (1996), no.\ 1,
  200--213; {MR}1420353 (97j:14047)].
\newblock {\em J. Differential Geom.}, 49(3):601--611, 1998.

\end{thebibliography}
\end{document}